\documentclass[12pt]{article}

\usepackage{amsfonts,amsmath,amsthm,amssymb,amscd}
\usepackage[dvips]{graphicx}

\theoremstyle{definition}
\newtheorem{theorem}{Theorem}
\newtheorem{lemma}[theorem]{Lemma}
\newtheorem{proposition}[theorem]{Proposition}
\newtheorem{corollary}[theorem]{Corollary}

\numberwithin{equation}{section}
\numberwithin{theorem}{section}

\begin{document}

\begin{center}
{\bf{\Large Rational Solutions of the Noumi and Yamada System of Type $A_5^{(1)}$}}
\end{center}

\begin{center}
Department of Engineering Science, Niihama National College of Technology,\\
7-1 Yagumo-chou, Niihama, Ehime, 792-8580. 
\end{center} 

\begin{center}
By Kazuhide Matsuda
\end{center}
\vskip 1mm
\par
\quad {\bf Abstract:} \,
In this paper,
we completely classify the rational solutions of the Noumi and Yamada System of type $A_5^{(1)}$,
which is a generalization of the fifth Painlev\'e equation. 
Noumi and Yamada system is a system of ordinary differential equations which has the 
af\/f\/ine Weyl group symmetry of type $A^{(1)}_l \,(l\geq 2).$ 
The Noumi and Yamada systems of types $A_2^{(1)}$ and $A_3^{(1)}$ are equivalent to the fourth and fifth 
Painlev\'e equations, respectively. 
The Noumi and Yamada system of type $A_4^{(1)}$ is a generalization of the fourth Painlev\'e equation. 

\section*{Introduction}
Paul Painlev\'e and his coworkers 
\cite{Painleve, Gambier} 
intended to find ``new transcendental functions'' defined by second order nonlinear differential equations. 
For this purpose, 
they investigated which second order ordinary differential equations of the form 
$$
y^{\prime\prime}
=F(t;y,y^{\prime}),
$$
where ${}^{\prime}=d/dt$ and $F$ is rational in $y$ and $y^{\prime}$ and analytic in $t,$ 
have the property that 
the solutions have no movable branch points, i.e., 
the locations of the multi-valued singularities are independent of the particular solution chosen 
and therefore dependent only on the equation; 
this is known as the Painlev\'e property. 
As a result, 
the differential equations are either 
integrable in terms of previously known functions 
(such as elliptic functions or are equivalent to linear differential equations) 
or reducible to  
one of the following six equations:
\begin{align*}
&P_{\rm I} & &: &  
&y^{\prime \prime}   
=
6
y^2+t, \\
&P_{\rm II}& &: &
&y^{\prime \prime}   
=
2 y^3 +ty+\alpha, \\
&P_{\rm III}& &: &
&y^{\prime \prime}   
=
\frac{1}{y}(y^{\prime})^2
-\frac{1}{t}y^{\prime}
+\frac{1}{t}(\alpha y^2+\beta)
+\gamma y^3 + \frac{\delta}{y}, \\
&P_{\rm IV}& &: &
&y^{\prime \prime}  
=
\frac{1}{2y}(y^{\prime})^2
+\frac32 y^3 +4ty^2
+2(t^2-\alpha)y+\frac{\beta}{y}, \\
&P_{\rm V}& &: &
&y^{\prime \prime}   
=
\left(
\frac{1}{2y}+\frac{1}{y-1}
\right)
(y^{\prime})^2
-\frac{1}{t}y^{\prime}
+\frac{(y-1)^2}{t^2} \left(\alpha y + \frac{\beta}{y} \right)
+\frac{\gamma}{t}y
+\delta \frac{y(y+1)}{y-1}, \\
&P_{\rm VI}& &: &
&y^{\prime \prime}   
= 
\frac12 
\left(
\frac{1}{y}+\frac{1}{y-1}+\frac{1}{y-t}
\right)
(y^{\prime})^2
-\left(\frac{1}{t}+\frac{1}{t-1}+\frac{1}{y-t}\right)y^{\prime} \\
& & & &
&\hspace{50mm}   
+\frac{y(y-1)(y-t)}{t^2(t-1)^2}
\left(
\alpha+\beta \frac{t}{y^2}+\gamma \frac{t-1}{(y-1)^2}
+\delta \frac{t(t-1)}{(y-t)^2}
\right), 
\end{align*}
where $\alpha, \beta, \gamma, \delta$ are all complex parameters. 
\par
While generic solutions of the Painlev\'e equations are 
``new transcendental functions,'' 
there are ``classical solutions'' which are expressible 
in terms of rational, algebraic or classical special functions for certain values of the parameters. 
In this paper, 
our concern is with the classical solutions and B\"acklund transformations 
which relate one solution to another solution of the same equation with different values of the parameters. 
\par
Examples of classical solutions are as follows: 
Airault \cite{Airault} constructed explicit rational solutions of $P_{\rm II}$ and $P_{\rm IV}$
with their B\"acklund transformations.
Milne, Clarkson and Bassom \cite{Milne} treated $P_{\rm III}$,  and
described their B\"acklund transformations
and exact solution hierarchies, which are given by rational, algebraic, or certain Bessel functions.
Bassom, Clarkson and Hicks \cite{Bassom-Clarkson-Hicks}
dealt with $P_{\rm IV}$, and
described their B\"acklund transformations and
exact solution hierarchies, which are expressed by rational functions, the parabolic cylinder functions or the complementary error functions.
Clarkson \cite{Clarkson}
studied some rational and algebraic solutions of $P_{\rm III}$ and
showed that
these solutions are expressible
in terms of special polynomials defined by second order, bilinear differential-difference equations
which are equivalent to the Toda equations.
\par
The rational solutions of $P_{\rm J} \,\,(\rm J=\rm II, III, IV, V, VI)$ 
were classified by 
Yablonski and Vorobev \cite{Yab:59,Vorob}, 
Gromak \cite{Gr:83,Gro}, 
Murata \cite{Mura1, Mura2}, 
Kitaev, Law and McLeod \cite{Kit-Law-McL}, 
Mazzoco \cite{Mazzo}, 
and 
Yuang and Li \cite{YuangLi}. 
\par
$P_{\rm J} \,\,(\rm J=\rm II, III, IV, V, VI)$ possesses the B\"acklund transformation group. 
It was shown by Okamoto \cite{oka1}, \cite{oka2}, \cite{oka3}, \cite{oka4} that 
the B\"acklund transformation groups of the Painlev\'e equations, except for $P_{\rm I},$ are isomorphic 
to the extended affine Weyl groups. 
For $ P_{\rm II},$ $P_{\rm III},$ $ P_{\rm IV}, P_{\rm V},$ and $P_{\rm VI}$, 
the B\"acklund transformation groups correspond to 
$A^{ ( 1 ) }_1,$ 
$A^{ ( 1 ) }_1 \bigoplus A^{ ( 1 ) }_1,$ 
$A^{ ( 1 ) }_2,$ 
$A^{ ( 1 ) }_3,$ and 
$D^{ ( 1 ) }_4$, 
respectively.
\par
Noumi and Yamada \cite{NoumiYamada-B} discovered the equation of type $A^{(1)}_l \,(l\geq 2),$ 
whose B\"acklund 
transformation group is isomorphic to the extended affine Weyl group $\tilde{W}(A^{(1)}_l)$. 
This system of differential equations is called the Noumi and Yamada system of type $A_l^{(1)}.$ 
The Noumi and Yamada systems of types $A_2^{(1)}$ and $A_3^{(1)}$ correspond to the fourth and fifth 
Painlev\'e equations, respectively. 
Noted is the fact that Murata \cite{Mura1} and Kitaev, Law and McLeod \cite{Kit-Law-McL} classified the rational solutions of 
the fourth and fifth Painlev\'e equations, respectively. 
Furthermore, 
we \cite{Matsuda1} classified the rational solutions of the Noumi and Yamada system of type $A_4^{(1)}.$ 
\par
The aim of this paper is to completely classify the rational solutions of 
the Noumi and Yamada system of type $A_5^{(1)},$ 
which is defined by 
\begin{equation*}
(*)
\begin{cases}
\frac{t}{2}
f_0^{\prime}
=
f_0
\left(
f_1f_2+f_1f_4+f_3 f_4-f_2f_3-f_2f_5-f_4 f_5
\right) 
+
\left(
\frac12
-\alpha_2-\alpha_4
\right)
f_0
+
\alpha_0
\left(
f_2+f_4
\right) \\
\frac{t}{2}
f_1^{\prime}
=
f_1
\left(
f_2f_3+f_2f_5+f_4 f_5-f_3f_4-f_3f_0-f_5 f_0
\right) 
+
\left(
\frac12
-\alpha_3-\alpha_5
\right)
f_1
+
\alpha_1
\left(
f_3+f_5
\right) \\
\frac{t}{2}
f_2^{\prime}
=
f_2
\left(
f_3f_4+f_3f_0+f_5 f_0-f_4f_5-f_4f_1-f_0 f_1
\right) 
+
\left(
\frac12
-\alpha_4-\alpha_0
\right)
f_2
+
\alpha_2
\left(
f_4+f_0
\right) \\
\frac{t}{2}
f_3^{\prime}
=
f_3
\left(
f_4f_5+f_4f_1+f_0 f_1-f_5f_0-f_5f_2-f_1 f_2
\right) 
+
\left(
\frac12
-\alpha_5-\alpha_1
\right)
f_3
+
\alpha_3
\left(
f_5+f_1
\right) \\
\frac{t}{2}
f_4^{\prime}
=
f_4
\left(
f_5f_0+f_5f_2+f_1f_2-f_0f_1-f_0f_3-f_2 f_3
\right) 
+
\left(
\frac12
-\alpha_0-\alpha_2
\right)
f_4
+
\alpha_4
\left(
f_0+f_2
\right) \\
\frac{t}{2}
f_5^{\prime}
=
f_5
\left(
f_0f_1+f_0f_3+f_2f_3-f_1f_2-f_1f_4-f_3 f_4
\right) 
+
\left(
\frac12
-\alpha_1-\alpha_3
\right)
f_5
+
\alpha_5
\left(
f_1+f_3
\right) \\
f_0+f_2+f_4=
f_1+f_3+f_5=t,
\quad
\alpha_0+
\alpha_1+
\alpha_2+
\alpha_3+
\alpha_4+
\alpha_5
=1.
\end{cases}
\end{equation*}
In this paper, $A_5^{(1)}(\alpha_j)_{0\leq j \leq 5}$ denotes 
the system of differential equations $(*)$. 
For $A_5^{(1)}(\alpha_j)_{0\leq j \leq 5}$, 
we consider the suffix of $f_i$ and $\alpha_i$ as elements of $\mathbb{Z} /6 \mathbb{Z}$.  
\par
$A_5^{(1)}(\alpha_j)_{0\leq j \leq 5}$ has the B\"acklund transformations, 
$s_0, s_1,s_2, s_3, s_4,s_5$ and $\pi$:
\begin{center}
\begin{tabular}{|c||c|c|c|c|c|c|c|}
\hline
$x$ & $s_0(x)$ & $s_1(x)$ & $s_2(x)$ & 
$s_3(x)$ & $s_4 (x)$ & $s_5(x)$ & $\pi(x)$ \\
\hline
$f_0$ & $f_0$ & $f_0-\alpha_1/f_1$ & $f_0$ & 
$f_0$ & $f_0$ & $f_0+\alpha_5/f_5$ & $f_1$ \\
\hline
$f_1$ & $f_1+\alpha_0/f_0$ & $f_1$ & 
$f_1-\alpha_2/f_2$ & $f_1$ & $f_1$ & $f_1$ &$f_2$ \\
\hline
$ f_2 $ & $ f_2 $ & $ f_2 + \alpha_1/f_1 $ & 
$ f_2 $ & $ f_2 - \alpha_3/f_3 $ & $ f_2 $ &$f_2$ & $ f_3 $ \\
\hline
$ f_3 $ & $ f_3 $ & $ f_3 $ & $ f_3 + \alpha_2/f_2 $ & 
$ f_3 $ & $ f_3 - \alpha_4/f_4 $ &$f_3$ & $ f_4 $ \\
\hline
$ f_4 $ & $ f_4$ & $ f_4 $ & $ f_4 $ & 
$ f_4 + \alpha_3/f_3 $ & $ f_4 $ & $f_4-\alpha_5/f_5$ & $ f_5 $ \\
\hline
$f_5$ & $f_5-\alpha_0/f_0$ & $f_5$ & $f_5$ & $f_5$ & $f_5+\alpha_4/f_4$ &$f_5$ &$f_0$  \\
\hline
\hline
$ \alpha_0 $ & $ - \alpha_0 $ & $ \alpha_0 + \alpha_1 $ & 
$ \alpha_0 $ & $ \alpha_0 $ & $ \alpha_0 $ & $\alpha_0+\alpha_5$ & $ \alpha_1 $ \\
\hline
$ \alpha_1 $ & $ \alpha_1 + \alpha_0 $ & $ - \alpha_1 $ & 
$ \alpha_1 + \alpha_2 $ & $ \alpha_1 $ & $ \alpha_1 $ & $\alpha_1$ &$ \alpha_2 $ \\
\hline
$ \alpha_2 $ & $ \alpha_2 $ & $ \alpha_2 + \alpha_1 $ & 
$ - \alpha_2 $ & $ \alpha_2 + \alpha_3 $ & $ \alpha_2 $ &$\alpha_2$ & $ \alpha_3 $ \\
\hline
$ \alpha_3 $ & $ \alpha_3 $ & $ \alpha_3 $ & $ \alpha_3 + \alpha_2 $ & 
$ - \alpha_3 $ & $ \alpha_3 + \alpha_4 $ & $\alpha_3$ & $ \alpha_4 $ \\
\hline
$ \alpha_4 $ & $ \alpha_4$ & $ \alpha_4 $ & $ \alpha_4 $ & 
$ \alpha_4 + \alpha_3 $ & $ - \alpha_4 $ & $\alpha_4+\alpha_5$ & $ \alpha_5 $ \\
\hline
$\alpha_5$ & $\alpha_5+\alpha_0$ & $\alpha_5$ & $\alpha_5$ & $\alpha_5$ &$\alpha_5+\alpha_4$ 
& $-\alpha_5$ & $\alpha_0$  \\
\hline
\end{tabular}
\end{center}
If $f_i \equiv 0$ for $i=0,1,2,3,4,5,$ 
which implies that $\alpha_i=0,$ 
then we consider $s_i$ as the identical transformation 
which is given by 
$$
s_i(f_j)=f_j \,\, \mathrm{and} \,\,s_i(\alpha_j)=\alpha_j \,\,(j=0,1,2,3,4,5).
$$
The B\"acklund transformation group $\langle s_0, s_1, s_2, s_3, s_4,s_5, \pi \rangle$
is isomorphic to the extended affine Weyl group $\tilde{W}(A^{(1)}_5)$.
\par
Our main theorem is as follows. 
\begin{theorem}
{\it
Suppose that for $A_5^{(1)}(\alpha_j)_{0\leq j \leq 5},$ 
there exists a rational solution. 
By some B\"acklund transformations, 
the parameters 
and solutions 
can then be transformed so that 
one of the following occurs:
\newline
(a-1)\quad $(\alpha_0,\alpha_1,\alpha_2,\alpha_3,\alpha_4,\alpha_5)=(\alpha_0,1-\alpha_0,0,0,0,0),$ 
and 
$$
(f_0,f_1,f_2,f_3,f_4,f_5)=(t,t,0,0,0,0),
$$
(a-2)\quad $(\alpha_0,\alpha_1,\alpha_2,\alpha_3,\alpha_4,\alpha_5)=(\alpha_0,0,0,1-\alpha_0,0,0),$ 
and 
$$
(f_0,f_1,f_2,f_3,f_4,f_5)=(t,0,0,t,0,0),
$$
(a-3)\quad $(\alpha_0,\alpha_1,\alpha_2,\alpha_3,\alpha_4,\alpha_5)=(0,1,0,0,0,0),$ 
and 
\begin{align*}
(f_0, f_1, f_2, f_3, f_4, f_5)
&=
(t,t,0,0,0,0), \,\,
(0,t,t,0,0,0), \,\,
(0,t,0,0,t,0),  \,\,
(t,t,t,0,-t,0),
\end{align*}
(b)\quad $(\alpha_0,\alpha_1,\alpha_2,\alpha_3,\alpha_4,\alpha_5)=(\alpha_0,-\alpha_0+1/2,\alpha_0,-\alpha_0+1/2,0,0),$ 
and 
$$
(f_0,f_1,f_2,f_3,f_4,f_5)=(t/2,t/2,t/2,t/2,0,0),
$$
(c)\quad  $(\alpha_0,\alpha_1,\alpha_2,\alpha_3,\alpha_4,\alpha_5)=(\alpha_4,-\alpha_4+1/3, \alpha_4,-\alpha_4+1/3,\alpha_4,-\alpha_4+1/3)$ 
and 
$$
(f_0,f_1,f_2,f_3,f_4,f_5)=(t/3,t/3,t/3,t/3,t/3,t/3).
$$
}
\end{theorem}
\par
Let us explain how 
this paper is organized. 
For this purpose, 
let us denote the coefficients of the Laurent series 
of $f_j \,\, (0\leq j \leq 5)$ at $t=\infty$ and $t=0$ by 
\begin{equation*}
\begin{cases}
a_{\infty,k}\,{\rm and} \,a_{0,k}, \, k \in \mathbb{Z} \,\,\,(for \,\,f_0),    \,\,
b_{\infty,k}\,{\rm and} \,b_{0,k}, \,  k \in \mathbb{Z} \,\,\,(for \,\,f_1),    \,\,
c_{\infty,k}\,{\rm and} \,c_{0,k}, \,  k \in \mathbb{Z} \,\,\,(for \,\,f_2),    \\
d_{\infty,k}\,{\rm and} \,d_{0,k}, \,  k \in \mathbb{Z} \,\,\,(for \,\,f_3),    \,\,
e_{\infty,k}\,{\rm and} \,e_{0,k}, \,  k \in \mathbb{Z} \,\,\,(for \,\,f_4),    \,\,
f_{\infty,k}\,{\rm and} \,f_{0,k}, \,  k \in \mathbb{Z} \,\,\,(for \,\,f_5),  \\
\end{cases}
\end{equation*}
respectively. 
For example, 
if all of $(f_j)_{0\leq j \leq 5}$ have a pole at $t=\infty,$ 
we set 
\begin{equation*}
\begin{cases}
f_0 =a_{\infty, n_0}t^{n_0}+a_{\infty, n_0-1}t^{n_0-1}+\cdots+a_{\infty,0}+a_{\infty,-1}t^{-1}+\cdots,   \\
f_1 =b_{\infty, n_1}t^{n_1}+b_{\infty, n_1-1}t^{n_1-1}+\cdots+b_{\infty,0}+b_{\infty,-1}t^{-1}+\cdots,   \\
f_2 =c_{\infty, n_2}t^{n_2}+c_{\infty, n_2-1}t^{n_2-1}+\cdots+c_{\infty,0}+c_{\infty,-1}t^{-1}+\cdots,   \\
f_3 =d_{\infty, n_3}t^{n_3}+d_{\infty, n_3-1}t^{n_3-1}+\cdots+d_{\infty,0}+d_{\infty,-1}t^{-1}+\cdots,   \\
f_4 =e_{\infty, n_4}t^{n_4}+e_{\infty, n_4-1}t^{n_4-1}+\cdots+e_{\infty,0}+e_{\infty,-1}t^{-1}+\cdots,   \\
f_5 =f_{\infty, n_5}t^{n_5}+f_{\infty, n_5-1}t^{n_5-1}+\cdots+f_{\infty,0}+f_{\infty,-1}t^{-1}+\cdots,   \\
\end{cases}
\end{equation*}
where $n_j \,(0\leq j \leq 5)$ are all positive integers and 
$
a_{\infty, n_0}, \, b_{\infty,n_1}, c_{\infty,n_2 }, \,
d_{\infty, n_3}, \, e_{\infty,n_4}, f_{\infty,n_5 }
\neq 0.$
Moreover, 
the coefficients of the Laurent series 
of the auxiliary function $H$ at $t=\infty$ and $t=0$ are defined by 
$h_{\infty, k}$ and $h_{0,k}, \,\,k\in\mathbb{Z},$ respectively.
\par
In Section 1, we treat the meromorphic solution near $t=\infty$ 
and find that the residues of $f_j \,(0\leq j \leq 5)$ at $t=\infty,$ 
\begin{align*}
&a_{\infty,-1}(=-\mathrm{Res}_{t=\infty} f_0), \, b_{\infty,-1} (=-\mathrm{Res}_{t=\infty} f_1), 
\,c_{\infty,-1}(=-\mathrm{Res}_{t=\infty} f_2), \\
&d_{\infty,-1}(=-\mathrm{Res}_{t=\infty} f_3), \,
e_{\infty,-1}(=-\mathrm{Res}_{t=\infty} f_4), \,f_{\infty,-1}(=-\mathrm{Res}_{t=\infty} f_5)
\end{align*}
are all expressed by the parameters $\alpha_j \,\,(0\leq j \leq 5).$ 
From the meromorphic solutions near $t=\infty$, 
we get three types of meromorphic solutions at $t=\infty,$ 
Type A, Type B and Type C. 
For example, 
the rational solutions of (a-1), (a-2) and (a-3) in the main theorem are the solutions of Type A. 
The solutions of (b) and (c) are of Type B and Type C, respectively. 
\par
In Section 2, we deal with the meromorphic solutions near $t=0$ 
and see that the residues of $f_j \,(0\leq j \leq 5)$ at $t=0,$ 
\begin{align*}
&a_{0,-1}(=\mathrm{Res}_{t=0} f_0), \,\, 
b_{0,-1}(=\mathrm{Res}_{t=0} f_1), \,\,
c_{0,-1}(=\mathrm{Res}_{t=0} f_2), \\
&d_{0,-1}(=\mathrm{Res}_{t=0} f_3), \,\,
e_{0,-1}(=\mathrm{Res}_{t=0} f_4), \,\,
f_{0,-1}(=\mathrm{Res}_{t=0} f_5)
\end{align*}  
are all expressed by the parameters $\alpha_j \,\,(0\leq j \leq 5).$
\par
In Section 3, we treat the meromorphic solutions 
near $t=c \in \mathbb{C}^{*},$ 
whose residues  
are half integers. 
Thus, we can prove that 
$-\mathrm{Res}_{t=\infty} f_j-\mathrm{Res}_{t=0} f_j \in \mathbb{Z}, (j=0,1,2,3,4,5),$ 
which are the necessary conditions for $A_5^{(1)}(\alpha_i)_{0\leq i \leq 5}$ to have rational solutions. 
\par
In Section 4, 
we define the auxiliary function $H$ 
for a meromorphic solution of $A_5^{(1)}(\alpha_i)_{0\leq i \leq 5}$ 
and treat the Laurent series of $H$ at $t=\infty, \,\,0, \,\,c\in\mathbb{C}^{*}$. 
Especially, 
we calculated the constant terms $h_{\infty,0}, h_{0,0}$ of the Laurent series of $H$ at $t=\infty, 0$ 
and 
computed the residue of $H$ at $t=c.$ 
$h_{\infty,0}, h_{0,0}$ are then expressed with the parameters $\alpha_j (0\leq j \leq 5)$ 
and 
the residue of $H$ at $t=c$ is $\epsilon c,$ where $\epsilon=1/6, 1/12, 5/12.$ 
Thus, 
we can show that 
$6(h_{0,0}-h_{\infty,0})$ is a non-positive integer, 
which is a necessary condition for $A_5^{(1)}(\alpha_i)_{0\leq i \leq 5}$ to have rational solutions. 
\par
In Sections 5, 6 and 7, 
we deal with the necessary conditions for 
$A_5^{(1)}(\alpha_i)_{0\leq i \leq 5}$ to have rational solutions 
of Type A, Type B and Type C, respectively. 
For this purpose, 
mainly using the residue calculus of $f_j \,\,(0\leq j \leq 5),$ that is, the formula 
$-\mathrm{Res}_{t=\infty}f_j -\mathrm{Res}_{t=0} f_j \in\mathbb{Z} \,\,(0\leq j \leq 5),$ 
we express necessary conditions by the parameters. 
\par
In Section 8, we 
transform the parameters $(\alpha_j)_{0\leq j \leq 5}$ into the standard forms. 
For Type A, Type B and Type C, 
there exist two, three and four kinds of standard forms, respectively. 
For example, in some cases, 
the standard forms of the parameters for Type A, Type B, and Type C are given by
\begin{equation*}
(\alpha_0,\alpha_1,\alpha_2,\alpha_3,\alpha_4,\alpha_5)
=
\begin{cases}
(\alpha_0,1-\alpha_0,0,0,0,0),  \\
(\alpha_0,-\alpha_0+1/2,\alpha_0,-\alpha_0+1/2,\alpha_0,0,0), \\
(\alpha_4,-\alpha_4+1/3, \alpha_4,-\alpha_4+1/3,\alpha_4,-\alpha_4+1/3),
\end{cases}
\end{equation*}
respectively.
\par
In Section 9, 
we determine the rational solutions of Type A  
for the standard forms of the parameters. 
For the purpose, 
we use the residue calculus of $f_j \,\,(0\leq j \leq 5). $
\par
In Sections 10 and 11, 
we determine 
the rational solutions 
of Type B and Type C of 
$A_5^{(1)}(\alpha_i)_{0\leq i \leq 5}$ 
when the parameters are the standard forms, respectively. 
For the purpose, 
we mainly use the residue calculus of $H,$ that is, the formula $6(h_{\infty,0}-h_{0,0})\in\mathbb{Z}.$
\par
In Section 12, 
we completely classify the rational solutions of 
$A_5^{(1)}(\alpha_j)_{0\leq j \leq 5}$ 
and prove the main theorems for Type A, Type B and Type C, 
that is, 
Theorems \ref{thm:a5mainA}, 
\ref{thm:a5mainB} and \ref{thm:a5mainC}.

\section{Meromorphic Solutions at $t=\infty$}
In this section, 
for $A_5^{(1)}(\alpha_j)_{0\leq  j \leq 5},$ 
we treat the meromorphic solutions at $t=\infty$ 
and 
calculate the Laurent series of 
$f_j \,\,(0\leq j \leq 5)$ at $t=\infty$. 
The residues of $f_j \,\,(0\leq j \leq 5)$ at $t=\infty$ are then 
expressed by the parameters $\alpha_j \,(0\leq j \leq 5)$. 
For the purpose, 
we consider the following seven cases:
\newline
(0)\quad none of $(f_i)_{0\leq i \leq 5}$ has a pole at $t=\infty,$
\newline
(1)\quad one of $(f_i)_{0\leq i \leq 5}$ has a pole at $t=\infty,$
\newline
(2)\quad two of $(f_i)_{0\leq i \leq 5}$ have a pole at $t=\infty,$
\newline
(3)\quad three of $(f_i)_{0\leq i \leq 5}$ have a pole at $t=\infty,$
\newline
(4)\quad four of $(f_i)_{0\leq i \leq 5}$ have a pole at $t=\infty,$
\newline
(5)\quad five of $(f_i)_{0\leq i \leq 5}$ have a pole at $t=\infty,$
\newline
(6)\quad all of $(f_i)_{0\leq i \leq 5}$ have a pole at $t=\infty.$
\par
Since 
$f_0+f_2+f_4=t$ 
and 
$f_1+f_3+f_5=t,$ 
cases (0) and (1) are both impossible.

\subsection{Two of $(f_i)_{0\leq i \leq 5}$ have a pole at $t=\infty$}
In this subsection, 
we suppose that two of $(f_i)_{0\leq i \leq 5}$ have a pole at $t=\infty$ 
and 
calculate the Laurent series of $f_i \,\,(0\leq i \leq 5)$ at $t=\infty$ 
for $A_5^{(1)}(\alpha_i)_{0\leq i \leq 5}.$ 
Since 
$f_0+f_2+f_4=t$ 
and 
$f_1+f_3+f_5=t,$
by $\pi,$ 
we have only to consider the following two cases:
\newline
(1)\quad 
for some $i=0,1,2,3,4,5,$ $f_i, f_{i+1}$ both have a pole at $t=\infty,$
\newline
(2)\quad 
for some $i=0,1,2,3,4,5,$ $f_i, f_{i+3}$ both have a pole at $t=\infty.$

\subsubsection{$f_i, f_{i+1}$ have a pole at $t=\infty$}

\begin{proposition}
\label{prop:a5inf-01-1}
{\it
Suppose that 
for $A_5^{(1)}(\alpha_j)_{0\leq j \leq 5},$ 
there exists a solution such that 
for some $i=0,1,2,3,4,5,$ 
$f_i, f_{i+1}$ both have a pole at $t=\infty$ and 
$f_{i+2}, f_{i+3}, f_{i+4}, f_{i+5}$ 
are all holomorphic at $t=\infty.$ 
$f_i, f_{i+1}$ then both 
have a pole of order one at $t=\infty$. 
We denote this case by Type A (1).
}
\end{proposition}

\begin{proof}
The proposition follows from 
the fact that 
$f_0+f_2+f_4=t$ 
and 
$f_1+f_3+f_5=t.$
\end{proof}
In order to compute the residues, we have

\begin{proposition}
\label{prop:a5inf-01-2}
{\it
Suppose that 
for $A_5^{(1)}(\alpha_j)_{0\leq j \leq 5},$ 
there exists a solution  
such that 
for some $i=0,1,2,3,4,5,$ 
$f_i, f_{i+1}$ both have a pole at $t=\infty$ and 
$f_{i+2}, f_{i+3}, f_{i+4}, f_{i+5}$ 
are all holomorphic at $t=\infty.$ 
Then, 
\begin{equation*}
\begin{cases}
f_i
=
t
-\left(\alpha_{i+2}+\alpha_{i+4}\right)t^{-1}
+
\cdots \\
f_{i+1}
=
t
+\left(\alpha_{i+3}+\alpha_{i+5}\right)t^{-1}
+
\cdots \\
f_{i+2}
=
\alpha_{i+2} t^{-1}
+ \cdots \\
f_{i+3}
=
- \alpha_{i+3} t^{-1}
+ \cdots \\
f_{i+4}
=
\alpha_{i+4} t^{-1}
+ \cdots \\
f_{i+5}
=
- \alpha_{i+5} t^{-1}
+ \cdots.
\end{cases}
\end{equation*}
}
\end{proposition}

\begin{proof}
By $\pi,$ 
we assume that $f_0,f_1$ both have a pole at $t=\infty.$ 
Since 
$f_0+f_2+f_4=t$ 
and 
$f_1+f_3+f_5=t,$ 
it follows that $a_{\infty,1}=b_{\infty,1}=1.$ 
\par
By comparing the coefficients of the term $t^2$ in
\begin{equation*}
\frac{t}{2}
f_2^{\prime}
=
f_2
\left(
f_3f_4+f_3f_0+f_5 f_0-f_4f_5-f_4f_1-f_0 f_1
\right) 
+
\left(
\frac12
-\alpha_4-\alpha_0
\right)
f_2
+
\alpha_2
\left(
f_4+f_0
\right), 
\end{equation*}
we have $c_{\infty,0}=0.$ 
Moreover, by comparing the coefficients of the term $t$ in this equation, 
we get $c_{\infty,-1}=\alpha_2.$ 
\par
By comparing the coefficients of the term $t^2$ in
\begin{equation*}
\frac{t}{2}
f_3^{\prime}
=
f_3
\left(
f_4f_5+f_4f_1+f_0 f_1-f_5f_0-f_5f_2-f_1 f_2
\right) 
+
\left(
\frac12
-\alpha_5-\alpha_1
\right)
f_3
+
\alpha_3
\left(
f_5+f_1
\right), 
\end{equation*}
we have $d_{\infty,0}=0.$ 
Furthermore, 
by comparing the coefficients of the term $t$ in this equation, 
we get $d_{\infty,-1}=-\alpha_3.$
\par
By comparing the coefficients of the term $t^2$ in
\begin{equation*}
\frac{t}{2}
f_4^{\prime}
=
f_4
\left(
f_5f_0+f_5f_2+f_1f_2-f_0f_1-f_0f_3-f_2 f_3
\right) 
+
\left(
\frac12
-\alpha_0-\alpha_2
\right)
f_4
+
\alpha_4
\left(
f_0+f_2
\right), 
\end{equation*}
we have $e_{\infty,0}=0.$ 
Moreover, by comparing the coefficients of the term $t$ in this equation, 
we get $e_{\infty,-1}=\alpha_4.$ 
\par
By comparing the coefficients of the term $t^2$ in
\begin{equation*}
\frac{t}{2}
f_5^{\prime}
=
f_5
\left(
f_0f_1+f_0f_3+f_2f_3-f_1f_2-f_1f_4-f_3 f_4
\right) 
+
\left(
\frac12
-\alpha_1-\alpha_3
\right)
f_5
+
\alpha_5
\left(
f_1+f_3
\right),
\end{equation*}
we have $f_{\infty,0}=0.$ 
Furthermore,  by comparing the coefficients of the term $t$ in this equation, 
we get $f_{\infty,-1}=-\alpha_5.$
\par
Since
$f_0+f_2+f_4=t$ 
and 
$f_1+f_3+f_5=t,$ 
it follows that 
\begin{equation*}
a_{\infty,0}=0, \,\,a_{\infty,-1}=-\alpha_2-\alpha_4 \,\,\mathrm{and} \,\,
b_{\infty,0}=0, \,\,b_{\infty,-1}=\alpha_3+\alpha_5,
\end{equation*}
respectively.
\end{proof}

In order to prove the uniqueness of the Laurent series, 
we have

\begin{proposition}
\label{prop:a5inf-01-3}
{\it
Suppose that 
for
$A_5^{(1)}(\alpha_j)_{0\leq j \leq 5},$ 
there exists a solution 
such that   
for some $i=0,1,2,3,4,5,$ 
$f_i, f_{i+1}$ both have a pole at $t=\infty$ and 
$f_{i+2}, f_{i+3}, f_{i+4}, f_{i+5}$ 
are all holomorphic at $t=\infty.$ 
Then, 
it is unique.
}
\end{proposition}

\begin{proof}
By comparing the coefficients of the term $t^{-(k-2)} \,\,(k\geq 2)$ in
\begin{equation*}
\frac{t}{2}
f_2^{\prime}
=
f_2
\left(
f_3f_4+f_3f_0+f_5 f_0-f_4f_5-f_4f_1-f_0 f_1
\right) 
+
\left(
\frac12
-\alpha_4-\alpha_0
\right)
f_2
+
\alpha_2
\left(
f_4+f_0
\right), 
\end{equation*}
we have
\begin{align*}
c_{\infty,-k}
&=
\frac12 (k-2) c_{\infty,-(k-2)}\\
&
+
\sum
c_{\infty,-l}
(
d_{\infty,-m}e_{\infty, -n}+d_{\infty, -m}a_{\infty, -n}+f_{\infty, -m}a_{\infty, -n}   \\
&\hspace{30mm}
-e_{\infty, -m}f_{\infty, -n}-e_{\infty, -m}b_{\infty, -n}-a_{\infty, -m}b_{\infty, -n}
) \\
&+c_{\infty, -(k-2)}(d_{\infty, -1}+f_{\infty, -1}-e_{\infty, -1}-a_{\infty, -1}-b_{\infty, -1}) \\
&
+\left(
\frac12-\alpha_4-\alpha_0
\right)
c_{\infty,-(k-2)}
+\alpha_2(e_{\infty,-(k-2)}+a_{\infty,-(k-2)}),
\end{align*}
where 
the sum extends over the integers $l,m$ and $n$ for which $l+m+n=k-2$ and $l,m,n\geq 1.$
\par
In the same way, 
we have 
\begin{align*}
d_{\infty,-k}
=&
-\frac12 (k-2) d_{\infty,-(k-2)}\\
&
-
\sum
d_{\infty,-l}
(
e_{\infty,-m}f_{\infty, -n}+e_{\infty, -m}b_{\infty, -n}+a_{\infty, -m}b_{\infty, -n}   \\
&\hspace{30mm}
-f_{\infty, -m}a_{\infty, -n}-f_{\infty, -m}c_{\infty, -n}-b_{\infty, -m}c_{\infty, -n}
) \\
&-d_{\infty, -(k-2)}(e_{\infty, -1}+a_{\infty, -1}+b_{\infty, -1}-f_{\infty, -1}-c_{\infty, -1}) \\
&
-\left(
\frac12-\alpha_5-\alpha_1
\right)
d_{\infty,-(k-2)}
-\alpha_3(f_{\infty,-(k-2)}+b_{\infty,-(k-2)}),
\end{align*}
\begin{align*}
e_{\infty,-k}
&=
\frac12 (k-2) e_{\infty,-(k-2)}\\
&
+
\sum
e_{\infty,-l}
(
f_{\infty,-m}a_{\infty, -n}+f_{\infty, -m}c_{\infty, -n}+b_{\infty, -m}c_{\infty, -n}   \\
&\hspace{30mm}
-a_{\infty, -m}b_{\infty, -n}-a_{\infty, -m}d_{\infty, -n}-c_{\infty, -m}d_{\infty, -n}
) \\
&+e_{\infty, -(k-2)}(f_{\infty, -1}+c_{\infty, -1}-a_{\infty, -1}-b_{\infty, -1}-d_{\infty, -1})  \\
&
+\left(
\frac12-\alpha_0-\alpha_2
\right)
e_{\infty,-(k-2)}
+\alpha_4(a_{\infty,-(k-2)}+c_{\infty,-(k-2)}),
\end{align*}
\begin{align*}
f_{\infty,-k}
&=
-
\frac12 (k-2) f_{\infty,-(k-2)}\\
&
-
\sum
f_{\infty,-l}
(
a_{\infty,-m}b_{\infty, -n}+a_{\infty, -m}d_{\infty, -n}+c_{\infty, -m}d_{\infty, -n}   \\
&\hspace{30mm}
-b_{\infty, -m}c_{\infty, -n}-b_{\infty, -m}d_{\infty, -n}-d_{\infty, -m}e_{\infty, -n}
) \\
&-f_{\infty, -(k-2)}(a_{\infty, -1}+b_{\infty, -1}+d_{\infty, -1}-c_{\infty, -1}-e_{\infty, -1})  \\
&
-\left(
\frac12-\alpha_1-\alpha_3
\right)
f_{\infty,-(k-2)}
-\alpha_5(b_{\infty,-(k-2)}+d_{\infty,-(k-2)}).
\end{align*}
Thus, $c_{\infty,-k}, d_{\infty,-k}, e_{\infty, -k}, f_{\infty, -k} \,\,(k\geq2)$ 
are inductively determined. 
Moreover, 
since
$f_0+f_2+f_4=t$ 
and 
$f_1+f_3+f_5=t,$ 
$a_{\infty,-k}, b_{\infty,-k} \,\,(k\geq 2)$ 
are also inductively determined. 
\end{proof}

\subsubsection{$f_i, f_{i+3}$ have a pole at $t=\infty$}
We consider the case in which 
for some $i=0,1,2,3,4,5,$ 
$f_i, f_{i+3}$ both have a pole at $t=\infty.$ 
We can prove 
Propositions \ref{prop:a5inf-03-1}, \ref{prop:a5inf-03-2} and \ref{prop:a5inf-03-3} 
in the same way as 
Propositions \ref{prop:a5inf-01-1}, \ref{prop:a5inf-01-2} and \ref{prop:a5inf-01-3}, 
respectively.

\begin{proposition}
\label{prop:a5inf-03-1}
{\it 
Suppose that 
for $A_5^{(1)}(\alpha_j)_{0\leq j \leq 5},$ 
there exists a solution  
such that 
for some $i=0,1,2,3,4,5,$ 
$f_i, f_{i+3}$ both have a pole at $t=\infty$ and 
$f_{i+1}, f_{i+2}, f_{i+4}, f_{i+5}$ 
are all holomorphic at $t=\infty.$ 
$f_i, f_{i+3}$ then both 
have a pole of order one at $t=\infty$. 
We denote this case by Type A (2).
}
\end{proposition}

\begin{proposition}
\label{prop:a5inf-03-2}
{\it 
Suppose that 
for 
$A_5^{(1)}(\alpha_j)_{0\leq j \leq 5},$ there exists a solution 
such that  
for some $i=0,1,2,3,4,5,$ 
$f_i, f_{i+3}$ both have a pole at $t=\infty$ and 
$f_{i+1}, f_{i+2}, f_{i+4}, f_{i+5}$ 
are all holomorphic at $t=\infty.$ 
Then, 
\begin{equation*}
\begin{cases}
f_i
=
t
+\left(\alpha_{i+2}-\alpha_{i+4}\right)t^{-1}
+
\cdots \\
f_{i+1}
=
\alpha_{i+1}t^{-1}
+
\cdots \\
f_{i+2}
=
-\alpha_{i+2} t^{-1}
+ \cdots \\
f_{i+3}
=
t
+\left(\alpha_{i+5}-\alpha_{i+1}\right) t^{-1}
+ \cdots \\
f_{i+4}
=
\alpha_{i+4} t^{-1}
+ \cdots \\
f_{i+5}
=
- \alpha_{i+5} t^{-1}
+ \cdots.
\end{cases}
\end{equation*}
}
\end{proposition}

\begin{proposition}
\label{prop:a5inf-03-3}
{\it 
Suppose that 
for $A_5^{(1)}(\alpha_j)_{0\leq j \leq 5},$ 
there exists a solution 
such that  
for some $i=0,1,2,3,4,5,$ 
$f_i, f_{i+3}$ both have a pole at $t=\infty$ and 
$f_{i+1}, f_{i+2}, f_{i+4}, f_{i+5}$ 
are all holomorphic at $t=\infty.$ 
It is then unique.
}
\end{proposition}

\subsection{Three of $(f_j)_{0\leq j \leq 5}$ have a pole at $t=\infty$}
In this subsection, 
we treat the case in which 
three of $(f_j)_{0\leq j \leq 5}$ have a pole at $t=\infty.$ 
By $\pi,$ 
we have only to consider the following three cases:
\newline
(1)\quad for some $i=0,1,2,3,4,5,$ 
$f_i, f_{i+1}, f_{i+2}$ all have a pole at $t=\infty,$
\newline
(2)\quad for some $i=0,1,2,3,4,5,$ 
$f_i, f_{i+1}, f_{i+3}$ all have a pole at $t=\infty,$
\newline
(3)\quad for some $i=0,1,2,3,4,5,$ 
$f_i, f_{i+1}, f_{i+4}$ all have a pole at $t=\infty.$

\subsubsection{$f_i, f_{i+1}, f_{i+2}$ have a pole at $t=\infty$}

\begin{proposition}
\label{prop:a5inf-012}
{\it  
For $A_5^{(1)}(\alpha_j)_{0\leq j \leq 5},$ 
there exists no solution such that 
for some $i=0,1,2,3,4,5,$ 
$f_{i}, f_{i+1}, f_{i+2}$ all 
have a pole at $t=\infty$ 
and $f_{i+3}, f_{i+4}, f_{i+5}$ are all holomorphic at $t=\infty.$
}
\end{proposition}

\begin{proof}
By $\pi,$ 
we assume that 
$f_0, f_1, f_2$ all have a pole of order $n_0, n_1, n_2$ at $t=\infty,$ 
respectively, 
where 
$n_0, n_1, n_2$ are all positive integers.
\par
By comparing the coefficients of the highest terms in
\begin{equation*}
\frac{t}{2}
f_0^{\prime}
=
f_0
\left(
f_1f_2+f_1f_4+f_3 f_4-f_2f_3-f_2f_5-f_4 f_5
\right) 
+
\left(
\frac12
-\alpha_2-\alpha_4
\right)
f_0
+
\alpha_0
\left(
f_2+f_4
\right), 
\end{equation*}
we have $a_{\infty, n_0} b_{\infty, n_1} c_{\infty, n_2}=0,$ 
which is a contradiction. 
\end{proof}

\subsubsection{$f_i, f_{i+1}, f_{i+3}$ have a pole at $t=\infty$}

\begin{proposition}
\label{prop:a5inf-013}
{\it 
For 
$A_5^{(1)}(\alpha_j)_{0\leq j \leq 5},$ 
there exists no solution such that 
for some $i=0,1,2,3,4,5,$ 
$f_{i}, f_{i+1}, f_{i+3}$ all 
have a pole at $t=\infty$ 
and 
$f_{i+2}, f_{i+4}, f_{i+5}$ 
are all holomorphic at $t=\infty.$
}
\end{proposition}

\begin{proof}
It can be proved in the same way as Proposition \ref{prop:a5inf-012}.

\end{proof}

\subsubsection{$f_i, f_{i+1}, f_{i+4}$ have a pole at $t=\infty$}

\begin{proposition}
\label{prop:a5inf-014}
{\it 
For $A_5^{(1)}(\alpha_j)_{0\leq j \leq 5},$ 
there exists no solution 
such that 
for some $i=0,1,2,3,4,5,$ 
$f_{i}, f_{i+1}, f_{i+4}$ all 
have a pole at $t=\infty$ 
and 
$f_{i+2}, f_{i+3},f_{i+5}$ 
are all holomorphic at $t=\infty.$
}
\end{proposition}

\begin{proof}
It can be proved in the same way as Proposition \ref{prop:a5inf-012}.

\end{proof}

\subsection{Four of $(f_j)_{0\leq j \leq 5}$ have a pole at $t=\infty$}
In this subsection, 
we deal with the case where 
four of $(f_j)_{0\leq j \leq 5}$ have a pole at $t=\infty.$ 
By $\pi,$ 
we have only to consider the following three cases:
\newline
(1)\quad for some $i=0,1,2,3,4,5,$ 
$f_i, f_{i+1}, f_{i+2}, f_{i+3}$ all have a pole at $t=\infty,$
\newline
(2)\quad for some $i=0,1,2,3,4,5,$ 
$f_i, f_{i+1}, f_{i+2}, f_{i+4}$ all have a pole at $t=\infty,$
\newline
(3)\quad for some $i=0,1,2,3,4,5,$ 
$f_i, f_{i+2}, f_{i+3}, f_{i+5}$ all have a pole at $t=\infty.$

\subsubsection{$f_i, f_{i+1}, f_{i+2}, f_{i+3}$ have a pole at $t=\infty$}

\begin{proposition}
\label{prop:a5inf-0123-1}
{\it 
Suppose that 
for $A_5^{(1)}(\alpha_j)_{0\leq j \leq 5},$ 
there exists a solution such that 
for some $i=0,1,2,3,4,5,$ 
$f_i, f_{i+1}, f_{i+2}, f_{i+3}$ all have a pole at $t=\infty$ and 
$f_{i+4}, f_{i+5}$ 
are both holomorphic at $t=\infty.$ 
$f_i, f_{i+1}, f_{i+2}, f_{i+3}$ then all 
have a pole of order one at $t=\infty.$ 
We denote this case as Type B.
}
\end{proposition}

\begin{proof}
By $\pi,$ 
we assume that 
$f_0, f_1, f_2, f_3$ all have a pole of order $n_0, n_1, n_2, n_3 $ at $t=\infty,$ 
where 
$n_0, n_1, n_2, n_3 $ are all positive integers. 
Since $f_0 +f_2 +f_4=t$ and $f_1 +f_3 +f_5=t,$ 
it follows that $n_0=n_2$ and $n_1=n_3,$ respectively. 
\par
By comparing the coefficients of the highest terms in
\begin{equation*}
\frac{t}{2}
f_0^{\prime}
=
f_0
\left(
f_1f_2+f_1f_4+f_3 f_4-f_2f_3-f_2f_5-f_4 f_5
\right) 
+
\left(
\frac12
-\alpha_2-\alpha_4
\right)
f_0
+
\alpha_0
\left(
f_2+f_4
\right), 
\end{equation*}
we have $b_{\infty, n_1} =d_{\infty, n_3},$ 
which implies that $n_1=n_3=1$ and $b_{\infty, 1}= d_{\infty, 1}=\displaystyle \frac12,$ 
because $f_1 +f_3 +f_5=t.$
\par
By comparing the coefficients of the highest terms in
\begin{equation*}
\frac{t}{2}
f_2^{\prime}
=
f_2
\left(
f_3f_4+f_3f_0+f_5 f_0-f_4f_5-f_4f_1-f_0 f_1
\right) 
+
\left(
\frac12
-\alpha_4-\alpha_0
\right)
f_2
+
\alpha_2
\left(
f_4+f_0
\right),
\end{equation*}
we get $a_{\infty, n_0} =c_{\infty, n_2},$ 
which implies that $n_0=n_2=1$ and $a_{\infty, 1}= c_{\infty, 1}=\displaystyle \frac12,$ 
because $f_0 +f_2 +f_4=t.$
\end{proof}

In order to calculate the residues of the Laurent series, 
we have

\begin{proposition}
\label{prop:a5inf-0123-2}
{\it 
Suppose that 
for 
$A_5^{(1)}(\alpha_j)_{0\leq j \leq 5},$ 
there exists a solution 
such that  
for some $i=0,1,2,3,4,5,$ 
$f_i, f_{i+1}, f_{i+2}, f_{i+3}$ all have a pole at $t=\infty$ and 
$f_{i+4}, f_{i+5}$ 
are both holomorphic at $t=\infty.$ 
Then, 
\begin{equation*}
\begin{cases}
f_i
=
\displaystyle \frac12 t +
\left(
\alpha_{i+1}-\alpha_{i+3}-2\alpha_{i+4}-\alpha_{i+5}
\right)
t^{-1}
+ \cdots \\
\displaystyle f_{i+1}
=
\frac12 t +
\left(
-\alpha_i+\alpha_{i+2}-\alpha_{i+4}
\right)
t^{-1}
+ \cdots \\
\displaystyle f_{i+2}
=
\frac12 t +
\left(
-\alpha_{i+1}+\alpha_{i+3}+\alpha_{i+5}
\right) t^{-1}
+ \cdots \\
\displaystyle f_{i+3}
=
\frac12 t +
\left(
\alpha_i-\alpha_{i+2}+\alpha_{i+4}+2\alpha_{i+5}
\right)
t^{-1}
+ \cdots \\
\displaystyle f_{i+4}
=
2\alpha_{i+4}t^{-1}
+ \cdots \\
\displaystyle f_{i+5}
=
-2\alpha_{i+5}t^{-1}
+ \cdots.
\end{cases}
\end{equation*}
}
\end{proposition}

\begin{proof}
By $\pi,$ 
we assume that $f_0, f_1, f_2, f_3$ all have a pole at $t=\infty.$ 
By comparing the coefficients of the terms $t^2$ and $t$ in 
\begin{equation*}
\frac{t}{2}
f_4^{\prime}
=
f_4
\left(
f_5f_0+f_5f_2+f_1f_2-f_0f_1-f_0f_3-f_2 f_3
\right) 
+
\left(
\frac12
-\alpha_0-\alpha_2
\right)
f_4
+
\alpha_4
\left(
f_0+f_2
\right),
\end{equation*} 
we have $e_{\infty,0}=0$ and $e_{\infty, -1}=2\alpha_4,$ 
respectively.
\par
By comparing the coefficients of the terms $t^2$ and $t$ in 
\begin{equation*}
\frac{t}{2}
f_5^{\prime}
=
f_5
\left(
f_0f_1+f_0f_3+f_2f_3-f_1f_2-f_1f_4-f_3 f_4
\right) 
+
\left(
\frac12
-\alpha_1-\alpha_3
\right)
f_5
+
\alpha_5
\left(
f_1+f_3
\right),
\end{equation*}
we get $f_{\infty,0}=0$ and $f_{\infty,-1}=-2\alpha_5,$ 
respectively.
\par
By comparing the coefficients of the terms $t^2$ and $t$ in 
\begin{equation*}
\frac{t}{2}
f_0^{\prime}
=
f_0
\left(
f_1f_2+f_1f_4+f_3 f_4-f_2f_3-f_2f_5-f_4 f_5
\right) 
+
\left(
\frac12
-\alpha_2-\alpha_4
\right)
f_0
+
\alpha_0
\left(
f_2+f_4
\right), 
\end{equation*}
we have $b_{\infty,0}=d_{\infty,0}=0$ and 
\begin{equation*}
d_{\infty, -1}-b_{\infty, -1}
=2(\alpha_0-\alpha_2+\alpha_4+\alpha_5),
\end{equation*}
which implies that 
\begin{equation*}
b_{\infty,-1}=-\alpha_0+\alpha_2-\alpha_4, \,\,
d_{\infty,-1}=\alpha_0-\alpha_2+\alpha_4+2\alpha_5,
\end{equation*}
because $f_1+f_3+f_5=t.$
\par
By comparing the coefficients of the terms $t^2$ and $t$ in 
\begin{equation*}
\frac{t}{2}
f_1^{\prime}
=
f_1
\left(
f_2f_3+f_2f_5+f_4 f_5-f_3f_4-f_3f_0-f_5 f_0
\right) 
+
\left(
\frac12
-\alpha_3-\alpha_5
\right)
f_1
+
\alpha_1
\left(
f_3+f_5
\right),
\end{equation*}
we get $a_{\infty,0}=c_{\infty,0}=0$ and 
\begin{equation*}
c_{\infty, -1}-a_{\infty, -1}
=2(-\alpha_1+\alpha_3+\alpha_4+\alpha_5),
\end{equation*}
which implies that 
\begin{equation*}
a_{\infty,-1}=\alpha_1-\alpha_3-2\alpha_4-\alpha_5, \,\,
c_{\infty,-1}=-\alpha_1+\alpha_3+\alpha_5,
\end{equation*}
because $f_0+f_2+f_4=t.$

\end{proof}

In order to prove the uniqueness of the Laurent series, 
we have

\begin{proposition}
\label{prop:a5inf-0123-3}
{\it 
Suppose that 
for $A_5^{(1)}(\alpha_j)_{0\leq j \leq 5},$ 
there exists a solution 
such that 
for some $i=0,1,2,3,4,5,$ 
$f_i, f_{i+1}, f_{i+2}, f_{i+3}$ all have a pole at $t=\infty$ and 
$f_{i+4}, f_{i+5}$ 
are both holomorphic at $t=\infty.$ 
It is then unique.
}
\end{proposition}

\begin{proof}
By $\pi,$ 
we assume that $f_0, f_1, f_2, f_3$ all have a pole at $t=\infty.$ 
\par
By comparing the coefficients of the terms $t^{-(k-2)} \,\,(k\geq 2)$ 
in 
\begin{equation*}
\frac{t}{2}
f_4^{\prime}
=
f_4
\left(
f_5f_0+f_5f_2+f_1f_2-f_0f_1-f_0f_3-f_2 f_3
\right) 
+
\left(
\frac12
-\alpha_0-\alpha_2
\right)
f_4
+
\alpha_4
\left(
f_0+f_2
\right),
\end{equation*} 
we have 
\begin{align*}
e_{-k}
&=
(k-2)e_{-(k-2)} \\
&+2
\sum
(
e_{\infty, -l}f_{\infty, -m}-e_{\infty, -l}a_{\infty, -m}-e_{\infty, -l}d_{\infty, -m}
)  \\
&+4
\sum
(e_{\infty, -l}f_{\infty, -m}a_{\infty,-n}-e_{\infty, -l}c_{\infty, -m}d_{\infty,-n})  \\
&+
2
\left(
\frac12-\alpha_0-\alpha_2-\alpha_4
\right)e_{\infty,-(k-2)},
\end{align*}
where the first sum extends over the positive integers $l,m$ for which $l+m=k-3,$ 
and 
the second sum extends over the positive integers $l,m,n$ for which $l+m+n=k-2.$ 
$e_{\infty,-k} \,\,(k\geq 2)$ are then inductively determined. 
\par
By comparing the coefficients of the terms $t^{-(k-2)} \,\,(k\geq 2)$ 
in 
\begin{equation*}
\frac{t}{2}
f_5^{\prime}
=
f_5
\left(
f_0f_1+f_0f_3+f_2f_3-f_1f_2-f_1f_4-f_3 f_4
\right) 
+
\left(
\frac12
-\alpha_1-\alpha_3
\right)
f_5
+
\alpha_5
\left(
f_1+f_3
\right),
\end{equation*}
we obtain
\begin{align*}
f_{\infty, -k}
&=
-(k-2)f_{\infty, -(k-2)}  \\
&-2
\sum
(
f_{\infty, -l}a_{\infty, -m}+f_{\infty, -l}d_{\infty, -m}-f_{\infty, -l}e_{\infty, -m}
)  \\
&-4
\sum
(f_{\infty, -l}a_{\infty, -m}b_{\infty,-n}-f_{\infty, -l}d_{\infty, -m}e_{\infty,-n}) \\
&
-
2
\left(
\frac12-\alpha_1-\alpha_3-\alpha_5
\right)f_{\infty,-(k-2)},
\end{align*}
which implies that $f_{\infty,-k} \,\,(k\geq 2)$ are inductively determined. 
\par
By comparing the coefficients of the terms $t^{-(k-2)} \,\,(k\geq 2)$ 
in 
\begin{equation*}
\frac{t}{2}
f_0^{\prime}
=
f_0
\left(
f_1f_2+f_1f_4+f_3 f_4-f_2f_3-f_2f_5-f_4 f_5
\right) 
+
\left(
\frac12
-\alpha_2-\alpha_4
\right)
f_0
+
\alpha_0
\left(
f_2+f_4
\right), 
\end{equation*}
we obtain
\begin{align*}
b_{\infty,-k}-d_{\infty, -k}
&=
-2(k-2)a_{\infty,-(k-2)}-2e_{\infty,-k}+f_{\infty,-k}  \\
&-4\sum(a_{\infty,-l}e_{\infty,-m}+d_{\infty,-l}e_{\infty,-m}) \\
&-8\sum(a_{\infty,-l}b_{\infty,-m}c_{\infty,-n}-a_{\infty,-l}e_{\infty,-m}f_{\infty,-n})  \\
&-4\left(\frac12-\alpha_0-\alpha_2-\alpha_4\right)a_{\infty,-k}, 
\end{align*}
which implies that $b_{\infty,-k}, \,\,d_{\infty,-k} \,\,(k\geq 2)$ are inductively determined, 
because $f_1+f_3+f_5=t.$
\par
By comparing the coefficients of the terms $t^{-(k-2)} \,\,(k\geq 2)$ 
in 
\begin{equation*}
\frac{t}{2}
f_1^{\prime}
=
f_1
\left(
f_2f_3+f_2f_5+f_4 f_5-f_3f_4-f_3f_0-f_5 f_0
\right) 
+
\left(
\frac12
-\alpha_3-\alpha_5
\right)
f_1
+
\alpha_1
\left(
f_3+f_5
\right),
\end{equation*}
we have
\begin{align*}
c_{\infty,-k}-a_{\infty,-k}
&=
-2(k-2)b_{\infty,-k}  \\
&-4\sum e_{\infty,-l}f_{\infty,-m}  \\
&-8\sum 
(
b_{\infty,-l}c_{\infty,-m}f_{\infty,-n}-b_{\infty,-l}f_{\infty, -m}a_{\infty, -n}
)  \\
&-4
\left(
\frac12-\alpha_1-\alpha_3-\alpha_5b_{\infty,-(k-2)}
\right),
\end{align*}
which implies that $a_{\infty,-k}, \,\,c_{\infty,-k} \,\,(k\geq 2)$ are inductively determined, 
because $f_0+f_2+f_4=t.$

\end{proof}

\subsubsection{$f_i, f_{i+1}, f_{i+2}, f_{i+4}$ have a pole at $t=\infty$}

\begin{proposition}
\label{prop:a5inf-0124-1}
{\it 
Suppose that 
for 
$A_5^{(1)}(\alpha_j)_{0\leq  j \leq 5},$ 
there exists a solution such that 
for some $i=0,1,2,3,4,5,$ 
$f_i, f_{i+1}, f_{i+2}, f_{i+4}$ all have a pole at $t=\infty$ and 
$f_{i+3}, f_{i+5}$ 
are both holomorphic at $t=\infty.$ 
$f_i, f_{i+1}, f_{i+2}, f_{i+4}$ then all 
have a pole of order one at $t=\infty.$ 
We denote this case as Type A (3).
}
\end{proposition}

\begin{proof}
By $\pi,$ 
we assume that 
$f_0, f_1, f_2, f_4$ all have a pole of order $n_0, n_1, n_2, n_4 $ at $t=\infty,$ 
where $n_0, n_1, n_2, n_4 $ are all positive integers. 
Since $f_1+f_3+f_5=t,$ it follows that $n_1=1$ and $b_{\infty,1}=1.$
\par
By comparing the coefficients of the highest terms in
\begin{equation*}
\frac{t}{2}
f_0^{\prime}
=
f_0
\left(
f_1f_2+f_1f_4+f_3 f_4-f_2f_3-f_2f_5-f_4 f_5
\right) 
+
\left(
\frac12
-\alpha_2-\alpha_4
\right)
f_0
+
\alpha_0
\left(
f_2+f_4
\right), 
\end{equation*}
we have $n_2=n_4$ and $c_{\infty, n_2}+e_{\infty, n_4}=0.$ 
\par
By comparing the coefficients of the highest terms in
\begin{equation*}
\frac{t}{2}
f_2^{\prime}
=
f_2
\left(
f_3f_4+f_3f_0+f_5 f_0-f_4f_5-f_4f_1-f_0 f_1
\right) 
+
\left(
\frac12
-\alpha_4-\alpha_0
\right)
f_2
+
\alpha_2
\left(
f_4+f_0
\right),
\end{equation*} 
we get $n_0=n_4$ and $a_{\infty,n_0}+e_{\infty, n_4}=0.$
\par
Therefore, we obtain 
\begin{equation*}
n_0=n_2=n_4=1, \,\,
a_{\infty,1}=c_{\infty,1}=1, \,\,e_{\infty,1}=-1,
\end{equation*}
because $f_0+f_2+f_4=t.$
\end{proof}

In order to compute the residues of the Laurent series, 
we have

\begin{proposition}
\label{prop:a5inf-0124-2}
{\it 
Suppose that 
for 
$A_5^{(1)}(\alpha_j)_{0\leq  j \leq 5},$ 
there exists a solution 
such that 
for some $i=0,1,2,3,4,5,$ 
$f_i, f_{i+1}, f_{i+2}, f_{i+4}$ all have a pole at $t=\infty$ and 
$f_{i+3}, f_{i+5}$ 
are both holomorphic at $t=\infty.$ 
Then, 
\begin{equation*}
\begin{cases}
f_i
=
t
+
\left(
-\alpha_{i+2}-2\alpha_{i+3}-\alpha_{i+4}
\right)
t^{-1}
+ \cdots \\
f_{i+1}
=
t
+
\left(
-\alpha_{i+3}+\alpha_{i+5}
\right)
t^{-1}
+ \cdots \\
f_{i+2}
=
t
+
\left(
\alpha_i+\alpha_{i+4}+2\alpha_{i+5}
\right)
t^{-1}
+ \cdots \\
f_{i+3}
=
\alpha_{i+3} t^{-1} 
+ \cdots \\
f_{i+4}
=
-t
+
\left(
-\alpha_i+\alpha_{i+2}+2\alpha_{i+3}-2\alpha_{i+5}
\right)
t^{-1}
+ \cdots \\
f_{i+5}
=
-\alpha_{i+5} t^{-1}
+ \cdots.
\end{cases}
\end{equation*}
}
\end{proposition}

\begin{proof}
By $\pi,$ 
we assume that $f_0, f_1, f_2, f_4$ all have a pole at $t=\infty.$ 
\par
By comparing the coefficients of the terms $t^2$ and $t$ 
in
\begin{equation*}
\frac{t}{2}
f_3^{\prime}
=
f_3
\left(
f_4f_5+f_4f_1+f_0 f_1-f_5f_0-f_5f_2-f_1 f_2
\right) 
+
\left(
\frac12
-\alpha_5-\alpha_1
\right)
f_3
+
\alpha_3
\left(
f_5+f_1
\right),
\end{equation*}
we have $d_{\infty,0}=0$ and $d_{\infty,-1}=\alpha_3,$ respectively.  
\par
By comparing the coefficients of the terms $t^2$ and $t$ 
in
\begin{equation*}
\frac{t}{2}
f_5^{\prime}
=
f_5
\left(
f_0f_1+f_0f_3+f_2f_3-f_1f_2-f_1f_4-f_3 f_4
\right) 
+
\left(
\frac12
-\alpha_1-\alpha_3
\right)
f_5
+
\alpha_5
\left(
f_1+f_3
\right),
\end{equation*}
we get $f_{\infty,0}=0$ and $f_{\infty,-1}=-\alpha_5,$ 
respectively. 
Therefore, 
since $f_1+f_3+f_5=t,$ it follows that $b_{\infty,0}=0$ and $b_{\infty,-1}=-\alpha_3+\alpha_5.$
\par
By comparing the coefficients of the terms $t^2$ and $t$ 
in
\begin{equation*}
\frac{t}{2}
f_0^{\prime}
=
f_0
\left(
f_1f_2+f_1f_4+f_3 f_4-f_2f_3-f_2f_5-f_4 f_5
\right) 
+
\left(
\frac12
-\alpha_2-\alpha_4
\right)
f_0
+
\alpha_0
\left(
f_2+f_4
\right),
\end{equation*}
we have $c_{\infty,0}+e_{\infty,0}=0$ 
and 
$c_{\infty,-1}+e_{\infty,-1}=\alpha_2+2\alpha_3+\alpha_4,$ 
respectively. 
\par
By comparing the coefficients of the terms $t^2$ and $t$ 
in
\begin{equation*}
\frac{t}{2}
f_2^{\prime}
=
f_2
\left(
f_3f_4+f_3f_0+f_5 f_0-f_4f_5-f_4f_1-f_0 f_1
\right) 
+
\left(
\frac12
-\alpha_4-\alpha_0
\right)
f_2
+
\alpha_2
\left(
f_4+f_0
\right),
\end{equation*} 
we get $a_{\infty,0}+c_{\infty,0}=0$ 
and 
$a_{\infty,-1}+c_{\infty,-1}=-\alpha_0-\alpha_4-2\alpha_5,$ 
respectively. 
\par
Therefore, since $f_0+f_2+f_4=t,$ 
it follows that 
\begin{equation*}
\begin{cases}
a_{\infty,-1}=-\alpha_2-2\alpha_3-\alpha_4, \\
c_{\infty,-1}=\alpha_0+\alpha_4+2\alpha_5, \\
e_{\infty,-1}=-\alpha_0+\alpha_2+2\alpha_3-2\alpha_5.
\end{cases}
\end{equation*}

\end{proof}

In order to show the uniqueness of the Laurent series, 
we have

\begin{proposition}
\label{prop:a5inf-0124-3}
{\it 
Suppose that 
for 
$A_5^{(1)}(\alpha_j)_{0\leq  j \leq 5},$ 
there exists a solution 
such that 
for some $i=0,1,2,3,4,5,$ 
$f_i, f_{i+1}, f_{i+2}, f_{i+4}$ all have a pole at $t=\infty$ and 
$f_{i+3}, f_{i+5}$ 
are both holomorphic at $t=\infty.$ 
It is then unique.
}
\end{proposition}

\begin{proof}
By $\pi,$ 
we assume that $f_0, f_1, f_2, f_4$ all have a pole at $t=\infty.$ 
\par
By comparing the coefficients of the terms $t^{-(k-2)} \,\,(k\geq 2)$ 
in
\begin{equation*}
\frac{t}{2}
f_3^{\prime}
=
f_3
\left(
f_4f_5+f_4f_1+f_0 f_1-f_5f_0-f_5f_2-f_1 f_2
\right) 
+
\left(
\frac12
-\alpha_5-\alpha_1
\right)
f_3
+
\alpha_3
\left(
f_5+f_1
\right),
\end{equation*}
we have 
\begin{align*}
d_{\infty, -k}
&=
\frac{k-2}{2} d_{\infty,-(k-2)}  \\
&-
\sum
(3d_{\infty,-l}f_{\infty,-m}+2d_{\infty,-l}c_{\infty,-m}+d_{\infty,-l}b_{\infty,-m})     \\
&
+2
\sum(d_{\infty,-l}e_{\infty,-m}f_{\infty,-m}-d_{\infty,-l}b_{\infty,-m}c_{\infty,-m})  \\
&+
\left(
\frac12-\alpha_1-\alpha_3-\alpha_5
\right)d_{\infty,-(k-2)},
\end{align*}
where the first sum extends over the positive integers $l,m$ 
for which $l+m=k-3,$ 
and the second sum extends over the positive integers $l,m,n$ 
for which $l+m+n=k-2.$ 
Therefore, it follows that 
$d_{\infty, -k} \,\,(k\geq 2)$ 
are inductively determined. 
\par
By comparing the coefficients of the terms $t^{-(k-2)} \,\,(k\geq 2)$ 
in
\begin{equation*}
\frac{t}{2}
f_5^{\prime}
=
f_5
\left(
f_0f_1+f_0f_3+f_2f_3-f_1f_2-f_1f_4-f_3 f_4
\right) 
+
\left(
\frac12
-\alpha_1-\alpha_3
\right)
f_5
+
\alpha_5
\left(
f_1+f_3
\right),
\end{equation*}
we have 
\begin{align*}
f_{\infty,-k}
&=
-\frac{k-2}{2}f_{\infty, -(k-2)}  \\
&-
\sum(3f_{\infty,-l}d_{\infty, m}+2_{\infty, -l}a_{\infty, -m}+f_{\infty,-l}b_{\infty, -m})  \\
&-2\sum(f_{\infty,-l}a_{\infty, -m}b_{\infty, -n}-f_{\infty,-l}d_{\infty, -m}e_{\infty, -n})  \\
&-\left(
\frac12-\alpha_1-\alpha_3-\alpha_5\right)f_{\infty, -(k-2)},
\end{align*}
which implies that $b_{\infty, -k}, \,\,f_{\infty, -k} \,\,(k\geq 2)$ 
are inductively determined, because $f_1+f_3+f_5=t.$
\par
By comparing the coefficients of the terms $t^{-(k-2)} \,\,(k\geq 2)$ 
in
\begin{equation*}
\frac{t}{2}
f_0^{\prime}
=
f_0
\left(
f_1f_2+f_1f_4+f_3 f_4-f_2f_3-f_2f_5-f_4 f_5
\right) 
+
\left(
\frac12
-\alpha_2-\alpha_4
\right)
f_0
+
\alpha_0
\left(
f_2+f_4
\right),
\end{equation*}
we have 
\begin{align*}
a_{\infty,-k}&=
\frac{k-2}{2}a_{\infty,-(k-2)}-2d_{\infty,-(k-2)}  \\
&+\sum(2a_{\infty,-l}f_{\infty,-m}+2d_{\infty,-l}e_{\infty,-m}+a_{\infty,-l}e_{\infty,-m})  \\
&+2
\sum
(a_{\infty,-l}b_{\infty,-m}c_{\infty,-n}-a_{\infty,-l}e_{\infty,-m}f_{\infty,-n}) \\
&
\left(
\frac12-\alpha_0-\alpha_2-\alpha_4\right)a_{\infty,-(k-2)},
\end{align*}
which implies that $a_{\infty, -k},  \,\,(k\geq 2)$ 
are inductively determined. 
\par
By comparing the coefficients of the terms $t^{-(k-2)} \,\,(k\geq 2)$ 
in
\begin{equation*}
\frac{t}{2}
f_2^{\prime}
=
f_2
\left(
f_3f_4+f_3f_0+f_5 f_0-f_4f_5-f_4f_1-f_0 f_1
\right) 
+
\left(
\frac12
-\alpha_4-\alpha_0
\right)
f_2
+
\alpha_2
\left(
f_4+f_0
\right),
\end{equation*} 
we obtain 
\begin{align*}
c_{\infty,-k}&=
-\frac{k-2}{2}c_{\infty,-(k-2)}-2f_{\infty,-(k-2)}  \\
&+\sum(2c_{\infty,-l}d_{\infty,-m}+2f_{\infty,-l}e_{\infty,-m}+c_{\infty,-l}c_{\infty,-m})  \\
&-2
\sum
(c_{\infty,-l}d_{\infty,-m}e_{\infty,-n}-c_{\infty,-l}a_{\infty,-m}b_{\infty,-n}) \\
&-
\left(
\frac12-\alpha_0-\alpha_2-\alpha_4\right)c_{\infty,-(k-2)},
\end{align*}
which implies that $c_{\infty, -k},  \,\,(k\geq 2)$ 
are inductively determined.

\end{proof}

\subsubsection{$f_i, f_{i+2}, f_{i+3}, f_{i+5}$ have a pole at $t=\infty$}

\begin{proposition}
\label{prop:a5inf-0235}
{\it 
For $A_5^{(1)}(\alpha_j)_{0\leq  j \leq 5},$ 
there exists no solution 
such that  
for some $i=0,1,2,3,4,5,$ 
$f_i, f_{i+2}, f_{i+3}, f_{i+5}$ all have a pole at $t=\infty$ and 
$f_{i+4}, f_{i+5}$ 
are both holomorphic at $t=\infty.$ 
}
\end{proposition}

\begin{proof}
By $\pi,$ 
we assume that $f_0, f_2, f_3, f_5$ all have a pole of order $n_0, n_2, n_3, n_5 $ at $t=\infty,$ 
where 
$n_0, n_2, n_2, n_3, n_5$ are all positive integers. 
Since $f_0+f_2+f_4=t$ and $f_1+f_3+f_5=t,$ 
it follows that $n_0=n_2$ and $n_3=n_5,$ respectively. 
We assume that $n_0=n_2\leq n_3=n_5,$ 
because we use $\pi^3$ if $n_0=n_2>n_3=n_5.$
\par 
By comparing the coefficients of the highest terms in 
\begin{equation*}
\frac{t}{2}
f_0^{\prime}
=
f_0
\left(
f_1f_2+f_1f_4+f_3 f_4-f_2f_3-f_2f_5-f_4 f_5
\right) 
+
\left(
\frac12
-\alpha_2-\alpha_4
\right)
f_0
+
\alpha_0
\left(
f_2+f_4
\right),
\end{equation*}
we have $d_{\infty, n_3}+f_{\infty, n_5}=0,$ 
which implies that $n_3=n_5\geq 2.$ 
\par 
By comparing the coefficients of the highest terms in 
\begin{equation*}
\frac{t}{2}
f_3^{\prime}
=
f_3
\left(
f_4f_5+f_4f_1+f_0 f_1-f_5f_0-f_5f_2-f_1 f_2
\right) 
+
\left(
\frac12
-\alpha_5-\alpha_1
\right)
f_3
+
\alpha_3
\left(
f_5+f_1
\right), 
\end{equation*}
we get $a_{\infty, n_0}+c_{\infty, n_2}=0,$ 
which implies that $n_0=n_2\geq 2.$ 
\par
By comparing the coefficients of the terms $t^{2n_0+n_3}, t^{2n_0+n_3-1}, \ldots, t^{2n_0+1}$ in 
\begin{equation*}
\frac{t}{2}
f_0^{\prime}
=
f_0
\left(
f_1f_2+f_1f_4+f_3 f_4-f_2f_3-f_2f_5-f_4 f_5
\right) 
+
\left(
\frac12
-\alpha_2-\alpha_4
\right)
f_0
+
\alpha_0
\left(
f_2+f_4
\right),
\end{equation*}
we have 
\begin{equation*}
d_{\infty, n_3}+f_{\infty, n_5}=
d_{\infty, n_3-1}+f_{\infty, n_5-1}=\cdots
=
d_1+f_1=0,
\end{equation*}
which is impossible because $f_1+f_3+f_5=t.$

\end{proof}

\subsection{Five of $(f_i)_{0\leq i \leq 5}$ have a pole at $t=\infty$}
In this subsection, 
we treat the case in which 
five of $(f_i)_{0\leq i \leq 5}$ have a pole at $t=\infty.$

\begin{proposition}
\label{prop:a5inf-01234}
{\it 
For 
$A_5^{(1)}(\alpha_j)_{0\leq j \leq 5},$ 
there exists no solution 
such that  
for some $i=0,1,2,3,4,5,$ 
$f_i, f_{i+1}, f_{i+2}, f_{i+3},  f_{i+4}$ all have a pole at $t=\infty$ and 
$f_{i+5}$ 
is holomorphic at $t=\infty.$ 
}
\end{proposition}

\begin{proof}
By $\pi,$ 
we assume that $f_0, f_1, f_2, f_3, f_4$ all have a pole of order $n_0, n_1, n_2, n_3, n_4$ at $t=\infty,$ 
where $n_0, n_1, n_2, n_3, n_4$ are all positive integers. 
\par
Since $f_0+f_2+f_4=t,$ 
we have only to consider the following four cases:
\newline
(1)\quad $n_0=n_2>n_4\geq 1,$ 
\newline
(2)\quad $n_2=n_4>n_0\geq 1,$
\newline
(3)\quad $n_4=n_0>n_2\geq 1,$
\newline
(4)\quad $n_0=n_2=n_4\geq 1.$
\par
If case (1) occurs, 
by comparing the coefficients of the highest terms in 
\begin{equation*}
\frac{t}{2}
f_1^{\prime}
=
f_1
\left(
f_2f_3+f_2f_5+f_4 f_5-f_3f_4-f_3f_0-f_5 f_0
\right) 
+
\left(
\frac12
-\alpha_3-\alpha_5
\right)
f_1
+
\alpha_1
\left(
f_3+f_5
\right), 
\end{equation*}
we have $a_{\infty, n_0}-c_{\infty, n_2}=0.$ 
On the other hand, since $f_0+f_2+f_4=t,$ 
it follows that $a_{\infty, n_0}+c_{\infty, n_2}=0,$ 
which is contradiction.
\par
In the same way, 
we can prove that cases (2) and (3) are impossible. 
Therefore, it follows that $n_0=n_2=n_4.$ 
\par
By comparing the coefficients of the highest terms in 
\begin{equation*}
\frac{t}{2}
f_1^{\prime}
=
f_1
\left(
f_2f_3+f_2f_5+f_4 f_5-f_3f_4-f_3f_0-f_5 f_0
\right) 
+
\left(
\frac12
-\alpha_3-\alpha_5
\right)
f_1
+
\alpha_1
\left(
f_3+f_5
\right), 
\end{equation*}
we have $c_{\infty, n_2}-e_{\infty, n_4}-a_{\infty, n_0}=0.$
\par 
If $n_0=n_2=n_4\geq 2, $ 
it follows that $a_{\infty, n_0}+c_{\infty, n_2}+ e_{\infty, n_4}=0$ 
because $f_0+f_2+f_4=t.$ 
However, these equations imply that $c_{\infty, n_2}=0,$ 
which is impossible. 
Therefore, it follows that $n_0=n_2=n_4=1.$
\par
Since $f_1+f_3+f_5=t,$ 
it follows that $n_1=n_3.$ 
If $n_1=n_3\geq 2,$ we have $b_{\infty, n_1}+d_{\infty, n_3}=0,$ 
because $f_1+f_3+f_5=t.$ 
By comparing the coefficients of the highest terms in 
\begin{equation*}
\frac{t}{2}
f_0^{\prime}
=
f_0
\left(
f_1f_2+f_1f_4+f_3 f_4-f_2f_3-f_2f_5-f_4 f_5
\right) 
+
\left(
\frac12
-\alpha_2-\alpha_4
\right)
f_0
+
\alpha_0
\left(
f_2+f_4
\right),
\end{equation*}
we get $b_{\infty, n_1}-d_{\infty, n_3}=0,$ 
which is impossible. Thus, it follows that $n_1=n_3=1.$ 
Moreover, 
since $f_1+f_3+f_5=t,$ it follows that $b_{\infty,1}+d_{\infty,1}=1.$
\par
By comparing the coefficients of the term $t^3$ in 
\begin{equation*}
\begin{cases}
\frac{t}{2}
f_0^{\prime}
=
f_0
\left(
f_1f_2+f_1f_4+f_3 f_4-f_2f_3-f_2f_5-f_4 f_5
\right) 
+
\left(
\frac12
-\alpha_2-\alpha_4
\right)
f_0
+
\alpha_0
\left(
f_2+f_4
\right),        \\
\frac{t}{2}
f_3^{\prime}
=
f_3
\left(
f_4f_5+f_4f_1+f_0 f_1-f_5f_0-f_5f_2-f_1 f_2
\right) 
+
\left(
\frac12
-\alpha_5-\alpha_1
\right)
f_3
+
\alpha_3
\left(
f_5+f_1
\right), \\
\frac{t}{2}
f_4^{\prime}
=
f_4
\left(
f_5f_0+f_5f_2+f_1f_2-f_0f_1-f_0f_3-f_2 f_3
\right) 
+
\left(
\frac12
-\alpha_0-\alpha_2
\right)
f_4
+
\alpha_4
\left(
f_0+f_2
\right),
\end{cases}
\end{equation*}
we obtain 
\begin{equation}
(*)
\begin{cases}
b_{\infty, 1}c_{\infty, 1}+b_{\infty, 1}e_{\infty, 1}+d_{\infty, 1}e_{\infty, 1}-c_{\infty, 1}d_{\infty, 1}=0 ,  \\
e_{\infty, 1}b_{\infty, 1}+a_{\infty, 1}b_{\infty, 1}-b_{\infty, 1}c_{\infty, 1}=0, \\
b_{\infty, 1}c_{\infty, 1}-a_{\infty, 1}b_{\infty, 1}-a_{\infty, 1}d_{\infty, 1}-c_{\infty, 1}d_{\infty, 1}=0,
\end{cases}
\end{equation}
respectively. 
The first and third equations in $(*)$ 
imply that $a_{\infty, 1}+e_{\infty, 1}=0, $ because $b_{\infty, 1}+d_{\infty, 1}=1.$ 
The second equation in $(*)$ then shows that 
$b_{\infty, 1}c_{\infty, 1}=0,$ 
which is a contradiction.

\end{proof}

\subsection{All of $(f_i)_{0\leq i \leq 5}$ have a pole at $t=\infty$}

\begin{proposition}
\label{prop:a5inf-012345-1}
{\it
Suppose that 
for 
$A_5^{(1)}(\alpha_j)_{0\leq  j \leq 5},$ 
there exists a solution 
such that 
all of $(f_i)_{0\leq i \leq 5}$ have a pole at $t=\infty.$ 
$f_0, f_{1}, f_{2}, f_{3}, f_4, f_5$ then all 
have a pole of order one at $t=\infty.$ 
We denote this case as Type C.
}
\end{proposition}

\begin{proof}
We assume that 
$f_0, f_1, f_2, f_3, f_4, f_5$ 
have a pole of order $n_0, n_1, n_2, n_3, n_4, n_5$ at $t=\infty,$ 
where $n_0, n_1, n_2, n_3, n_4, n_5$ are all positive integers. 
\par
Since $f_0+f_2+f_4=t,$ we consider the following four cases:
\newline
(1)\quad $n_0=n_2>n_4\geq 1,$ 
\newline
(2)\quad $n_2=n_4>n_0\geq 1,$
\newline
(3)\quad $n_4=n_0>n_2\geq 1,$
\newline
(4)\quad $n_0=n_2=n_4\geq 1.$
\par
We suppose that case (1) occurs. 
By comparing the coefficients of the highest term in 
\begin{equation*}
\frac{t}{2}
f_0^{\prime}
=
f_0
\left(
f_1f_2+f_1f_4+f_3 f_4-f_2f_3-f_2f_5-f_4 f_5
\right) 
+
\left(
\frac12
-\alpha_2-\alpha_4
\right)
f_0
+
\alpha_0
\left(
f_2+f_4
\right),      
\end{equation*}
we have $n_1=1$ and $b_{\infty, 1}=\displaystyle \frac12.$ 
Moreover, 
by comparing the coefficients of the highest terms in 
\begin{equation*}
\frac{t}{2}
f_1^{\prime}
=
f_1
\left(
f_2f_3+f_2f_5+f_4 f_5-f_3f_4-f_3f_0-f_5 f_0
\right) 
+
\left(
\frac12
-\alpha_3-\alpha_5
\right)
f_1
+
\alpha_1
\left(
f_3+f_5
\right), 
\end{equation*}
we get $a_{\infty, n_0}-c_{\infty, n_2}=0.$ 
On the other hand, 
we obtain $a_{\infty, n_0}+c_{\infty, n_2}=0,$ 
because $f_0+f_2+f_4=t.$ 
Thus, 
it follows that 
$a_{\infty, n_0}=c_{\infty, n_2}=0,$ 
which is a contradiction.
\par
In the same way, 
we can prove that 
cases (2) and (3) are impossible. Therefore, 
it follows that $n_0=n_2=n_4\geq 1.$ 
In the same way, we can show that 
$n_1=n_3=n_5\geq 1.$
\par
By comparing the coefficients of the highest terms in 
\begin{equation*}
\begin{cases}
\frac{t}{2}
f_0^{\prime}
=
f_0
\left(
f_1f_2+f_1f_4+f_3 f_4-f_2f_3-f_2f_5-f_4 f_5
\right) 
+
\left(
\frac12
-\alpha_2-\alpha_4
\right)
f_0
+
\alpha_0
\left(
f_2+f_4
\right) \\
\frac{t}{2}
f_1^{\prime}
=
f_1
\left(
f_2f_3+f_2f_5+f_4 f_5-f_3f_4-f_3f_0-f_5 f_0
\right) 
+
\left(
\frac12
-\alpha_3-\alpha_5
\right)
f_1
+
\alpha_1
\left(
f_3+f_5
\right) \\
\frac{t}{2}
f_2^{\prime}
=
f_2
\left(
f_3f_4+f_3f_0+f_5 f_0-f_4f_5-f_4f_1-f_0 f_1
\right) 
+
\left(
\frac12
-\alpha_4-\alpha_0
\right)
f_2
+
\alpha_2
\left(
f_4+f_0
\right) \\
\frac{t}{2}
f_3^{\prime}
=
f_3
\left(
f_4f_5+f_4f_1+f_0 f_1-f_5f_0-f_5f_2-f_1 f_2
\right) 
+
\left(
\frac12
-\alpha_5-\alpha_1
\right)
f_3
+
\alpha_3
\left(
f_5+f_1
\right) \\
\frac{t}{2}
f_4^{\prime}
=
f_4
\left(
f_5f_0+f_5f_2+f_1f_2-f_0f_1-f_0f_3-f_2 f_3
\right) 
+
\left(
\frac12
-\alpha_0-\alpha_2
\right)
f_4
+
\alpha_4
\left(
f_0+f_2
\right) \\
\frac{t}{2}
f_5^{\prime}
=
f_5
\left(
f_0f_1+f_0f_3+f_2f_3-f_1f_2-f_1f_4-f_3 f_4
\right) 
+
\left(
\frac12
-\alpha_1-\alpha_3
\right)
f_5
+
\alpha_5
\left(
f_1+f_3
\right),
\end{cases}
\end{equation*} 
we have
\begin{equation*}
(**)
\begin{cases}
b_{\infty, n_1}c_{\infty, n_2}+b_{\infty, n_1}e_{\infty, n_4}+d_{\infty, n_3}e_{\infty, n_4}
-c_{\infty, n_2}d_{\infty, n_3}-c_{\infty, n_2}f_{\infty, n_5}-e_{\infty, n_4}f_{\infty, n_5}=0, \\
c_{\infty, n_2}d_{\infty, n_3}+c_{\infty, n_2}f_{\infty, n_5}+e_{\infty, n_4}f_{\infty, n_5}
-d_{\infty, n_3}e_{\infty, n_4}-d_{\infty, n_3}a_{\infty, n_0}-f_{\infty, n_5}a_{\infty, n_0}=0, \\
d_{\infty, n_3}e_{\infty, n_4}+d_{\infty, n_3}a_{\infty, n_0}+f_{\infty, n_5}a_{\infty, n_0}
-e_{\infty, n_4}f_{\infty, n_5}-e_{\infty, n_4}b_{\infty, n_1}-a_{\infty, n_0}b_{\infty, n_1}=0, \\
e_{\infty, n_4}f_{\infty, n_5}+e_{\infty, n_4}b_{\infty, n_1}+a_{\infty, n_0}b_{\infty, n_1}
-f_{\infty, n_5}a_{\infty, n_0}-f_{\infty, n_5}c_{\infty, n_2}-b_{\infty, n_1}c_{\infty, n_2}=0, \\
f_{\infty, n_5}a_{\infty, n_0}+f_{\infty, n_5}c_{\infty, n_2}+b_{\infty, n_1}c_{\infty, n_2}
-a_{\infty, n_0}b_{\infty, n_1}-a_{\infty, n_0}d_{\infty, n_3}-c_{\infty, n_2}d_{\infty, n_3}=0, \\
a_{\infty, n_0}b_{\infty, n_1}+a_{\infty, n_0}d_{\infty, n_3}+c_{\infty, n_2}d_{\infty, n_3}
-b_{\infty, n_1}c_{\infty, n_2}-b_{\infty, n_1}e_{\infty, n_4}-d_{\infty, n_3}e_{\infty, n_4}=0. \\
\end{cases}
\end{equation*}
From the first and second equations in $(**),$ 
we have 
\begin{equation}
\label{eqn:012345-relation-1}
b_{\infty, n_1}c_{\infty, n_2}
+
b_{\infty, n_1}e_{\infty, n_4}
-
d_{\infty, n_3}a_{\infty, n_0}
-
f_{\infty, n_5}a_{\infty, n_0}
=0.
\end{equation}
\par
We consider the following four cases:
\newline
(1)\quad $n_0=n_2=n_4\geq 2$ and $n_1=n_3=n_5=1,$ 
\newline
(2)\quad $n_1=n_3=n_5\geq 2$ and $n_0=n_2=n_4=1,$
\newline
(3)\quad $n_1=n_3=n_5\geq 2$ and $n_0=n_2=n_4\geq 2,$ 
\newline
(4)\quad $n_0=n_2=n_4=1$ and $n_1=n_3=n_5=1.$ 
\par
If case (1) occurs, we get 
\begin{equation*}
a_{\infty, n_0}+c_{\infty, n_2}+e_{\infty, n_4}=0 \,\,
\mathrm{and} \,\,
b_{\infty, 1}+d_{\infty, 1}+f_{\infty, 1}=1,
\end{equation*}
because $f_0+f_2+f_4=t$ and $f_1+f_3+f_5=t,$ 
respectively. 
From the equation (\ref{eqn:012345-relation-1}), 
it follows that 
\begin{equation*}
-a_{\infty, n_0}
=b_{\infty, 1}(-a_{\infty, n_0})-a_{\infty, n_0}(1-b_{\infty, 1})
=b_{\infty, 1}(c_{\infty, n_2}+e_{\infty, n_4})-a_{\infty, n_0}(d_{\infty, 1}+f_{\infty, 1})
=0,
\end{equation*}
which is impossible. 
In the same way, 
we can prove that case (2) is impossible.
\par
We consider case (3). 
Since $f_0+f_2+f_4=t$ and $f_1+f_3+f_5=t,$ 
it then follows that 
\begin{equation*}
a_{\infty, n_0}+c_{\infty, n_2}+e_{\infty, n_4}=0 \,\, 
\mathrm{and} \,\, 
b_{\infty, n_1}+d_{\infty, n_3}+f_{\infty, n_5}=0,
\end{equation*} 
respectively.
By considering $(**)$ and comparing the coefficients of the terms $t^{2n_0+n_1}$ or $t^{2n_1+n_0}$ in 
\begin{equation*}
\begin{cases}
\frac{t}{2}
f_0^{\prime}
=
f_0
\left(
f_1f_2+f_1f_4+f_3 f_4-f_2f_3-f_2f_5-f_4 f_5
\right) 
+
\left(
\frac12
-\alpha_2-\alpha_4
\right)
f_0
+
\alpha_0
\left(
f_2+f_4
\right) \\
\frac{t}{2}
f_1^{\prime}
=
f_1
\left(
f_2f_3+f_2f_5+f_4 f_5-f_3f_4-f_3f_0-f_5 f_0
\right) 
+
\left(
\frac12
-\alpha_3-\alpha_5
\right)
f_1
+
\alpha_1
\left(
f_3+f_5
\right) \\
\frac{t}{2}
f_2^{\prime}
=
f_2
\left(
f_3f_4+f_3f_0+f_5 f_0-f_4f_5-f_4f_1-f_0 f_1
\right) 
+
\left(
\frac12
-\alpha_4-\alpha_0
\right)
f_2
+
\alpha_2
\left(
f_4+f_0
\right) \\
\frac{t}{2}
f_3^{\prime}
=
f_3
\left(
f_4f_5+f_4f_1+f_0 f_1-f_5f_0-f_5f_2-f_1 f_2
\right) 
+
\left(
\frac12
-\alpha_5-\alpha_1
\right)
f_3
+
\alpha_3
\left(
f_5+f_1
\right) \\
\frac{t}{2}
f_4^{\prime}
=
f_4
\left(
f_5f_0+f_5f_2+f_1f_2-f_0f_1-f_0f_3-f_2 f_3
\right) 
+
\left(
\frac12
-\alpha_0-\alpha_2
\right)
f_4
+
\alpha_4
\left(
f_0+f_2
\right) \\
\frac{t}{2}
f_5^{\prime}
=
f_5
\left(
f_0f_1+f_0f_3+f_2f_3-f_1f_2-f_1f_4-f_3 f_4
\right) 
+
\left(
\frac12
-\alpha_1-\alpha_3
\right)
f_5
+
\alpha_5
\left(
f_1+f_3
\right),
\end{cases}
\end{equation*} 
we have 
\begin{equation*}
(***)
\begin{cases}
a_{\infty, n_0}b_{\infty, n_1-1}
-2b_{\infty, n_1}c_{\infty, n_0-1}
+(c_{\infty, n_0}-e_{\infty, n_0})d_{\infty, n_1-1}
+2f_{\infty, n_1}e_{\infty, n_0-1}
-a_{\infty, n_0}f_{\infty, n_1-1}=0, \\
b_{\infty, n_1}c_{\infty, n_0-1}
-2c_{\infty, n_0}d_{\infty, n_1-1}
+(d_{\infty, n_1}-f_{\infty, n_1})e_{\infty, n_0-1}
+2a_{\infty, n_0}f_{\infty, n_1-1}
-b_{\infty, n_1}a_{\infty, n_0-1}=0, \\
c_{\infty, n_0}d_{\infty, n_1-1}
-2d_{\infty, n_1}e_{\infty, n_0-1}
+(e_{\infty, n_0}-a_{\infty, n_0})f_{\infty, n_1-1}
+2b_{\infty, n_1}a_{\infty, n_0-1}
-c_{\infty, n_0}b_{\infty, n_1-1}=0, \\
d_{\infty, n_1}e_{\infty, n_0-1}
-2e_{\infty, n_0}f_{\infty, n_1-1}
+(f_{\infty, n_1}-b_{\infty, n_1})a_{\infty, n_0-1}
+2c_{\infty, n_0}b_{\infty, n_1-1}
-d_{\infty, n_1}c_{\infty, n_0-1}=0, \\
e_{\infty, n_0}f_{\infty, n_1-1}
-2f_{\infty, n_1}a_{\infty, n_0-1}
+(a_{\infty, n_0}-c_{\infty, n_0})b_{\infty, n_1-1}
+2d_{\infty, n_1}c_{\infty, n_0-1}
-e_{\infty, n_0}d_{\infty, n_1-1}=0, \\
f_{\infty, n_1}a_{\infty, n_0-1}
-2a_{\infty, n_0}b_{\infty, n_1-1}
+(b_{\infty, n_1}-d_{\infty, n_1})c_{\infty, n_0-1}
+2e_{\infty, n_0}d_{\infty, n_1-1}
-f_{\infty, n_1}e_{\infty, n_0-1}=0. \\
\end{cases}
\end{equation*}
Moreover, we define the matrix $A$ and the vector $\mathbf{u}$ by 
\begin{equation*}
A=
\left(
\begin{array}{cccccc}
0                               & a_{\infty, n_0}                 & -2b_{\infty, n_1}                & 
c_{\infty, n_0}-e_{\infty, n_0} & 2f_{\infty, n_1}                & -a_{\infty, n_0}  \\
-b_{\infty, n_1}                & 0                               & b_{\infty, n_1}                   & 
-2c_{\infty, n_0}               & d_{\infty, n_1}-f_{\infty, n_1} & 2a_{\infty, n_0}    \\
2b_{\infty, n_1}                & -c_{\infty, n_0}                & 0                                 & 
c_{\infty, n_0}                 & -2d_{\infty, n_1}               & e_{\infty, n_0}-a_{\infty, n_0}     \\
f_{\infty, n_1}-b_{\infty, n_1} &2c_{\infty, n_2}                 & -d_{\infty, n_1}                  & 
0                               & d_{\infty, n_1}                 & -2e_{\infty, n_0}      \\
-2f_{\infty, n_1}               &a_{\infty, n_0}-c_{\infty, n_0}  &2d_{\infty, n_1}                   &
 -e_{\infty, n_0}               & 0                               & e_{\infty, n_0}                       \\
f_{\infty, n_1}                 &-2a_{\infty, n_0}                &b_{\infty, n_1}-d_{\infty, n_1}    &
2e_{\infty, n_0}                & -f_{\infty, n_1}                & 0                                                     \\
\end{array}
\right),
\end{equation*}
and 
\begin{equation*}
\mathbf{u}=
{}^{t}
(
a_{\infty, n_0-1}, b_{\infty, n_1-1}, c_{\infty, n_0-1},
d_{\infty, n_1-1}, e_{\infty, n_0-1}, f_{\infty, n_1-1}
),
\end{equation*}
respectively. Thus, the system of equations $(***)$ 
is given by 
\begin{equation*}
A \mathbf{u}=0.
\end{equation*}
\par
By the fundamental transformations of $A$ with respect to the rows, 
we have 
\begin{equation*}
A
\longrightarrow
\left(
\begin{array}{cccccc}
0                               & a_{\infty, n_0}                 & 0                                 & 
a_{\infty, n_0}                 & 0                               & a_{\infty, n_0}  \\
-b_{\infty, n_1}                & 0                               & -b_{\infty, n_1}                   & 
0                               & -b_{\infty, n_1}                & 0                         \\
0                               & -c_{\infty, n_0}                & 0                                 & 
-c_{\infty, n_0}                & 0                               & -c_{\infty, n_0}                 \\
0                               & 0                               & b_{\infty, n_1}                  & 
-c_{\infty, n_0}                & -f_{\infty, n_1}                & a_{\infty, n_0}      \\
0                               &0                                & 0                       &
0                               & 0                               & 0                       \\
f_{\infty, n_1}                 &0                                & 0                        &
-c_{\infty, n_0}                & 0                & a_{\infty, n_0}                                                     \\
\end{array}
\right),
\end{equation*}
which implies that 
\begin{equation*}
a_{\infty, n_0-1}+c_{\infty, n_0-1}+e_{\infty, n_0-1}=0 \,\,
\mathrm{and} \,\,
b_{\infty, n_1-1}+d_{\infty, n_1-1}+f_{\infty, n_1-1}=0.
\end{equation*}
By induction, 
we can prove that 
\begin{equation*}
\begin{cases}
a_{\infty, n_0}+c_{\infty, n_0}+e_{\infty, n_0}=
a_{\infty, n_0-1}+c_{\infty, n_0-1}+e_{\infty, n_0-1}=
\cdots=
a_{\infty, 1}+c_{\infty, 1}+e_{\infty, 1}=0, \\
b_{\infty, n_1}+d_{\infty, n_1}+f_{\infty, n_1}=
b_{\infty, n_1-1}+d_{\infty, n_1-1}+f_{\infty, n_1-1}=
\cdots=
b_{\infty, 1}+d_{\infty, 1}+f_{\infty, 1}=0,
\end{cases}
\end{equation*}
which is impossible because 
$f_0+f_2+f_4=t$ and 
$f_1+f_3+f_5=t.$

\end{proof}

In order to compute the residues of the Laurent series, 
we have

\begin{proposition}
\label{prop:a5inf-012345-2}
{\it
Suppose that 
for 
$A_5^{(1)}(\alpha_j)_{0\leq  j \leq 5},$ 
there exists a solution 
such that  
all of $(f_i)_{0\leq i \leq 5}$ have a pole at $t=\infty.$ 
Then, 
\begin{equation*}
\begin{cases}
f_0 
=
\frac13 t 
+
\left(
2\alpha_1+\alpha_2-\alpha_4-2\alpha_5
\right)t^{-1} + \cdots \\
f_1 
=
\frac13 t 
+
\left(
2\alpha_2+\alpha_3-\alpha_5-2\alpha_0
\right)t^{-1} + \cdots \\
f_2 
=
\frac13 t 
+
\left(
2\alpha_3+\alpha_4-\alpha_0-2\alpha_1
\right)t^{-1} + \cdots \\
f_3 
=
\frac13 t 
+
\left(
2\alpha_4+\alpha_5-\alpha_1-2\alpha_2
\right)t^{-1} + \cdots \\
f_4 
=
\frac13 t 
+
\left(
2\alpha_5+\alpha_0-\alpha_2-2\alpha_3
\right)t^{-1} + \cdots \\
f_5
=
\frac13 t 
+
\left(
2\alpha_0+\alpha_1-\alpha_3-2\alpha_4
\right)t^{-1} + \cdots. 
\end{cases}
\end{equation*}
}
\end{proposition}

\begin{proof}
By Proposition \ref{prop:a5inf-012345-2}, 
we set 
\begin{equation*}
\begin{cases}
f_0=a_{\infty, 1}t+a_{\infty, 0}+a_{\infty, -1}t^{-1}+\cdots, \\
f_1=b_{\infty, 1}t+b_{\infty, 0}+b_{\infty, -1}t^{-1}+\cdots, \\
f_2=c_{\infty, 1}t+c_{\infty, 0}+c_{\infty, -1}t^{-1}+\cdots, \\
f_3=c_{\infty, 1}t+c_{\infty, 0}+c_{\infty, -1}t^{-1}+\cdots, \\
f_4=d_{\infty, 1}t+d_{\infty, 0}+d_{\infty, -1}t^{-1}+\cdots, \\
f_5=f_{\infty, 1}t+f_{\infty, 0}+f_{\infty, -1}t^{-1}+\cdots, 
\end{cases}
\end{equation*}
where $a_{\infty, 1}b_{\infty, 1}c_{\infty, 1}d_{\infty, 1}e_{\infty, 1}f_{\infty, 1}\neq 0$ 
and 
$a_{\infty, 1}+c_{\infty, 1}+e_{\infty, 1}=1$ 
and 
$b_{\infty, 1}+d_{\infty, 1}+f_{\infty, 1}=1,$ 
because $f_0+f_2+f_4=t$ and $f_1+f_3+f_5=t.$  
By comparing the coefficients of the term $t^3$ in 
\begin{equation*}
\begin{cases}
\frac{t}{2}
f_0^{\prime}
=
f_0
\left(
f_1f_2+f_1f_4+f_3 f_4-f_2f_3-f_2f_5-f_4 f_5
\right) 
+
\left(
\frac12
-\alpha_2-\alpha_4
\right)
f_0
+
\alpha_0
\left(
f_2+f_4
\right) \\
\frac{t}{2}
f_1^{\prime}
=
f_1
\left(
f_2f_3+f_2f_5+f_4 f_5-f_3f_4-f_3f_0-f_5 f_0
\right) 
+
\left(
\frac12
-\alpha_3-\alpha_5
\right)
f_1
+
\alpha_1
\left(
f_3+f_5
\right) \\
\frac{t}{2}
f_2^{\prime}
=
f_2
\left(
f_3f_4+f_3f_0+f_5 f_0-f_4f_5-f_4f_1-f_0 f_1
\right) 
+
\left(
\frac12
-\alpha_4-\alpha_0
\right)
f_2
+
\alpha_2
\left(
f_4+f_0
\right) \\
\frac{t}{2}
f_3^{\prime}
=
f_3
\left(
f_4f_5+f_4f_1+f_0 f_1-f_5f_0-f_5f_2-f_1 f_2
\right) 
+
\left(
\frac12
-\alpha_5-\alpha_1
\right)
f_3
+
\alpha_3
\left(
f_5+f_1
\right) \\
\frac{t}{2}
f_4^{\prime}
=
f_4
\left(
f_5f_0+f_5f_2+f_1f_2-f_0f_1-f_0f_3-f_2 f_3
\right) 
+
\left(
\frac12
-\alpha_0-\alpha_2
\right)
f_4
+
\alpha_4
\left(
f_0+f_2
\right) \\
\frac{t}{2}
f_5^{\prime}
=
f_5
\left(
f_0f_1+f_0f_3+f_2f_3-f_1f_2-f_1f_4-f_3 f_4
\right) 
+
\left(
\frac12
-\alpha_1-\alpha_3
\right)
f_5
+
\alpha_5
\left(
f_1+f_3
\right),
\end{cases}
\end{equation*} 
we have
\begin{equation*}
(*)
\begin{cases}
b_{\infty, 1}c_{\infty,  1}+b_{\infty,  1}e_{\infty,  1}+d_{\infty,  1}e_{\infty,  1}
-c_{\infty,  1}d_{\infty,  1}-c_{\infty,  1}f_{\infty,  1}-e_{\infty,  1}f_{\infty,  1}=0, \\
c_{\infty,  1}d_{\infty,  1}+c_{\infty,  1}f_{\infty,  1}+e_{\infty, 1 }f_{\infty,  1}
-d_{\infty,  1}e_{\infty,  1}-d_{\infty, 1}a_{\infty, 1}-f_{\infty, 1}a_{\infty, 1}=0, \\
d_{\infty, 1}e_{\infty, 1}+d_{\infty, 1}a_{\infty, 1}+f_{\infty, 1}a_{\infty, 1}
-e_{\infty, 1}f_{\infty, 1}-e_{\infty, 1}b_{\infty, 1}-a_{\infty, 1}b_{\infty, 1}=0, \\
e_{\infty, 1}f_{\infty, 1}+e_{\infty, 1}b_{\infty, 1}+a_{\infty, 1}b_{\infty, 1}
-f_{\infty, 1}a_{\infty, 1}-f_{\infty, 1}c_{\infty, 1}-b_{\infty, 1}c_{\infty, 1}=0, \\
f_{\infty, 1}a_{\infty, 1}+f_{\infty, 1}c_{\infty, 1}+b_{\infty, 1}c_{\infty, 1}
-a_{\infty, 1}b_{\infty, 1}-a_{\infty, 1}d_{\infty, 1}-c_{\infty, 1}d_{\infty, 1}=0, \\
a_{\infty, 1}b_{\infty, 1}+a_{\infty, 1}d_{\infty, 1}+c_{\infty, 1}d_{\infty, 1}
-b_{\infty, 1}c_{\infty, 1}-b_{\infty, 1}e_{\infty, 1}-d_{\infty, 1}e_{\infty, 1}=0. \\
\end{cases}
\end{equation*}
Based on the sums of the first and second equations, the second and third, the third and fourth, the fourth and fifth, 
the fifth and sixth, the sixth and first in $(*),$ 
we have
\begin{equation*}
a_{\infty,1}=
b_{\infty,1}=
c_{\infty,1}=
d_{\infty,1}=
e_{\infty,1}=
f_{\infty,1},
\end{equation*}
which implies that 
\begin{equation*}
a_{\infty,1}=
b_{\infty,1}=
c_{\infty,1}=
d_{\infty,1}=
e_{\infty,1}=
f_{\infty,1}=
\frac13.
\end{equation*}
By comparing the coefficients of the term $t^2$ in 
\begin{equation*}
\begin{cases}
\frac{t}{2}
f_0^{\prime}
=
f_0
\left(
f_1f_2+f_1f_4+f_3 f_4-f_2f_3-f_2f_5-f_4 f_5
\right) 
+
\left(
\frac12
-\alpha_2-\alpha_4
\right)
f_0
+
\alpha_0
\left(
f_2+f_4
\right) \\
\frac{t}{2}
f_1^{\prime}
=
f_1
\left(
f_2f_3+f_2f_5+f_4 f_5-f_3f_4-f_3f_0-f_5 f_0
\right) 
+
\left(
\frac12
-\alpha_3-\alpha_5
\right)
f_1
+
\alpha_1
\left(
f_3+f_5
\right) \\
\frac{t}{2}
f_2^{\prime}
=
f_2
\left(
f_3f_4+f_3f_0+f_5 f_0-f_4f_5-f_4f_1-f_0 f_1
\right) 
+
\left(
\frac12
-\alpha_4-\alpha_0
\right)
f_2
+
\alpha_2
\left(
f_4+f_0
\right) \\
\frac{t}{2}
f_3^{\prime}
=
f_3
\left(
f_4f_5+f_4f_1+f_0 f_1-f_5f_0-f_5f_2-f_1 f_2
\right) 
+
\left(
\frac12
-\alpha_5-\alpha_1
\right)
f_3
+
\alpha_3
\left(
f_5+f_1
\right) \\
\frac{t}{2}
f_4^{\prime}
=
f_4
\left(
f_5f_0+f_5f_2+f_1f_2-f_0f_1-f_0f_3-f_2 f_3
\right) 
+
\left(
\frac12
-\alpha_0-\alpha_2
\right)
f_4
+
\alpha_4
\left(
f_0+f_2
\right) \\
\frac{t}{2}
f_5^{\prime}
=
f_5
\left(
f_0f_1+f_0f_3+f_2f_3-f_1f_2-f_1f_4-f_3 f_4
\right) 
+
\left(
\frac12
-\alpha_1-\alpha_3
\right)
f_5
+
\alpha_5
\left(
f_1+f_3
\right),
\end{cases}
\end{equation*} 
we get 
\begin{equation*}
(\square)
\begin{cases}
2b_{\infty, 0}-c_{\infty, 0}+e_{\infty, 0}-2f_{\infty, 0}=0, \\
2c_{\infty, 0}-d_{\infty, 0}+f_{\infty, 0}-2a_{\infty, 0}=0, \\
2d_{\infty, 0}-e_{\infty, 0}+a_{\infty, 0}-2b_{\infty, 0}=0, \\
2e_{\infty, 0}-f_{\infty, 0}+b_{\infty, 0}-2c_{\infty, 0}=0, \\
2f_{\infty, 0}-a_{\infty, 0}+c_{\infty, 0}-2d_{\infty, 0}=0, \\
2a_{\infty, 0}-b_{\infty, 0}+d_{\infty, 0}-2e_{\infty, 0}=0. 
\end{cases}
\end{equation*}
We define the matrix $B$ and the vector $\mathbf{v}_0$ by 
\begin{equation*}
B=
\left(
\begin{array}{cccccc}
0 &2 &-1 & 0 & 1 &2 \\
-2 & 0 &2 & -1 & 0 &1 \\
1 & -2 & 0 & 2 &-1 & 0  \\
0 & 1 & -2 & 0 & 2 & -1 \\
-1 & 0 & 1 &-2 & 0 &2  \\
2 & -1 & 0 & 1 & -2 & 0
\end{array}
\right),
\end{equation*}
and 
\begin{equation*}
\mathbf{v}_0=
{}^t
(
a_{\infty, 0}, b_{\infty, 0}, c_{\infty, 0}, 
d_{\infty,0},  e_{\infty, 0}, f_{\infty, 0}
),
\end{equation*}
respectively. 
The system of equations, $(\square),$ 
is then expressed by 
\begin{equation*}
B \mathbf{v}_0=0.
\end{equation*}
\par
By the fundamental transformations of $B$ 
with respect to the rows, 
we get 
\begin{equation*}
B
\longrightarrow
\left(
\begin{array}{cccccc}
0 & 0 & -1 & 0 & 1  & 0        \\
0 & 0 & 0  & 1 & 0  & -1     \\
1 & 0 & 0  &0  & -1 & 0     \\
0 & 0 & 0  &0  & 0  & 0  \\
0 & 0 & 0  &0   & 0  & 0  \\
0 & 1 & 0  &-1  & 0  & 0  
\end{array}
\right),
\end{equation*}
which implies that 
\begin{equation*}
a_{\infty, 0}=
c_{\infty, 0}=
e_{\infty, 0} \,\,
\mathrm{and} \,\,
b_{\infty, 0}=
d_{\infty, 0}=
f_{\infty, 0}. 
\end{equation*}
Therefore, we obtain 
\begin{equation*}
a_{\infty, 0}=
c_{\infty, 0}=
e_{\infty, 0}=0 \,\,
\mathrm{and} \,\,
b_{\infty, 0}=
d_{\infty, 0}=
f_{\infty, 0}=0,
\end{equation*}
because $f_0+f_2+f_4=t$ and $f_1+f_3+f_5=t.$
\par
By comparing the coefficients of the term $t$ in 
\begin{equation*}
\begin{cases}
\frac{t}{2}
f_0^{\prime}
=
f_0
\left(
f_1f_2+f_1f_4+f_3 f_4-f_2f_3-f_2f_5-f_4 f_5
\right) 
+
\left(
\frac12
-\alpha_2-\alpha_4
\right)
f_0
+
\alpha_0
\left(
f_2+f_4
\right) \\
\frac{t}{2}
f_1^{\prime}
=
f_1
\left(
f_2f_3+f_2f_5+f_4 f_5-f_3f_4-f_3f_0-f_5 f_0
\right) 
+
\left(
\frac12
-\alpha_3-\alpha_5
\right)
f_1
+
\alpha_1
\left(
f_3+f_5
\right) \\
\frac{t}{2}
f_2^{\prime}
=
f_2
\left(
f_3f_4+f_3f_0+f_5 f_0-f_4f_5-f_4f_1-f_0 f_1
\right) 
+
\left(
\frac12
-\alpha_4-\alpha_0
\right)
f_2
+
\alpha_2
\left(
f_4+f_0
\right) \\
\frac{t}{2}
f_3^{\prime}
=
f_3
\left(
f_4f_5+f_4f_1+f_0 f_1-f_5f_0-f_5f_2-f_1 f_2
\right) 
+
\left(
\frac12
-\alpha_5-\alpha_1
\right)
f_3
+
\alpha_3
\left(
f_5+f_1
\right) \\
\frac{t}{2}
f_4^{\prime}
=
f_4
\left(
f_5f_0+f_5f_2+f_1f_2-f_0f_1-f_0f_3-f_2 f_3
\right) 
+
\left(
\frac12
-\alpha_0-\alpha_2
\right)
f_4
+
\alpha_4
\left(
f_0+f_2
\right) \\
\frac{t}{2}
f_5^{\prime}
=
f_5
\left(
f_0f_1+f_0f_3+f_2f_3-f_1f_2-f_1f_4-f_3 f_4
\right) 
+
\left(
\frac12
-\alpha_1-\alpha_3
\right)
f_5
+
\alpha_5
\left(
f_1+f_3
\right),
\end{cases}
\end{equation*} 
we get 
\begin{equation*}
(\square\square)
\begin{cases}
2b_{\infty, -1}-c_{\infty, -1}+e_{\infty, -1}-2f_{\infty, -1}=3(-2\alpha_0+\alpha_2+\alpha_4), \\
2c_{\infty, -1}-d_{\infty, -1}+f_{\infty, -1}-2a_{\infty, -1}=3(-2\alpha_1+\alpha_3+\alpha_5), \\
2d_{\infty, -1}-e_{\infty, -1}+a_{\infty, -1}-2b_{\infty, -1}=3(-2\alpha_2+\alpha_4+\alpha_0), \\
2e_{\infty, -1}-f_{\infty, -1}+b_{\infty, -1}-2c_{\infty, -1}=3(-2\alpha_3+\alpha_5+\alpha_1), \\
2f_{\infty, -1}-a_{\infty, -1}+c_{\infty, -1}-2d_{\infty, -1}=3(-2\alpha_4+\alpha_0+\alpha_2), \\
2a_{\infty, -1}-b_{\infty, -1}+d_{\infty, -1}-2e_{\infty, -1}=3(-2\alpha_5+\alpha_1+\alpha_3). 
\end{cases}
\end{equation*}
Since $f_0+f_2+f_4=t$ and $f_1+f_3+f_5=t,$ 
from $(\square\square),$ 
we obtain 
\begin{equation*}
\begin{cases}
a_{\infty, -1}=2\alpha_1+\alpha_2-\alpha_4-2\alpha_5, \\
b_{\infty, -1}=2\alpha_2+\alpha_3-\alpha_5-2\alpha_0, \\
c_{\infty, -1}=2\alpha_3+\alpha_4-\alpha_0-2\alpha_1, \\
d_{\infty, -1}=2\alpha_4+\alpha_5-\alpha_1-2\alpha_2, \\
e_{\infty, -1}=2\alpha_5+\alpha_0-\alpha_2-2\alpha_3, \\
f_{\infty, -1}=2\alpha_0+\alpha_1-\alpha_3-2\alpha_4. 
\end{cases}
\end{equation*}

\end{proof}

In order to prove the uniqueness of the Laurent series, 
we have

\begin{proposition}
\label{prop:a5inf-012345-3}
{\it
Suppose that 
for 
$A_5^{(1)}(\alpha_j)_{0\leq  j \leq 5},$ 
there exists a solution $(f_i)_{0\leq i \leq 5}$ 
such that 
all of $(f_i)_{0\leq i \leq 5}$ have a pole at $t=\infty.$ 
It is then unique.
}
\end{proposition}

\begin{proof}
By comparing the coefficients of the term $t^{-(k-2)} \,\,(k\geq 2)$ in 
\begin{equation*}
\begin{cases}
\frac{t}{2}
f_0^{\prime}
=
f_0
\left(
f_1f_2+f_1f_4+f_3 f_4-f_2f_3-f_2f_5-f_4 f_5
\right) 
+
\left(
\frac12
-\alpha_2-\alpha_4
\right)
f_0
+
\alpha_0
\left(
f_2+f_4
\right) \\
\frac{t}{2}
f_1^{\prime}
=
f_1
\left(
f_2f_3+f_2f_5+f_4 f_5-f_3f_4-f_3f_0-f_5 f_0
\right) 
+
\left(
\frac12
-\alpha_3-\alpha_5
\right)
f_1
+
\alpha_1
\left(
f_3+f_5
\right) \\
\frac{t}{2}
f_2^{\prime}
=
f_2
\left(
f_3f_4+f_3f_0+f_5 f_0-f_4f_5-f_4f_1-f_0 f_1
\right) 
+
\left(
\frac12
-\alpha_4-\alpha_0
\right)
f_2
+
\alpha_2
\left(
f_4+f_0
\right) \\
\frac{t}{2}
f_3^{\prime}
=
f_3
\left(
f_4f_5+f_4f_1+f_0 f_1-f_5f_0-f_5f_2-f_1 f_2
\right) 
+
\left(
\frac12
-\alpha_5-\alpha_1
\right)
f_3
+
\alpha_3
\left(
f_5+f_1
\right) \\
\frac{t}{2}
f_4^{\prime}
=
f_4
\left(
f_5f_0+f_5f_2+f_1f_2-f_0f_1-f_0f_3-f_2 f_3
\right) 
+
\left(
\frac12
-\alpha_0-\alpha_2
\right)
f_4
+
\alpha_4
\left(
f_0+f_2
\right) \\
\frac{t}{2}
f_5^{\prime}
=
f_5
\left(
f_0f_1+f_0f_3+f_2f_3-f_1f_2-f_1f_4-f_3 f_4
\right) 
+
\left(
\frac12
-\alpha_1-\alpha_3
\right)
f_5
+
\alpha_5
\left(
f_1+f_3
\right),
\end{cases}
\end{equation*} 
we have 
\begin{align*}
2b_{\infty, -k}-c_{\infty, -k}+e_{\infty, -k}-2f_{\infty, -k}
&=
-\frac92(k-2)a_{\infty,-(k-2)}
-9
\left(
\frac12-\alpha_0-\alpha_2-\alpha_4
\right)
a_{\infty, -(k-2)}  \\
&-3
\sum
(
a_{\infty, -l}b_{\infty, -m}+d_{\infty, -l}e_{\infty, -m}+e_{\infty, -l}a_{\infty, -m}  \\
&\quad\quad
-c_{\infty, -l}d_{\infty, -m}-a_{\infty, -l}c_{\infty, -m}-f_{\infty, -l}a_{\infty, -m}
)  \\
&-2
\sum(a_{\infty,-l}b_{\infty,-m}c_{\infty,-n}-a_{\infty,-l}e_{\infty,-m}f_{\infty,-n}),
\end{align*}
\begin{align*}
2c_{\infty, -k}-d_{\infty, -k}+f_{\infty, -k}-2a_{\infty, -k}
&=
-\frac92(k-2)b_{\infty,-(k-2)}
-9
\left(
\frac12-\alpha_1-\alpha_3-\alpha_5
\right)
b_{\infty, -(k-2)}  \\
&-3
\sum
(
b_{\infty, -l}c_{\infty, -m}+e_{\infty, -l}f_{\infty, -m}+f_{\infty, -l}b_{\infty, -m}  \\
&\quad\quad
-d_{\infty, -l}e_{\infty, -m}-b_{\infty, -l}d_{\infty, -m}-a_{\infty, -l}b_{\infty, -m}
)  \\
&-2
\sum(b_{\infty,-l}c_{\infty,-m}d_{\infty,-n}-b_{\infty,-l}f_{\infty,-m}a_{\infty,-n}),
\end{align*}
\begin{align*}
2d_{\infty, -k}-e_{\infty, -k}+a_{\infty, -k}-2b_{\infty, -k}
&=
-\frac92(k-2)c_{\infty,-(k-2)}
-9
\left(
\frac12-\alpha_0-\alpha_2-\alpha_4
\right)
c_{\infty, -(k-2)}  \\
&-3
\sum
(
c_{\infty, -l}d_{\infty, -m}+f_{\infty, -l}a_{\infty, -m}+a_{\infty, -l}c_{\infty, -m}  \\
&\quad\quad
-e_{\infty, -l}f_{\infty, -m}-c_{\infty, -l}e_{\infty, -m}-b_{\infty, -l}c_{\infty, -m}
)  \\
&-2
\sum(c_{\infty,-l}d_{\infty,-m}e_{\infty,-n}-c_{\infty,-l}a_{\infty,-m}b_{\infty,-n}),
\end{align*}
\begin{align*}
2e_{\infty, -k}-f_{\infty, -k}+b_{\infty, -k}-2c_{\infty, -k}
&=
-\frac92(k-2)d_{\infty,-(k-2)}
-9
\left(
\frac12-\alpha_1-\alpha_3-\alpha_5
\right)
d_{\infty, -(k-2)}  \\
&-3
\sum
(
d_{\infty, -l}e_{\infty, -m}+a_{\infty, -l}b_{\infty, -m}+b_{\infty, -l}d_{\infty, -m}  \\
&\quad\quad
-f_{\infty, -l}a_{\infty, -m}-d_{\infty, -l}f_{\infty, -m}-c_{\infty, -l}d_{\infty, -m}
)  \\
&-2
\sum(d_{\infty,-l}e_{\infty,-m}f_{\infty,-n}-d_{\infty,-l}b_{\infty,-m}c_{\infty,-n}),
\end{align*}
\begin{align*}
2f_{\infty, -k}-a_{\infty, -k}+c_{\infty, -k}-2d_{\infty, -k}
&=
-\frac92(k-2)e_{\infty,-(k-2)}
-9
\left(
\frac12-\alpha_0-\alpha_2-\alpha_4
\right)
e_{\infty, -(k-2)}  \\
&-3
\sum
(
e_{\infty, -l}f_{\infty, -m}+b_{\infty, -l}c_{\infty, -m}+c_{\infty, -l}e_{\infty, -m}  \\
&\quad\quad
-a_{\infty, -l}b_{\infty, -m}-e_{\infty, -l}a_{\infty, -m}-d_{\infty, -l}e_{\infty, -m}
)  \\
&-2
\sum(e_{\infty,-l}f_{\infty,-m}a_{\infty,-n}-e_{\infty,-l}c_{\infty,-m}d_{\infty,-n}),
\end{align*}
\begin{align*}
2a_{\infty, -k}-b_{\infty, -k}+d_{\infty, -k}-2e_{\infty, -k}
&=
-\frac92(k-2)f_{\infty,-(k-2)}
-9
\left(
\frac12-\alpha_1-\alpha_3-\alpha_5
\right)
f_{\infty, -(k-2)}  \\
&-3
\sum
(
f_{\infty, -l}a_{\infty, -m}+c_{\infty, -l}d_{\infty, -m}+d_{\infty, -l}f_{\infty, -m}  \\
&\quad\quad
-b_{\infty, -l}c_{\infty, -m}-f_{\infty, -l}b_{\infty, -m}-e_{\infty, -l}f_{\infty, -m}
)  \\
&-2
\sum(f_{\infty,-l}a_{\infty,-m}b_{\infty,-n}-f_{\infty,-l}d_{\infty,-m}e_{\infty,-n}),
\end{align*}
where 
the first sum extends over the positive integers $l,m$ for which $l+m=k-3,$ 
and 
the second sum extends over the positive integers $l,m,n$ for which $l+m+n=k-2.$  
\par
This system of equations with respect to 
$a_{\infty,-k}, b_{\infty,-k},c_{\infty,-k},d_{\infty,-k},e_{\infty,-k},f_{\infty,-k}$ 
is expressed by 
\begin{equation*}
B \mathbf{v}_{k}=\tilde{\mathbf{v}}_{k-1},
\end{equation*}
where $B$ is defined by 
\begin{equation*}
B=
\left(
\begin{array}{cccccc}
0 &2 &-1 & 0 & 1 &2 \\
-2 & 0 &2 & -1 & 0 &1 \\
1 & -2 & 0 & 2 &-1 & 0  \\
0 & 1 & -2 & 0 & 2 & -1 \\
-1 & 0 & 1 &-2 & 0 &2  \\
2 & -1 & 0 & 1 & -2 & 0
\end{array}
\right),
\end{equation*}
and 
$\mathbf{v}_{k}$ 
is defined by 
\begin{equation*}
\mathbf{v}_{k}=
^{t}
(
a_{\infty,-k}, b_{\infty,-k},c_{\infty,-k},d_{\infty,-k},e_{\infty,-k},f_{\infty,-k}
),
\end{equation*}
and 
all the components of $\tilde{\mathbf{v}}_{k-1}$ 
are expressed by 
\begin{equation*}
a_{\infty, -l}, \,\, b_{\infty, -l}, \,\,c_{\infty, -l}, \,\,
d_{\infty, -l}, \,\, e_{\infty, -l}, \,\, f_{\infty, -l} \,\,(1\leq l \leq k-1).
\end{equation*}
\par
By the fundamental transformations of $B$ 
with respect to the rows, 
we get 
\begin{equation*}
B
\longrightarrow
\left(
\begin{array}{cccccc}
0 & 0 & -1 & 0 & 1  & 0        \\
0 & 0 & 0  & 1 & 0  & -1     \\
1 & 0 & 0  &0  & -1 & 0     \\
0 & 0 & 0  &0  & 0  & 0  \\
0 & 0 & 0  &0   & 0  & 0  \\
0 & 1 & 0  &-1  & 0  & 0  
\end{array}
\right),
\end{equation*}
which implies that 
$
a_{\infty,-k}, \,\,b_{\infty,-k}, \,\, c_{\infty,-k}, \,\,d_{\infty,-k},\,\,e_{\infty,-k},\,\,f_{\infty,-k}\,\,
(k\geq 2) 
$ 
are inductively determined, 
because 
$f_0+f_2+f_4=t$ and $f_1+f_3+f_5=t.$
\end{proof}

\subsection{Summary}
In this subsection, 
we summarize the results of the Laurent series of a meromorphic solution $(f_i)_{0\leq i \leq 5}$ and 
show the basic properties of $(f_i)_{0\leq i \leq 5}.$ 
Furthermore, 
we give examples of the rational solution of $A_5(\alpha_i)_{0\leq i \leq 5}.$

\subsubsection{The Laurent series of  $(f_i)_{0\leq i \leq 5}$ at $t=\infty$}

\begin{proposition}
\label{prop:a5inf}
\it{
Suppose that 
for $A_5^{(1)}(\alpha_j)_{0\leq  j \leq 5},$ 
there exists a meromorphic solution near $t=\infty.$ 
Then, one of Type A (1), Type A (2), Type A (3), Type B and Type C occurs.
\newline
Type \,A \, (1) \quad 
for some $i = 0, 1, 2, 3, 4, 5,$ 
$f_j \,\,(j=0,1,2,3,4,5)$ 
are uniquely determined  
near $t=\infty$  
as:
\begin{equation*}
\begin{cases}
f_i
=
t
-\left(\alpha_{i+2}+\alpha_{i+4}\right)t^{-1}
+
\cdots \\
f_{i+1}
=
t
+\left(\alpha_{i+3}+\alpha_{i+5}\right)t^{-1}
+
\cdots \\
f_{i+2}
=
\alpha_{i+2} t^{-1}
+ \cdots \\
f_{i+3}
=
- \alpha_{i+3} t^{-1}
+ \cdots \\
f_{i+4}
=
\alpha_{i+4} t^{-1}
+ \cdots \\
f_{i+5}
=
- \alpha_{i+5} t^{-1}
+ \cdots;
\end{cases}
\end{equation*}
Type \, A \,(2) 
\quad
for some $i = 0, 1, 2, 3, 4, 5,$ 
$f_j \,\,(j=0,1,2,3,4,5)$ 
are uniquely determined  
near $t=\infty$  
as:
\begin{equation*}
\begin{cases}
f_i
=
t
+\left(\alpha_{i+2}-\alpha_{i+4}\right)t^{-1}
+
\cdots \\
f_{i+1}
=
\alpha_{i+1}t^{-1}
+
\cdots \\
f_{i+2}
=
-\alpha_{i+2} t^{-1}
+ \cdots \\
f_{i+3}
=
t
+\left(\alpha_{i+5}-\alpha_{i+1}\right) t^{-1}
+ \cdots \\
f_{i+4}
=
\alpha_{i+4} t^{-1}
+ \cdots \\
f_{i+5}
=
- \alpha_{i+5} t^{-1}
+ \cdots; 
\end{cases}
\end{equation*}
Type \,A \,(3) 
\quad
for some $i = 0, 1, 2, 3, 4, 5,$ 
$f_j \,\,(j=0,1,2,3,4,5)$ 
are uniquely determined  
near $t=\infty$  
as:
\begin{equation*}
\begin{cases}
f_i
=
t
+
\left(
-\alpha_{i+2}-2\alpha_{i+3}-\alpha_{i+4}
\right)
t^{-1}
+ \cdots \\
f_{i+1}
=
t
+
\left(
-\alpha_{i+3}+\alpha_{i+5}
\right)
t^{-1}
+ \cdots \\
f_{i+2}
=
t
+
\left(
\alpha_i+\alpha_{i+4}+2\alpha_{i+5}
\right)
t^{-1}
+ \cdots \\
f_{i+3}
=
\alpha_{i+3} t^{-1} 
+ \cdots \\
f_{i+4}
=
-t
+
\left(
-\alpha_i+\alpha_{i+2}+2\alpha_{i+3}-2\alpha_{i+5}
\right)
t^{-1}
+ \cdots \\
f_{i+5}
=
-\alpha_{i+5} t^{-1}
+ \cdots;
\end{cases}
\end{equation*}
Type \, B 
\quad 
for some $i = 0, 1, 2, 3, 4, 5,$ 
$f_j \,\,(j=0,1,2,3,4,5)$ 
are uniquely determined  
near $t=\infty$  
as:
\begin{equation*}
\begin{cases}
f_i
=
\frac12 t +
\left(
\alpha_{i+1}-\alpha_{i+3}-2\alpha_{i+4}-\alpha_{i+5}
\right)
t^{-1}
+ \cdots \\
f_{i+1}
=
\frac12 t +
\left(
-\alpha_i+\alpha_{i+2}-\alpha_{i+4}
\right)
t^{-1}
+ \cdots \\
f_{i+2}
=
\frac12 t +
\left(
-\alpha_{i+1}+\alpha_{i+3}+\alpha_{i+5}
\right) t^{-1}
+ \cdots \\
f_{i+3}
=
\frac12 t +
\left(
\alpha_i-\alpha_{i+2}+\alpha_{i+4}+2\alpha_{i+5}
\right)
t^{-1}
+ \cdots \\
f_{i+4}
=
2\alpha_{i+4}t^{-1}
+ \cdots \\
f_{i+5}
=
-2\alpha_{i+5}t^{-1}
+ \cdots;
\end{cases}
\end{equation*}
Type \,C 
\quad
$f_j \,\,(j=0,1,2,3,4,5)$ 
are uniquely determined  
near $t=\infty$  
as:
\begin{equation*}
\begin{cases}
f_0 
=
\frac13 t 
+
\left(
2\alpha_1+\alpha_2-\alpha_4-2\alpha_5
\right)t^{-1} + \cdots \\
f_1 
=
\frac13 t 
+
\left(
2\alpha_2+\alpha_3-\alpha_5-2\alpha_0
\right)t^{-1} + \cdots \\
f_2 
=
\frac13 t 
+
\left(
2\alpha_3+\alpha_4-\alpha_0-2\alpha_1
\right)t^{-1} + \cdots \\
f_3 
=
\frac13 t 
+
\left(
2\alpha_4+\alpha_5-\alpha_1-2\alpha_2
\right)t^{-1} + \cdots \\
f_4 
=
\frac13 t 
+
\left(
2\alpha_5+\alpha_0-\alpha_2-2\alpha_3
\right)t^{-1} + \cdots \\
f_5
=
\frac13 t 
+
\left(
2\alpha_0+\alpha_1-\alpha_3-2\alpha_4
\right)t^{-1} + \cdots. 
\end{cases}
\end{equation*}
}
We also denote Type A (1) by 
$(f_i,f_{i+1})_{\infty}$, 
Type A (2) by 
$(f_i, f_{i+3})_{\infty}$, 
Type A (3) by 
$(f_i,f_{i+1},f_{i+2},f_{i+4})_{\infty}$, 
Type B by 
$(f_i, f_{i+1}, f_{i+2}, f_{i+3}, f_{i+4})_{\infty}$, 
respectively.
\end{proposition}

\subsubsection{Basic properties of meromorphic solutions at $t=\infty$}
By using the uniqueness of the Laurent series at $t=\infty$, 
we show that $f_i \,\,(0\leq i \leq 5)$ are odd functions. 
\begin{proposition}
\label{prop:a5odd}
\it{
Suppose that 
for $A_5^{(1)}(\alpha_i)_{0\leq i \leq 5}$, 
there exists a meromorphic solution near $t=\infty.$ 
$f_i \, (i=0,1,2,3,4,5)$ 
are then odd functions.
}
\end{proposition}

\begin{proof}
$A^{(1)}_5(\alpha_j)_{0\leq j \leq 5}$ 
is invariant under the transformation defined by 
\begin{equation*}
s_{-1}:
t 
\longrightarrow 
-t, \quad
f_j 
\longrightarrow
-f_j 
\quad
(0\leq j \leq 5).
\end{equation*}
Each of 
Type A, Type B and Type C on Proposition \ref{prop:a5inf} 
is also invariant under $s_{-1}$. 
Then 
$ f_j ( t ) = - f_j ( - t ) \,\,(0\leq j \leq 4)$, 
because the Laurent series of $f_j$ at $t=\infty$ for each type is unique. 
Therefore, $ f_j $ are odd functions. 
\end{proof}
Furthermore, based on the uniqueness of the Laurent series at $t=\infty$, we have

\begin{proposition}
\label{prop:uniqueness}
Suppose that 
for $A_5^{(1)}(\alpha_j)_{0\leq  j \leq 5},$ 
there exists a meromorphic solution near $t=\infty.$ 
\newline
\newline
Type \,A \, (1): 
$f_i,f_{i+1}$ have a pole at $t=\infty$ and 
$f_{i+2},f_{i+3},f_{i+4},f_{i+5}$ are all holomorphic at $t=\infty$ 
for some $i = 0, 1, 2, 3, 4, 5.$ 
Then, 
\begin{equation*}
\begin{cases}
f_{i+2} \equiv 0 & {\it if} \,\, \alpha_{i+2}=0 \\
f_{i+3} \equiv 0 & {\it if} \,\, \alpha_{i+3}=0 \\
f_{i+4} \equiv 0 & {\it if} \,\, \alpha_{i+4}=0 \\
f_{i+5} \equiv 0 & {\it if} \,\,  \alpha_{i+5}=0. \\
\end{cases}
\end{equation*}
Type \, A \,(2): 
$f_i, f_{i+3}$ both have a pole at $t=\infty$ 
and 
$f_{i+1},f_{i+2},f_{i+4},f_{i+5}$ 
are all holomorphic at $t=\infty$ 
for some $i = 0, 1, 2, 3, 4, 5.$ 
Then, 
\begin{equation*}
\begin{cases}
f_{i+1} \equiv 0 & {\it if} \,\, \alpha_{i+1}=0 \\
f_{i+2} \equiv 0 & {\it if} \,\, \alpha_{i+2}=0 \\
f_{i+4} \equiv 0 & {\it if} \,\, \alpha_{i+4}=0 \\
f_{i+5} \equiv 0 & {\it if} \,\,  \alpha_{i+5}=0. \\
\end{cases}
\end{equation*}
Type \,A \,(3): 
$f_i,f_{i+1},f_{i+2},f_{i+4}$ all have a pole at $t=\infty$ 
and 
$f_{i+3},f_{i+5}$ 
are both holomorphic at $t=\infty$ 
for some $i = 0, 1, 2, 3, 4, 5.$ 
Then, 
\begin{equation*}
\begin{cases}
f_{i+3} \equiv 0 & {\it if} \,\, \alpha_{i+3}=0 \\
f_{i+5} \equiv 0 & {\it if} \,\,  \alpha_{i+5}=0. \\
\end{cases}
\end{equation*}
Type \, B: 
$f_i, f_{i+1}, f_{i+2},f_{i+3}$ all have a pole at $t=\infty$ 
and $f_{i+4},f_{i+5}$ 
are both holomorphic at $t=\infty$ 
for some $i = 0, 1, 2, 3, 4, 5.$ 
Then, 
\begin{equation*}
\begin{cases}
f_{i+4} \equiv 0 & {\it if} \,\, \alpha_{i+4}=0 \\
f_{i+5} \equiv 0 & {\it if} \,\,  \alpha_{i+5}=0. \\
\end{cases}
\end{equation*}
\end{proposition}

\subsubsection{Examples of rational solutions}
By considering 
Proposition \ref{prop:a5inf}, 
we give examples of a rational solution of $A_5^{(1)}(\alpha_i)_{0\leq i \leq 5}$.

\begin{proposition}
\label{prop:tri}
\it{
Type \, A \,(1) 
\quad
for some $i=0, 1, 2, 3, 4, 5,$ 
\begin{equation*}
\begin{cases}
f_i=f_{i+1}=t, \,\, 
f_{i+2}=f_{i+3}=f_{i+4}=f_{i+5} \equiv 0 \\
\alpha_i+\alpha_{i+1}=1, \,\, 
\alpha_{i+2}=\alpha_{i+3}=\alpha_{i+4}=\alpha_{i+5}=0;
\end{cases}
\end{equation*}
Type \, A \,(2)
\quad 
for some $i=0, 1, 2, 3, 4, 5,$ 
\begin{equation*}
\begin{cases}
f_i=f_{i+3}=t, \,\,
f_{i+1}=f_{i+2}=f_{i+4}=f_{i+5} \equiv 0  \\
\alpha_i+\alpha_{i+3}=1, \,\,
\alpha_{i+1}=\alpha_{i+2}=\alpha_{i+4}=\alpha_{i+5}=0;
\end{cases}
\end{equation*}
Type \, A\,(3) 
\quad 
for some $i=0, 1, 2, 3, 4, 5,$ 
\begin{equation*}
\begin{cases}
f_i=f_{i+1}=f_{i+2}=t, \, f_{i+3} \equiv 0, \, 
f_{i+4}=-t, \, f_{i+5} \equiv 0  \\
\alpha_{i}=\alpha_{i+2}, \,
\alpha_{i}+\alpha_{i+4}=0, \,
\alpha_{i+3}=\alpha_{i+5}=0;
\end{cases}
\end{equation*}
Type \, B 
\quad
for some $i=0, 1, 2, 3, 4, 5,$ 
\begin{equation*}
\begin{cases}
f_{i}=f_{i+1}=f_{i+2}=f_{i+3}=\frac12 t, \,
f_{i+4}=f_{i+5} \equiv 0, \\
\alpha_{i}=\alpha_{i+2}, \, \alpha_{i+1}=\alpha_{i+3}, \,
\alpha_{i+4}=\alpha_{i+5}=0;
\end{cases}
\end{equation*}
Type \, C 
\begin{equation*}
\begin{cases}
f_0=f_1=f_2=f_3=f_4=f_5=\frac13 t, \\
\alpha_0=\alpha_2=\alpha_4, \,
\alpha_1=\alpha_3=\alpha_5.
\end{cases}
\end{equation*}
}
\end{proposition}

\begin{proof}
It can be proved by direct calculation.
\end{proof}

\section{Meromorphic Solutions at $t=0$}
In this section, 
we prove Proposition \ref{prop:a5zero}, where 
we calculate the Laurent series of $f_i \,\,(i =0,1,2,3,4,5)$ at $t=0,$ 
whose residues are expressed 
by the parameters $\alpha_j \,(j=0,1,2,3,4,5).$ 
Furthermore, in Corollary \ref{coro:reg}, 
we show that 
under some conditions, 
by some B\"acklund transformations, 
a meromorphic solution at $t=0$ can be transformed into 
a holomorphic solution at $t=0,$ that is, 
a solution 
such that all of $(f_j)_{0\leq j \leq 5}$ are holomorphic at $t=0.$

\begin{proposition}
\it{
\label{prop:a5zero}
Suppose that for $A_5^{(1)}(\alpha_j)_{0\leq j \leq 5},$ 
there exists a meromorphic solution at $t=0.$ 
Also, assume that some of $(f_j)_{0 \leq j \leq 5}$ have a pole at $t=0.$ 
Either of the following then occurs:
\newline
(1)\quad $f_i,f_{i+2}$ both have a pole at $t=0$ and $f_{i+1}, f_{i+3}, f_{i+4}, f_{i+5} $ are all holomorphic at $t=0$ 
for some $i=0, 1, 2, 3, 4, 5,$
\newline
(2)\quad $f_i,f_{i+2},f_{i+3},f_{i+5}$ all have a pole at $t=0$ and $f_{i+1}, f_{i+4}$ are both holomorphic at $t=0$ for some $i=0,1,2,3,4,5.$
\par
If $f_i,f_{i+2}$ both have a pole at $t=0$ for some $i=0, 1, 2, 3, 4, 5,$ 
\begin{equation*}
\begin{cases}
f_i
=
\left(
\alpha_{i+1}-\alpha_{i+3}-\alpha_{i+5}
\right)
t^{-1} + \cdots  \\
f_{i+1}
=\displaystyle
\frac
{
\alpha_{i+1}
}
{
\alpha_{i+1}-\alpha_{i+3}-\alpha_{i+5}
}
t + \cdots \\
f_{i+2}
=
\left(
-\alpha_{i+1}+\alpha_{i+3}+\alpha_{i+5}
\right)
t^{-1} + \cdots \\
f_{i+3}
=\displaystyle
\frac{-\alpha_{i+3}}
{
\alpha_{i+1}-\alpha_{i+3}-\alpha_{i+5}
}
t + \cdots \\
f_{i+4}
=
e_{0,i+4} t + \cdots \\
f_{i+5}
=\displaystyle
\frac{-\alpha_{i+5}}
{\alpha_{i+1}-\alpha_{i+3}-\alpha_{i+5}}
t + \cdots,
\end{cases} 
\end{equation*}
where 
\begin{equation*}
e_{0,i+4} 
=
\begin{cases}
\displaystyle \frac{\alpha_4}{1+\alpha_1-\alpha_3-\alpha_5} \quad {\it if} \, \alpha_1-\alpha_3-\alpha_5\neq -1, \\
{\it arbitrary \,\,constant \,\,and} \,\, \alpha_4=0 \,\,{\it if} \,\alpha_1-\alpha_3-\alpha_5=-1.
\end{cases}
\end{equation*}
\par
If $f_i,f_{i+2},f_{i+3},f_{i+5}$ all have a pole at $t=0$ for some $i=0,1,2,3,4,5,$
\begin{equation*}
\begin{cases}
f_{i}
=
\left(
\alpha_{i}-\alpha_{i+3}-2\alpha_{i+4}-\alpha_{i+5}
\right)t^{-1} + \cdots \\
f_{i+1}
=\displaystyle
\frac{\alpha_{i+1}}
{\alpha_{i}-\alpha_{i+3}-2\alpha_{i+4}-\alpha_{i+5}}
t + \cdots \\
f_{i+2}
=
\left(
-\alpha_{i}+\alpha_{i+3}+2\alpha_{i+4}+\alpha_{i+5}
\right)t^{-1} + \cdots \\
f_{i+3}
=
\left(
-\alpha_{i}-2\alpha_{i+1}-\alpha_{i+2}+\alpha_{i+4}
\right)t^{-1} + \cdots \\
f_{i+4}
=\displaystyle
\frac{\alpha_{i+4}}
{-\alpha_{i}-2\alpha_{i+1}-\alpha_{i+2}+\alpha_{i+4}}
t + \cdots \\
f_{i+5}
=
\left(
\alpha_{i}+2\alpha_{i+1}+\alpha_{i+2}-\alpha_{i+4}
\right)t^{-1} + \cdots. 
\end{cases}
\end{equation*}
We denote case (1) by $(f_i,f_{i+2})_0$ and 
case (2) by $(f_i,f_{i+2},f_{i+3},f_{i+5})_0$, respectively.
}
\end{proposition}

\begin{proof}
It can be proved by direct calculation.
\end{proof}

By Proposition \ref{prop:a5zero}, 
we can transform a meromorphic solution of $A_5^{(1)}(\alpha_j)_{0\leq j \leq 5}$ 
into a holomorphic solution at $t=0$. 

\begin{corollary}
\label{coro:reg}
\it{
Suppose that 
$(f_i)_{0\leq i \leq 5}$ 
is a meromorphic solution of $A_5^{(1)}(\alpha_i)_{0\leq i \leq 5}$ at $t=0$. 
\newline
(1) 
\quad 
If $f_i,f_{i+2}$ both have a pole at $t=0$  
and $\alpha_{i+1} \neq 0$ for some $i=0, 1, 2, 3, 4, 5,$ 
then 
all of 
$(s_{i+1}(f_j))_{0\leq j \leq 5}$ are 
holomorphic $t=0.$
\newline
(2) 
\quad 
If $f_i,f_{i+2},f_{i+3},f_{i+5}$ all have a pole at $t=0$ 
and 
$\alpha_{i+1} \alpha_{i+4} \neq 0$ 
for some $i=0,1,2,3,4,5,$ then 
all of 
$(s_{i+1} s_{i+4}(f_j))_{0\leq j \leq 5}$ 
are holomorphic at $t=0.$
}
\end{corollary}

\begin{proof}
We deal with case (1).  Case (2) can be proved in the same way. 
By $\pi,$ 
we first assume that $f_0,f_2$ both have a pole at $t=0$ and $\alpha_1\neq 0.$ 
We then set 
\begin{equation*}
\begin{cases}
f_0=a_{0,-1}t^{-1}+a_{0,0}+a_{0,1}t+\cdots, \\
f_1=b_{0,1}t+\cdots, \\
f_2=c_{0,-1}t^{-1}+c_{0,0}+c_{0,1}t+\cdots, \\
f_3=d_{0,0}+d_{0,1}t+\cdots, \\
f_4=e_{0,0}+e_{0,1}t+\cdots, \\
f_5=f_{0,0}+f_{0,1}t+\cdots, \\
\end{cases}
\end{equation*}
where 
$$
a_{0,-1}=\alpha_{1}-\alpha_{3}-\alpha_{5}, \,\,
b_{0,1}=\frac
{
\alpha_{1}
}
{
\alpha_{1}-\alpha_{3}-\alpha_{5}
}, \,\,
c_{0,-1}=-(\alpha_{1}-\alpha_{3}-\alpha_{5}).
$$
From the definition of $s_1,$ 
it then follows that 
$
s_1(f_l)=f_l, \,\,(l=1,3,4,5), 
$
and 
\begin{align*}
&s_1(f_0)=f_0-\alpha_1/f_1=(a_{0,-1}t^{-1}+a_{0,0}+a_{0,1}t+\cdots)-
\frac{\alpha_1}{\alpha_1/a_{0,-1}t(1+\cdots)}   \\
 &s_1(f_2)=f_2+\alpha_1/f_1=(-a_{0,-1}t^{-1}+c_{0,0}+c_{0,1}t+\cdots)+
\frac{\alpha_1}{\alpha_1/a_{0,-1}t(1+\cdots)},   \\
\end{align*}
which implies that all of $s_1(f_j)_{0\leq j \leq 5}$ 
are holomorphic at $t=0.$
\end{proof}

\section{Meromorphic Solutions at $t=c\in \mathbb{C}^{*}$}
In this section, 
we prove 
Proposition \ref{prop:a5c}, 
where 
we calculate the Laurent series of $f_j \,\,(j=0,1,2,3,4,5)$ at $t=c \in \mathbb{C}^{*},$ 
whose residues are half integers. 
Furthermore, from the residue theorem, 
we prove Corollary \ref{coro:a5res}, 
which we use 
in order to 
obtain necessary conditions for 
$A_5^{(1)}(\alpha_j)_{0\leq j \leq 5}$ to have rational solutions.

\begin{proposition}
\it{
\label{prop:a5c}
Suppose that 
for $A_5^{(1)}(\alpha_j)_{0\leq  j \leq 5},$ 
there exists a meromorphic solution at $t=c\in\mathbb{C}^{*}.$ 
Moreover, 
assume that some of $(f_j)_{0\leq j \leq 5}$ have a pole at $t=c.$ 
Either of the following then occurs:
\newline
(1)\quad $f_i, f_{i+2}$ both have a pole at $t=c$ and 
$f_{i+1}, f_{i+3}, f_{i+4},f_{i+5}$ are all holomorphic at $t=c$ for some $i=0, 1, 2, 3, 4, 5,$
\newline
(2)\quad $f_i,f_{i+2},f_{i+3},f_{i+5}$ all have a pole at $t=c$ and $f_{i+1},f_{i+4}$ are both holomorphic at $t=c$ for some $i=0,1,2,3,4,5.$
\par
If $f_i, f_{i+2}$ both have a pole at $t=c$ for some $i=0, 1, 2, 3, 4, 5,$ 
one of the following then occurs:
\begin{equation*}
(a)
\begin{cases}
f_{i}
=
\frac12 (t-c)^{-1} 
+
\Bigl\{
\frac{c}{2}
-
\frac{1}{4c}
+
\frac{1}{2c}
\left(
\alpha_{i+1}
-
\alpha_{i+3}
-
\alpha_{i+5}
\right)
\Bigr\}
+
\cdots \\
f_{i+1}
=
c 
+
\left(
1+2\alpha_{i+3}+2\alpha_{i+5}
\right)
(t-c)
+
\cdots \\
f_{i+2}
=
-\frac12 (t-c)^{-1}
+
\Bigl\{
\frac{c}{2}
+
\frac{1}{4c}
-
\frac{1}{2c}
\left(
\alpha_{i+1}
-
\alpha_{i+3}
-
\alpha_{i+5}
\right)
\Bigr\}
+\cdots \\
f_{i+3}
=
-2\alpha_{i+3}(t-c)
+ \cdots \\
f_{i+4}
=
\frac{2\alpha_{i+4}}{3}
(t-c)
+\cdots \\
f_{i+5}
=
-2\alpha_{i+5}
(t-c)
+\cdots,
\end{cases} 
\end{equation*}
\begin{equation*}
(b) 
\begin{cases}
f_{i}
=
-\frac12 (t-c)^{-1} \\
\hspace{10mm}
+
\Bigl\{
\frac{1}{2c}
\left(
\alpha_{i}+\alpha_{i+2}+\alpha_{4}-\frac12
\right)
+\frac{\alpha_{i+1}}{c}
-
\frac{q_{0,i+4}}{2c}
\left(q_{0,i+3}-q_{0,i+5}+c \right)
+\frac{c}{2}
\Bigr\}
+\cdots \\
f_{i+1}
=
-2\alpha_{i+1}
(t-c)+\cdots \\
f_{i+2}
=
\frac12 (t-c)^{-1} \\
\hspace{10mm}
+
\Bigl\{
\frac{-1}{2c}
\left(
\alpha_{i}+\alpha_{i+2}+\alpha_{4}-\frac12
\right)
-\frac{\alpha_{i+1}}{c}
+
\frac{q_{0,i+4}}{2c}
\left(q_{0,i+3}-q_{0,i+5}-c \right)
+\frac{c}{2}
\Bigr\}
+\cdots \\
f_{i+3}
=
q_{0,i+3} 
+
\Bigl\{
q_{0,i+3} q_{0,i+5}
\left(
\frac{4 q_{0,i+4}}{c}
-2
\right)
+
\frac{4 q_{0,i+3}}{c} \alpha_{i+1} \\
\hspace{40mm}
+
\frac{2 q_{0,i+3}}{c}
\left(
\frac12 - \alpha_{i+1}-\alpha_{i+3}-\alpha_{i+5}
\right)
+2\alpha_{i+3}
\Bigr\}(t-c) + \cdots \\
f_{i+4}
=
q_{0,i+4}
+
\Big\{
\frac{2}{c}
q_{0,i+4}
(c-q_{0,i+4})(q_{0,i+5}-q_{0,i+3})  \\
\hspace{40mm}
-\frac{q_{0,i+4}}{c}
\left(
\alpha_{i+1}-\alpha_{i+3}-\alpha_{i+5}
\right)
+2\alpha_{i+4}
\Bigr\}
(t-c)
+\cdots  \\
f_{i+5}
=
q_{0,i+5} 
+
\Bigl\{
q_{0,i+3} q_{0,i+5}
\left(2
-
\frac{4 q_{0,i+4}}{c}
\right)
+
\frac{4 q_{0,i+5}}{c} \alpha_{i+1} \\
\hspace{40mm}
+
\frac{2 q_{0,i+5}}{c}
\left(
\frac12 - \alpha_{i+1}-\alpha_{i+3}-\alpha_{i+5}
\right)
+2\alpha_{i+5}
\Bigr\}(t-c) + \cdots, 
\end{cases}
\end{equation*}
where $q_{0,i+3}, \,\,q_{0,i+4}, \,\,q_{0,i+5}$ are arbitrary constants.
\par
If $f_i,f_{i+2},f_{i+3},f_{i+5}$ all have a pole at $t=c$ for some $i=0,1,2,3,4,5,$ 
one of the following then occurs:
\begin{align*}
&(c) 
\begin{cases}
f_{i}
=
-\frac12 (t-c)^{-1}
+
\Bigl\{
\frac{1}{2c}
\left(
\alpha_{i}+2\alpha_{i+1}+\alpha_{i+2}-\alpha_{i+4}-\frac12
\right)
+\frac{c}{2}
\Bigr\}
+ \cdots \\
f_{i+1}
=
-2\alpha_{i+1} (t-c) +\cdots \\
f_{i+2}
=
\frac12 (t-c)^{-1}
+
\Bigl\{
\frac{-1}{2c}
\left(
\alpha_{i}+2\alpha_{i+1}+\alpha_{i+2}-\alpha_{i+4}-\frac12
\right)
+\frac{c}{2}
\Bigr\}
+ \cdots \\
f_{i+3}
=
-\frac12 (t-c)^{-1} 
+
\Bigl\{
\frac{1}{2c}
\left(
-\alpha_{i+1}+\alpha_{i+3}+2\alpha_{i+4}+\alpha_{i+5}-\frac12
\right)
+\frac{c}{2}
\Bigr\}
+ \cdots \\
f_{i+4}
=
-2\alpha_{i+4} (t-c) +\cdots \\
f_{i+5}
=
\frac12 (t-c)^{-1} 
+
\Bigl\{
\frac{-1}{2c}
\left(
-\alpha_{i+1}+\alpha_{i+3}+2\alpha_{i+4}+\alpha_{i+5}-\frac12
\right)
+\frac{c}{2}
\Bigr\}
+ \cdots,
\end{cases}
\\
&(d) 
\begin{cases}
f_{i}
=
-\frac32 (t-c)^{-1} +O (t-c) \\
f_{i+1}
=
-\frac23 \alpha_{i+1} (t-c) +\cdots \\
f_{i+2}
=
\frac32 (t-c)^{-1} + O(t-c) \\
f_{i+3}
=
\frac12 (t-c)^{-1}
+
\Bigl\{
\frac{c}{2}
-
\frac{1}{2c}
\left(
\alpha_{i+1}-\alpha_{i+3}-2\alpha_{i+4}-\alpha_{i+5}+\frac32
\right)
\Bigr\}
+\cdots \\
f_{i+4}
=
c 
+ 
\left(
1+2\alpha_{i}+4\alpha_{i+1}+2\alpha_{i+2}
\right)
(t-c) + \cdots \\
f_{i+5}
=
-\frac12 (t-c)^{-1}
+
\Bigl\{
\frac{c}{2}
+
\frac{1}{2c}
\left(
\alpha_{i+1}-\alpha_{i+3}-2\alpha_{i+4}-\alpha_{i+5}+\frac32
\right)
\Bigr\}
+\cdots,
\end{cases}
\\
&
(e) 
\begin{cases}
f_{i}
=
\frac12 (t-c)^{-1}
+
\Bigl\{
\frac{c}{2}
+
\frac{-1}{2c}
\left(
-\alpha_{i+1}+\alpha_{i+3}+2\alpha_{i+4}+\alpha_{i+5}
+\frac12
\right)
\Bigr\} + \cdots \\
f_{i+1}
=
c + 
\left(
1+2\alpha_{i+3}+4\alpha_{i+4}+2\alpha_{i+5}
\right)(t-c) +\cdots \\
f_{i+2}
=
-\frac12 (t-c)^{-1}
+
\Bigl\{
\frac{c}{2}
+
\frac{1}{2c}
\left(
-\alpha_{i+1}+\alpha_{i+3}+2\alpha_{i+4}+\alpha_{i+5}
+\frac12
\right)
\Bigr\} +\cdots \\
f_{i+3}
=
-\frac32 (t-c)^{-1} + O(t-c) \\
f_{i+4}
=
-\frac23 \alpha_{i+4} (t-c) +\cdots \\
f_{i+5}
=
\frac32 (t-c)^{-1} + O(t-c). 
\end{cases}
\end{align*}
We denote case (a) by $(f_i,f_{i+2})(I)$, 
case (b) by $(f_i,f_{i+2})(II)$, 
case (c) by $(f_i,f_{i+2},f_{i+3},f_{i+5})(I)$, 
case (d) by $(f_i,f_{i+2},f_{i+3},f_{i+5})(II)$, 
and 
case (e) by $(f_i,f_{i+2},f_{i+3},f_{i+5})(III)$, 
respectively.
}
\end{proposition}

\begin{proof}
It can be proved by direct calculation.
\end{proof}


\begin{corollary}
\label{coro:a5res}
\it{
Suppose that $A_5^{(1)}(\alpha_j)_{0\leq  j \leq 5}$ has a rational solution. 
It then follows that 
$$
-\mathrm{Res}_{t=\infty} f_j -\mathrm{Res}_{t=0} f_j \,\,(j=0,1,2,3,4,5).
$$
}
\end{corollary}

\begin{proof}
If $A_5^{(1)}(\alpha_j)_{0\leq  j \leq 5}$ has a rational solution,  
it follows 
from Corollary \ref{prop:a5odd} that 
$f_i \,\,(i=0,1,2,3,4,5)$ are odd functions. 
Therefore, if $c \in \mathbb{C}^{*}$ is a pole of $f_j$ for some $j=0,1,2,3,4,5,$ 
$-c$ is also a pole of $f_j$ with the same residue. 
\par
Suppose that 
$\pm c_1, \pm c_2, \ldots, \pm c_k \in\mathbb{C}^{*}$ are poles of $(f_j)_{0\leq j \leq 5}.$ 
It then follows from Propositions \ref{prop:a5inf} and \ref{prop:a5zero} that 
\begin{equation*}
\begin{cases}
\displaystyle f_0=a_{\infty,1}t+a_{0,-1}t^{-1}+\sum_{l=1}^k \left(\frac{\epsilon_{0,l}}{t-c_{l}}+\frac{\epsilon_{0,l}}{t+c_{l}} \right), \,\,
\displaystyle f_1=b_{\infty,1}t+b_{0,-1}t^{-1}+\sum_{l=1}^k \left(\frac{\epsilon_{1,l}}{t-c_{l}}+\frac{\epsilon_{1,l}}{t+c_{l}}\right), \\
\displaystyle f_2=c_{\infty,1}t+c_{0,-1}t^{-1}+\sum_{l=1}^k \left(\frac{\epsilon_{2,l}}{t-c_{l}}+\frac{\epsilon_{2,l}}{t+c_{l}}\right), \,\,
\displaystyle f_3=d_{\infty,1}t+d_{0,-1}t^{-1}+\sum_{l=1}^k \left(\frac{\epsilon_{3,l}}{t-c_{l}}+\frac{\epsilon_{3,l}}{t+c_{l}}\right), \\
\displaystyle f_4=e_{\infty,1}t+e_{0,-1}t^{-1}+\sum_{l=1}^k \left(\frac{\epsilon_{4,l}}{t-c_{l}}+\frac{\epsilon_{4,l}}{t+c_{l}}\right), \,\,
\displaystyle f_5=f_{\infty,1}t+f_{0,-1}t^{-1}+\sum_{l=1}^k \left(\frac{\epsilon_{5,l}}{t-c_{l}}+\frac{\epsilon_{5,l}}{t+c_{l}}\right),
\end{cases}
\end{equation*}
where $\epsilon_{i,j} \,\,(0\leq i \leq 5, \,1\leq j \leq k)$ are all half integers. 
Thus, 
by comparing the coefficients of the term $t^{-1}$ of the Laurent series of $f_j \,(0\leq j \leq 5)$ at $t=\infty,$  
we have 
\begin{equation*}
\begin{cases}
\displaystyle a_{\infty, -1}=a_{0,-1}+2\sum_{l=1}^k \epsilon_{0,l}, \quad
\displaystyle b_{\infty, -1}=b_{0,-1}+2\sum_{l=1}^k \epsilon_{1,l} \\
\displaystyle c_{\infty, -1}=c_{0,-1}+2\sum_{l=1}^k \epsilon_{2,l}, \quad
\displaystyle d_{\infty, -1}=d_{0,-1}+2\sum_{l=1}^k \epsilon_{3,l} \\
\displaystyle e_{\infty, -1}=e_{0,-1}+2\sum_{l=1}^k \epsilon_{4,l}, \quad
\displaystyle f_{\infty, -1}=f_{0,-1}+2\sum_{l=1}^k \epsilon_{5,l}, 
\end{cases}
\end{equation*}
which proves the corollary. 

\end{proof}

\section{The Laurent Series of The Auxiliary Function $H$}
In this section, 
we define the auxiliary function $H$ 
and study the properties of $H$. 
This section consists of five subsections. 
In Subsection 4.1, 
following Noumi and Yamada \cite{NoumiYamada-B}, 
we introduce the Hamiltonians $h_{i} \,\,(0\leq i \leq 5)$ for $A_5^{(1)}(\alpha_j)_{0\leq j \leq 5}$ 
and 
define the auxiliary function $H.$ 
In Subsections 4.2, 4.3 and 4.4, 
based on the meromorphic solutions at $t=\infty,$ $t=0$ and $t=c\in\mathbb{C}^{*},$ 
we calculate the Laurent series of $H$ at $t=\infty,$ $t=0$ and $t=c\in\mathbb{C}^{*},$ 
respectively. 
Especially, 
we compute the constant terms of the Laurent series of $H$ at $t=\infty,$ $t=0$ 
and the residue of $H$ at $t=c\in\mathbb{C}^{*}.$
In Subsection 4.5, 
based on the residue calculus of $H,$ 
we obtain a necessary condition for $A_5^{(1)}(\alpha_j)_{0\leq j \leq 5}$ to have 
a rational solution.

\subsection{The definition of the auxiliary function $H$} 
In this subsection, 
following Noumi and Yamada \cite{NoumiYamada-B}, 
we introduce the Hamiltonians $h_i \,\,(i=0,1,2,3,4,5)$ of $A_5^{(1)}(\alpha_j)_{0\leq j \leq 5}$ 
and define the auxiliary function $H$. 
\par
Noumi and Yamada \cite{NoumiYamada-B} defined 
the Hamiltonians $h_i \,\,(i=0,1,2,3,4,5)$ of $A_5^{(1)}(\alpha_j)_{0\leq j \leq 5}$ 
by 
\begin{align*}
h_{i}
&=
\sum_{j=0}^{5}
f_{j}f_{j+1}f_{j+2}f_{j+3} \\
&+\frac13
\left(\alpha_{i+1}+2\alpha_{i+2}+\alpha_{i+4}-\alpha_{i+5}\right)
f_{i}f_{i+1} 
+\frac13
\left(\alpha_{i+1}+2\alpha_{i+2}+3\alpha_{i+3}+\alpha_{i+4}+2\alpha_{i+5}\right)
f_{i+1}f_{i+2} \\
&-\frac13
\left(2\alpha_{i+1}+\alpha_{i+2}-\alpha_{i+4}+\alpha_{i+5}\right)
f_{i+2}f_{i+3}
+\frac13
\left(\alpha_{i+1}-\alpha_{i+2}+\alpha_{i+4}+2\alpha_{i+5}\right)
f_{i+3}f_{i+4}  \\
&-\frac13
\left(2\alpha_{i+1}+\alpha_{i+2}+3\alpha_{i+3}+2\alpha_{i+4}+\alpha_{i+5}\right)
f_{i+4}f_{i+5}
+\frac13
\left(\alpha_{i+1}-\alpha_{i+2}-2\alpha_{i+4}-\alpha_{i+5}\right)
f_{i+5}f_{i} \\
&+
\frac13
\left(\alpha_{i+1}-\alpha_{i+2}+\alpha_{i+4}-\alpha_{i+5}\right)
f_{i}f_{i+3}
+\frac13
\left(\alpha_{i+1}+2\alpha_{i+2}+\alpha_{i+4}+2\alpha_{i+5}\right)
f_{i+1}f_{i+4}  \\
&-\frac13
\left(2\alpha_{i+1}+\alpha_{i+2}+2\alpha_{i+4}+\alpha_{i+5}\right)
f_{i+2}f_{i+5} +\frac14(\alpha_{i+1}+\alpha_{i+3}+\alpha_{i+5})^2.
\end{align*} 
We then define $\tilde{h}_{i}$ and the axillary function $H$ by 
\begin{equation*}
\tilde{h}_{i}=h_{i}-\frac14(\alpha_{i+1}+\alpha_{i+3}+\alpha_{i+5})^2,
\end{equation*}
and 
\begin{equation*}
H=\frac16 
\left(\tilde{h}_{0}+\tilde{h}_{1}+\tilde{h}_{2}+\tilde{h}_{3}+\tilde{h}_{4}+\tilde{h}_{5}
\right),
\end{equation*}
respectively. 
\par
If $A_5^{(1)}(\alpha_i)_{0\leq i \leq 5}$ has a rational solution $(f_i)_{0\leq i \leq 5},$ 
it follows 
from 
Proposition \ref{prop:a5inf} and \ref{prop:a5zero} that 
$H$ can be expanded as 
\begin{equation*}
H 
=
\begin{cases}
h_{\infty, 4}t^4 +
h_{\infty, 2}t^2 +
h_{\infty, 0} + \cdots 
&
\textrm{at \, \, $t=\infty,$} \\
h_{0, -2} t^{-2} 
+ h_{0,0} + \cdots 
&
\textrm{at \,\, $t=0.$}
\end{cases}
\end{equation*}

\subsection{The Laurent series of $H$ at $t=\infty$}
In this subsection, 
we calculate the constant term $h_{\infty,0}$ of $H$ at $t=\infty$ 
by using the meromorphic solutions at $t=\infty$ in Proposition \ref{prop:a5inf}. 
\begin{proposition}
\label{prop:hinf}
\it{
Suppose that 
for $A_5^{(1)}(\alpha_j)_{0\leq j \leq 5},$ 
there exists a meromorphic solution at $t=\infty.$ 
Then, one of Type A (1), Type A (2), Type A (3), Type B and Type C occurs. 
\newline
Type A (1): 
$f_i, f_{i+1}$ both have a pole at $t=\infty$ for some $i=0, 1, 2, 3, 4, 5.$ 
Then, 
\begin{equation*}
h_{\infty,0}
=
-\frac16
(2\alpha_{i+2}+\alpha_{i+3}+\alpha_{i+4}+2\alpha_{i+5})
+\alpha_{i+2}\alpha_{i+3}+\alpha_{i+4}\alpha_{i+5}
+\alpha_{i+2}\alpha_{i+5}.
\end{equation*}
Type A (2): 
$f_{i}, f_{i+3}$ both have a pole at $t=\infty$ for some $i=0, 1, 2, 3, 4, 5.$ 
Then, 
\begin{equation*}
h_{\infty,0}
=
\alpha_{i+1}\alpha_{i+2}+\alpha_{i+4}\alpha_{i+5}
+\frac16
(-\alpha_{i+1}-\alpha_{i+2}-\alpha_{i+4}-\alpha_{i+5}).
\end{equation*}
Type A (3): 
$f_{i}, f_{i+1}, f_{i+2}, f_{i+4}$ all have a pole at $t=\infty$ for some $i=0, 1, 2, 3, 4, 5.$ 
Then, 
\begin{equation*}
h_{\infty, 0}=
\frac16
(
-1+\alpha_{i+1}-3\alpha_{i+3}-\alpha_{i+4}-3\alpha_{i+5}
)+
\alpha_{i+3}(\alpha_i+\alpha_{i+4}+\alpha_{i+5})+
\alpha_{i+5}(\alpha_{i+2}+\alpha_{i+3}+\alpha_{i+4}).
\end{equation*}
\newline
Type B:
$f_i, f_{i+1}, f_{i+2}, f_{i+3}$ 
have a pole at $t=\infty$ for some $i=0, 1, 2, 3, 4, 5.$ 
Then, 
\begin{equation*}
h_{\infty, 0}
=
\frac{1}{12}
(\beta^i_{-1}-\gamma^i_{-1})
+\frac16(\phi^i_{-1}-\epsilon^i_{-1})
+\frac14((\beta^i_{-1})^2+(\gamma^i_{-1})^2)
-\frac12\epsilon^i_{-1}\phi^i_{-1},
\end{equation*}
where 
\begin{equation*}
\beta^i_{-1}
=-\alpha_{i}+\alpha_{i+2}-\alpha_{i+4}, \,\,
\gamma^i_{-1}
=-\alpha_{i+1}+\alpha_{i+3}+\alpha_{i+5}, \,\,
\epsilon^i_{-1}=2\alpha_{i+4}, \,\,
\phi^i_{-1}=-2\alpha_{i+5}.
\end{equation*}
\newline
Type C: 
all of $(f_i)_{0\leq i \leq 5}$ have a pole at $t=\infty.$ 
Then, 
\begin{equation*}
h_{\infty, 0}
=
\frac13
(
2x^2+2y^2+2z^2+2\omega^2
+xy-2yz+z\omega
-2y\omega+x\omega-2xz
),
\end{equation*}
where 
\begin{equation*}
x=\alpha_2-\alpha_4, \,
y=\alpha_3-\alpha_5, \,
z=\alpha_0-\alpha_4, \,
\omega=\alpha_1-\alpha_5.
\end{equation*}
}
\end{proposition}

\subsection{The Laurent series of $H$ at $t=0$}
In this subsection, 
we compute the constant term $h_{0,0}$ of $H$ at $t=0$ 
using the meromorphic solutions at $t=0$ in Proposition \ref{prop:a5zero}.

\begin{proposition}
\label{prop:h0}
\it{
Suppose that for $A_5^{(1)}(\alpha_j)_{0\leq j \leq 5},$ 
there exists a meromorphic solution at $t=0.$ 
\par
(1)\quad If all of $(f_j)_{0\leq j \leq 5}$ are holomorphic at $t=0,$ 
$$
h_{0,0}=0,
$$
because $f_j \,\,(j=0,1,2,3,4,5)$ are all odd functions.
\par
(2)\quad 
If for some $i=0,1,2,3,4,5,$ case $(f_i, f_{i+2})_0$ occurs in Proposition \ref{prop:a5zero}, 
\begin{equation*}
h_{0,0}
=
\frac13 \alpha_{i+1}+
\left(
\frac16-\alpha_{i+1}
\right)
(\alpha_{i+3}+\alpha_{i+5}).
\end{equation*}
\par
(3)\quad
If for some $i=0,1,2,3,4,5,$ case $(f_i, f_{i+2}, f_{i+3}, f_{i+5})_0$ occurs in Proposition \ref{prop:a5zero}, 
\begin{equation*}
h_{0,0}
=
\frac16
-2\alpha_{i+1} \alpha_{i+4}
+
\frac16 (\alpha_{i+1}+\alpha_{i+4}).
\end{equation*}
}
\end{proposition}

{\bf Remark} 
\quad 
In Proposition \ref{prop:a5zero}, 
we find it impossible to 
compute all the coefficients 
of the Laurent series of $(f_j)_{0\leq j \leq 5}$ at $t=0.$ 
However, we can obtain the relations of the coefficients, 
because $f_0+f_2+f_4=t$ and $f_1+f_3+f_5=t.$ 
With them, 
we can then compute $h_{0,0}.$

\subsection{The Laurent series of $H$ at $t=c \in \mathbb{C}^{*}$}
In this subsection, 
we compute the residues of $H$ at $t=c \in \mathbb{C}^{*}$ 
using the Laurent series of 
$f_j \,\,(0\leq j \leq 5)$ at $t=c \in \mathbb{C}^{*}$ in Proposition \ref{prop:a5c}.

\begin{proposition}
\label{prop:hc}
\it{
Suppose that 
for $A_5^{(1)}(\alpha_j)_{0\leq j \leq 5},$ 
there exists a meromorphic solution at $t=c\in\mathbb{C}^{*}.$ 
Moreover, 
assume that some of $(f_j)_{0\leq j \leq 5}$ have a pole at $t=c.$ 
Either of the following then occurs:
\newline
(1)\quad 
for some $i=0, 1, 2, 3, 4, 5,$
\begin{equation*}
\mathrm{Res}_{t=c} H = \begin{cases}
             \frac{1}{6}c  &  \textrm{in case of $(f_{i}, f_{i+2})(I)$ in Proposition \ref{prop:a5c},} \\ 
             \frac{1}{12}c &  \textrm{in case of $(f_{i}, f_{i+2})(II)$ in Proposition \ref{prop:a5c},}
           \end{cases}
\end{equation*}
\newline
(2)\quad 
for some $i=0, 1, 2, 3, 4, 5,$
\begin{equation*}
\mathrm{Res}_{t=c} H = \begin{cases}
              \frac{1}{6}c   &  \textrm{in case of $(f_{i}, f_{i+2}, f_{i+3}, f_{i+5})(I)$ in Proposition \ref{prop:a5c}, }  \\
              \frac{5}{12}c  &  \textrm{in case of $(f_{i}, f_{i+2}, f_{i+3}, f_{i+5})(II)$ in Proposition \ref{prop:a5c},} \\
              \frac{5}{12}c  &  \textrm{in case of $(f_{i}, f_{i+2}, f_{i+3}, f_{i+5})(III)$ in Proposition \ref{prop:a5c}. }
           \end{cases}
\end{equation*}
}
\end{proposition}

{\bf Remark-1} 
\quad 
In Proposition \ref{prop:a5c}, 
we find it impossible to 
compute all the coefficients 
of the Laurent series of $(f_j)_{0\leq j \leq 5}$ at $t=c.$ 
However, we can obtain the relations of the coefficients, 
because $f_0+f_2+f_4=t$ and $f_1+f_3+f_5=t.$ 
With them, 
we can then compute the residues of $H$ at $t=c.$
\newline
\par
{\bf Remark-2}
\quad 
We first defined the auxiliary function $H$ by 
$$
H=\sum_{j=0}^{5}
f_{j}f_{j+1}f_{j+2}f_{j+3}
$$ 
in the same way as we \cite{Matsuda1} did in case of the Noumi and Yamada system of type $A_4^{(1)}.$ 
However, 
the residues of $H$ at $t=c\in\mathbb{C}^{*}$ depend on the parameters $\alpha_j \,(0\leq j \leq 5).$ 
That is the reason why we adopted the complicated definition of  the auxiliary function $H$ 
in this paper.

\subsection{Rational solutions and the auxiliary function $H$}
In this subsection, 
by the residue calculus of $H,$
we obtained the necessary condition for $A_5^{(1)}(\alpha_j)_{0\leq j \leq 5}$ 
to have a rational solution.

\begin{proposition}
\label{prop:hplus}
\it{
(1)\quad 
Suppose that 
for $A_5^{(1)}(\alpha_j)_{0\leq j \leq 5},$ 
there exists a rational solution. 
$6(h_{0,0}-h_{\infty, 0})$ is then a non-positive integer. 
\newline
(2)\quad 
Suppose that 
for $A_5^{(1)}(\alpha_j)_{0\leq j \leq 5},$ 
there exists a rational solution 
such that 
some of 
$(f_i)_{0\leq i \leq 5}$ 
have poles in $\mathbb{C}^{*}.$ 
$
6(h_{0,0}-h_{\infty, 0})
$ 
is then a negative integer. 
\newline
(3)\quad
Suppose that 
for $A_5^{(1)}(\alpha_j)_{0\leq j \leq 5},$ 
there exists a rational solution and 
$h_{0,0}-h_{\infty, 0}=0.$ 
All of $(f_i)_{0\leq i \leq 5}$ 
are then holomorphic in $\mathbb{C}^{*}$.
}
\end{proposition}

\begin{proof}
We first treat case (1) and assume that $\pm c_1, \pm c_2, \ldots, c_n \in \mathbb{C}^{*}$ are poles of $(f_i)_{0\leq i \leq 5}.$ 
By Propositions \ref{prop:hinf}, \ref{prop:h0} and \ref{prop:hc}, 
we then get
\begin{equation*}
H
=
h_{\infty, 4}t^4 +
h_{\infty, 2}t^2 +
h_{\infty, 0} +
h_{0, -2} t^{-2} +
\sum_{k=1}^n 
\left(
\frac{\epsilon_k c_k}{t-c_k}-
\frac{\epsilon_k c_k}{t+c_k}
\right),
\end{equation*}
where $12\epsilon_k \,\,(1\leq k \leq n)$ are all positive integers. 
\par
By comparing the constant terms of 
the Laurent series of $H$ at $t=0$, 
we obtain 
\begin{equation*}
h_{\infty, 0}
-2\sum_{k=1}^n \epsilon_k 
=h_{0,0}.
\end{equation*} 
Therefore, we have
\begin{equation*}
\label{eqn:h-relation}
-h_{\infty, 0}+h_{0,0}
=
-2\sum_{k=1}^n \epsilon_k 
=
-\frac16 m, \cdots (\ast)
\end{equation*}
for a positive integer $m,$ which proves case (1). 
\par
Cases (2) and (3) can be proved by equation $(*)$. 
\end{proof}

\section{Necessary Conditions for Type A}
In this section, we prove Proposition \ref{prop:necA}, 
in which we show the necessary conditions for $A_5^{(1)}(\alpha_i)_{0\leq i \leq 5}$ 
to have a rational solution of Type A 
and transform the rational solution into a holomorphic solution, 
a solution such that 
all of $(f_j)_{0\leq j \leq 5}$ are holomorphic at $t=0.$ 

\begin{proposition}
\label{prop:necA}
\it{
Suppose that 
for 
$A_5^{(1)}(\alpha_j)_{0\leq j \leq 5},$ 
there exists 
a rational solution of Type A. 
By some B\"acklund transformations, 
the solution 
can then be transformed into 
a holomorphic solution at $t=0$. 
Furthermore, 
the parameters 
$\alpha_i \,\,(i=0,1,2,3,4,5)$ 
satisfy one of the following five conditions: 
\newline
(1)\quad for some $i=0, 1, 2, 3, 4, 5,$ 
$
\alpha_{i+2}, \, \alpha_{i+3}, \, 
\alpha_{i+4}, \, \alpha_{i+5} 
\in \mathbb{Z}; 
$
\newline
(2)\quad for some $i=0, 1, 2, 3, 4, 5,$ 
$
\alpha_{i+1}, \, \alpha_{i+2}, \,
\alpha_{i+4}, \, \alpha_{i+5} \,
\in \mathbb{Z};
$
\newline
(3)\quad for some $i=0, 1, 2, 3, 4, 5,$ 
$
\alpha_{i+3}, \, \alpha_{i+5}, \, 
\alpha_i+\alpha_{i+4}, \, 
\alpha_i-\alpha_{i+2} 
\in \mathbb{Z};
$
\newline
(4)\quad for some $i=0, 1, 2, 3, 4, 5,$ 
$
\alpha_{i+3}+\alpha_{i+4}, \, 
\alpha_{i+4}+\alpha_{i+5}, \,
\alpha_i+\alpha_{i+1}, \,
\alpha_i-\alpha_{i+4} 
\in \mathbb{Z};
$
\newline
(5)\quad for some $i=0, 1, 2, 3, 4, 5,$  
$
\alpha_{i}+\alpha_{i+1}, \,
\alpha_{i}+\alpha_{i+5}, \,
\alpha_{i+2}+\alpha_{i+3}, \,
\alpha_{i+3}+\alpha_{i+4}, \,
\alpha_i+\alpha_{i+3} \in \mathbb{Z}.
$
\par
Especially, 
one of cases (1), (2) and (3) occurs if all of $(f_j)_{0\leq j \leq 5}$ are holomorphic at $t=0.$ 
}
\end{proposition}

\begin{proof}
For the proof, 
we have to first consider the following three cases: 
\newline
Type A (1): 
$f_i, f_{i+1}$ both have a pole at $t=\infty$ for some $i=0, 1, 2, 3, 4, 5,$ 
\newline
Type A (2): 
$f_{i}, f_{i+3}$ both have a pole at $t=\infty$ for some $i=0, 1, 2, 3, 4, 5,$ 
\newline
Type A (3): 
$f_{i}, f_{i+1}, f_{i+2}, f_{i+4}$ all have a pole at $t=\infty$ for some $i=0, 1, 2, 3, 4, 5.$ 
\newline
We treat Type A (3). The other cases can be proved in the same way. 
By $\pi,$ 
we may then assume that  $f_0, f_1, f_2, f_4$ all 
have a pole at $t=\infty.$ 
\par
Proposition \ref{prop:a5zero} shows that 
the behaviors of $(f_k)_{0\leq k \leq 5}$ at $t=0$ are one of the following:
\newline
(1)\quad all of $(f_j)_{0\leq j \leq 5}$ are holomorphic at $t=0,$
\newline
(2)\quad $f_j, f_{j+2}$ both have a pole at $t=0$ for some $j=0,1,2,3,4,5,$ 
\newline
(3)\quad $f_j, f_{j+2}, f_{j+3}, f_{j+5}$ all have a pole at $t=0$ for some $j=0,1,2,3,4,5.$ 
\newline 
We deal with case (2) for $j=1,$ that is, 
the case where $f_3, f_5$ both have a pole at $t=0.$ 
Thus, 
we treat the case where $f_0, f_1, f_2, f_4$ all 
have a pole at $t=\infty$ and  $f_3, f_5$ both have a pole at $t=0.$ 
The other cases can be proved in the same way. 
\par
From Corollary \ref{coro:a5res}, 
it follows that 
\begin{equation*}
\mathrm{Res}_{t=\infty} f_i \in \mathbb{Z} \,\, (i=0, 1, 2, 4),  \,\,
-\mathrm{Res}_{t=\infty} f_i - \mathrm{Res}_{t=0} f_i \in \mathbb{Z} \,\, (i=3, 5),
\end{equation*}
which implies that 
\begin{equation*}
\alpha_0+\alpha_1, \,
\alpha_3-\alpha_5, \,
\alpha_1+\alpha_2, \,
\alpha_0+\alpha_2-\alpha_4+\alpha_5 
\in \mathbb{Z}
\end{equation*}
from Propositions \ref{prop:a5inf} and \ref{prop:a5zero}. 
\par
We suppose that $\alpha_4 \neq 0$. From Corollary \ref{coro:reg}, 
it follows that 
$s_4$ 
transforms $(f_i)_{0\leq i \leq 5}$ into a holomorphic solution. 
Therefore, 
$s_4(f_i) \, (i=0, 1, 2, 4)$ 
have a pole at $t=\infty$ and 
all of 
$(s_4(f_j))_{0\leq j \leq 4}$ 
are holomorphic at $t=0$. 
We express this fact by 
\begin{alignat*}{5}
&t=\infty & \,\,  &(f_0, f_1, f_2, f_4)_{\infty} & \,\, &\stackrel{s_4}{\longrightarrow} &  \,\,
&(f_0, f_1, f_2, f_4)_{\infty}   \\
&t=0      &   &(f_3, f_5)_0                  &  &\stackrel{s_4}{\longrightarrow} &      
&\mathrm{holomorphic}.
\end{alignat*} 
We set 
$
\hat{\alpha}_j:=s_4(\alpha_j) 
\,\,
(j=0,1,2,3,4,5).
$ 
Since all of 
$(s_4(f_i))_{0\leq i \leq 5}$ are holomorphic at $t=0$, 
it follows 
from Proposition \ref{prop:a5inf} and Corollary \ref{coro:a5res} that 
\begin{equation*}
\hat{\alpha}_0+\hat{\alpha}_4, 
\hat{\alpha}_0-\hat{\alpha}_2, 
\hat{\alpha}_3, \hat{\alpha}_5 \in \mathbb{Z}.
\end{equation*} 
Therefore, 
we have 
\begin{equation*}
\alpha_3+\alpha_4, 
\alpha_4+\alpha_5, 
\alpha_0+\alpha_1, 
\alpha_0-\alpha_4 \in \mathbb{Z}, 
\alpha_4 \neq 0,
\end{equation*}
which is the condition (4) in the proposition. 
\par
We suppose that $\alpha_4=0$. 
Since $f_3$ has a pole at $t=0$ and $f_3 \not\equiv 0,$ 
it follows from 
Proposition \ref{prop:uniqueness} that 
$\alpha_3 \neq 0,$ which implies that 
\begin{alignat*}{7} 
&t=\infty                        & \,\,
&(f_0, f_1, f_2, f_4)_{\infty}   & \,\,
&\stackrel{s_3}{\longrightarrow} & \,\,
&(f_0, f_1)_{\infty}  \\
&t=0                             & 
&(f_3, f_5)_0                    & 
&\stackrel{s_3}{\longrightarrow} & 
&(f_3, f_5)_0  \\
&                                & 
&(\alpha_j)_{0\leq j \leq 5}     & 
&\stackrel{s_3}{\longrightarrow} & 
&(
\alpha_0, 
\alpha_1, 
\alpha_2+\alpha_3, 
-\alpha_3, 
\alpha_3, 
\alpha_5
).
\end{alignat*}
We set 
$
\hat{\alpha}_j
=s_3(\alpha_j) \,\,(0\leq j \leq 5).
$ 
It follows 
from Corollary \ref{coro:a5res} that 
\begin{equation*}
\hat{\alpha}_2, 
\hat{\alpha}_4, 
\hat{\alpha}_0-\hat{\alpha}_3, 
\hat{\alpha}_0+\hat{\alpha}_5 \in \mathbb{Z}.
\end{equation*} 
Therefore, we obtain $\alpha_j \in \mathbb{Z} \, (j=0,1,2,3,4,5).$ 
By $s_4$, we also have 
\begin{alignat*}{7}
&t=\infty                          & \,\,
&(f_0, f_1)_{\infty}               & \,\,
&\stackrel{s_4}{\longrightarrow}   & \,\,
&(f_5, f_0, f_1, f_3)_{\infty} \\
&t=0                               & 
&(f_3, f_5)_0                      & 
&\stackrel{s_4}{\longrightarrow}   & 
&\mathrm{holomorphic}.
\end{alignat*}
\end{proof}

\section{Necessary Conditions for Type B}
In this section, 
we prove Proposition \ref{prop:necB}, 
which 
shows the necessary conditions for $A_5^{(1)}(\alpha_i)_{0\leq i \leq 5}$ 
to have a rational solution of Type B, 
and transform the solution into a holomorphic solution at $t=0$.

\begin{proposition}
\label{prop:necB}
\it{
Suppose that for $A_5^{(1)}(\alpha_j)_{0\leq j \leq 5},$ 
there exists a rational solution of Type B. 
By some B\"acklund transformations, 
the solution can then be transformed into a holomorphic solution at $t=0$. 
\newline
(1) \,
if 
$f_{i}, f_{i+1}, f_{i+2}, f_{i+3}$ all 
have a pole at $t=\infty$ 
and 
all of $(f_j)_{0\leq j \leq 5}$ 
are holomorphic at $t=0$ for some $i=0,1,2,3,4,5,$  
\begin{equation*}
-\alpha_{i}+\alpha_{i+2}-\alpha_{i+4}, \,
-\alpha_{i+1}+\alpha_{i+3}+\alpha_{i+5}, \,
2\alpha_{i+4}, \, -2\alpha_{i+5} \in \mathbb{Z};
\end{equation*}
(2) \,
if 
$f_{i}, f_{i+1}, f_{i+2}, f_{i+3}$ all 
have a pole at $t=\infty,$ 
$f_{i}, f_{i+2}$ both have a pole at $t=0$ 
and $\alpha_{i+1}$ is not zero for some $i=0,1,2,3,4,5,$ 
\begin{equation*}
-\alpha_i+\alpha_{i+2}-\alpha_{i+4}, \,
\alpha_{i+1}+\alpha_{i+3}+\alpha_{i+5}, \, 
2\alpha_{i+4}, \, -2\alpha_{i+5}, 
\in \mathbb{Z};
\end{equation*}
(3) \,
if 
$f_{i}, f_{i+1}, f_{i+2}, f_{i+3}$ all 
have a pole at $t=\infty,$ 
$f_{i+1}, f_{i+3}$ both 
have a pole at $t=0$ 
and $\alpha_{i+2}$ is not zero for some $i=0,1,2,3,4,5,$  
\begin{equation*}
-\alpha_{i}-\alpha_{i+2}-\alpha_{i+4}, \,
-\alpha_{i+1}+\alpha_{i+3}+\alpha_{i+5}, \,
2\alpha_{i+4}, \, -2\alpha_{i+5}, 
\in \mathbb{Z};
\end{equation*}
(4) \,
if 
$f_{i}, f_{i+1}, f_{i+2}, f_{i+3}$ all 
have a pole at $t=\infty,$ 
$f_{i+2}, f_{i+4}$ both 
have a pole at $t= 0$ 
and $\alpha_{i+3}$ is not zero for some $i=0,1,2,3,4,5,$  
\begin{equation*}
-\alpha_{i}+\alpha_{i+2}-\alpha_{i+4}\in\mathbb{Z}, \, 
-\alpha_{i+1}-\alpha_{i+3}+\alpha_{i+5}\in\mathbb{Z}, \,
2\alpha_{i+3}+2\alpha_{i+4}\in\mathbb{Z}, \,
-2\alpha_{i+5}\in\mathbb{Z}; 
\end{equation*}
(5) \,
if 
$f_{i}, f_{i+1}, f_{i+2}, f_{i+3}$ all 
have a pole at $t=\infty,$ 
$f_{i+3}, f_{i+5}$ both 
have a pole at $t=0$ 
and $\alpha_{i+4}$ is not zero for some $i=0,1,2,3,4,5,$ 
\begin{equation*}
\alpha_{i+1}-\alpha_{i+3}-\alpha_{i+5}, \,
-\alpha_i+\alpha_{i+2}-\alpha_{i+4}, \,
2\alpha_{i+3}+2\alpha_{i+4}, \,
-2\alpha_{i+4} 
\in \mathbb{Z};
\end{equation*}
(6) \,
if 
$f_{i}, f_{i+1}, f_{i+2}, f_{i+3}$ all 
have a pole at $t=\infty,$ 
$f_{i+4}, f_{i}$ both 
have a pole at $t= 0$ 
and $\alpha_{i+5}$ is not zero for some $i=0,1,2,3,4,5,$ 
\begin{equation*}
-\alpha_{i+1}+\alpha_{i+3}+\alpha_{i+5}, \,
\alpha_i-\alpha_{i+2}+\alpha_{i+4}+2\alpha_{i+5}, \, 
-2\alpha_{i+5}, \,
-2\alpha_{i}-2\alpha_{i+5}  
\in \mathbb{Z}; 
\end{equation*}
(7) \,
if 
$f_{i}, f_{i+1}, f_{i+2}, f_{i+3}$ all 
have a pole at $t=\infty,$ 
$f_{i+5}, f_{i+1}$ both 
have a pole at $t=0$ 
and $\alpha_i$ is not zero for some $i=0,1,2,3,4,5,$  
\begin{equation*}
\alpha_{i}+\alpha_{i+2}-\alpha_{i+4}, \, 
-\alpha_{i+1}+\alpha_{i+3}+\alpha_{i+5}, \, 
2\alpha_{i+4}, \,
-2\alpha_{i}-2\alpha_{i+5} \in \mathbb{Z}; 
\end{equation*}
(8) \,
if 
$f_{i}, f_{i+1}, f_{i+2}, f_{i+3}$ all 
have a pole at $t=\infty,$ 
$f_{i}, f_{i+2}, f_{i+3}, f_{i+5}$ all 
have a pole at $t=0$ 
and $\alpha_{i+1},\alpha_{i+4}$ are not zero for some $i=0,1,2,3,4,5,$  
\begin{equation*}
-\alpha_{i+1}-\alpha_{i+3}-2\alpha_{i+4}-\alpha_{i+5}, \,
-\alpha_i+\alpha_{i+2}-\alpha_{i+4}, \,
2\alpha_{i+3}+2\alpha_{i+4}, \,
-2\alpha_{i+4} 
\in \mathbb{Z};
\end{equation*}
(9) \,
if 
$f_{i}, f_{i+1}, f_{i+2}, f_{i+3}$ all 
have a pole at $t=\infty,$ 
$f_{i+1}, f_{i+3}, f_{i+4}, f_{i}$ all 
have a pole at $t=0$ 
and $\alpha_{i+2},\alpha_{i+5}$ are not zero for some $i=0,1,2,3,4,5,$  
\begin{equation*}
-\alpha_i-\alpha_{i+2}-\alpha_{i+4}-2\alpha_{i+5}, \,
-\alpha_{i+1}+\alpha_{i+3}-\alpha_{i+5}, \,
-2\alpha_{i+5}, \,
-2\alpha_i-2\alpha_{i+5}, 
\in \mathbb{Z};
\end{equation*}
(10) \,
if 
$f_{i}, f_{i+1}, f_{i+2}, f_{i+3}$ all 
have a pole at $t=\infty,$ 
$f_{i+2}, f_{i+4}, f_{i+5}, f_{i+1}$ all 
have a pole at $t=0$ 
and $\alpha_{i+3},\alpha_i$ are not zero for some $i=0,1,2,3,4,5,$ 
\begin{equation*}
\alpha_i+\alpha_{i+2}-\alpha_{i+4}, \,
-\alpha_{i+1}-\alpha_{i+3}+\alpha_{i+5}, \,
2\alpha_{i+3}+2\alpha_{i+4}, \,
-2\alpha_{i}-2\alpha_{i+5} 
\in \mathbb{Z}; 
\end{equation*}
(11) \,\, 
if for $A_5^{(1)}(\alpha_j)_{0\leq j \leq 5}$ there exists a rational solution 
and none of cases (1), (2), $\ldots,$ (10)  occurs, 
$$
2\alpha_0, \,\,2\alpha_1, \,\,
2\alpha_2, \,\,2\alpha_3, \,\,
2\alpha_4, \,\,2\alpha_5 \in\mathbb{Z}.
$$
}
\end{proposition}

\begin{proof}
We prove cases (1) and (2). 
The other cases can be proved in the same way. 
\newline
Case (1)
\quad
By $\pi$, we assume that 
$f_0, f_1, f_2, f_3$ all have a pole at $t=\infty$ and 
all of $(f_i)_{0\leq i \leq 5}$ are holomorphic at $t=0$. 
From Proposition \ref{prop:a5inf} and Corollary \ref{coro:a5res} 
it follows that 
\begin{equation}
\label{eqn:breg}
-\alpha_0+\alpha_2-\alpha_4, \, 
-\alpha_1+\alpha_3+\alpha_5, \,
2\alpha_4, \, -2\alpha_5 \in \mathbb{Z}.
\end{equation}
Case (2)
\quad
By $\pi$, we assume that 
$f_0, f_1, f_2, f_3$ all have a pole at $t=\infty$ and 
$f_0, f_2$ both have a pole at $t=0$. 
\par
When $\alpha_1 \neq 0$, 
we obtain 
\begin{align*}
&t=\infty                         & 
&(f_0, f_1, f_2, f_3)_{\infty}    & 
&\stackrel{s_1}{\longrightarrow}  & 
&(f_0, f_1, f_2, f_3)_{\infty}  \\
&t=0                              & 
&(f_0, f_2)_0                     & 
&\stackrel{s_1}{\longrightarrow}  & 
&\mathrm{holomorphic} \\
&                                 & 
&(\alpha_j)_{0\leq j \leq 5}      & 
&\stackrel{s_1}{\longrightarrow}  & 
&(
\alpha_0+\alpha_1, 
-\alpha_1, 
\alpha_2+\alpha_1, 
\alpha_3, 
\alpha_4, 
\alpha_5
).
\end{align*}
We set $\hat{\alpha}_j:=s_1(\alpha_j) \,\,(j=0,1,2,3,4,5).$ 
Equation (\ref{eqn:breg}) implies that 
\begin{equation*}
\hat{\alpha}_1-\hat{\alpha}_3-\hat{\alpha}_5, \,
\hat{\alpha}_0-\hat{\alpha}_2+\hat{\alpha}_4, \,
2\hat{\alpha}_4, \, 2\hat{\alpha}_5 \in \mathbb{Z}.
\end{equation*}
Therefore, we obtain 
\begin{equation*}
-\alpha_0+\alpha_{2}-\alpha_{4}, \,
\alpha_{1}+\alpha_{3}+\alpha_{5}, \,
2\alpha_{4}, \, 2\alpha_{5} 
\in \mathbb{Z}.
\end{equation*}
\par
We suppose that $\alpha_1 =0$. 
We show that $(\alpha_i)_{0\leq i\leq 5}$ is in $\frac12 \mathbb{Z}^6$ 
and $(f_i)_{0\leq i \leq 5}$ can be transformed into a holomorphic solution at $t=0$. 
From Propositions \ref{prop:a5inf}, \ref{prop:a5zero} 
and Corollary \ref{coro:a5res}, 
it follows that 
\begin{equation}
\label{eqn:bcon1}
2\alpha_4, \,
2\alpha_5, \,
\alpha_0-\alpha_2+\alpha_4 \in \mathbb{Z}.
\end{equation}
\par
If $\alpha_1=0$ and $ \alpha_0 \neq 0,$ 
we have 
\begin{align*}
&t=\infty                          & 
&(f_0, f_1, f_2, f_3)_{\infty}     & 
&\stackrel{s_0}{\longrightarrow}   & 
&(f_0, f_1, f_2, f_3)_{\infty}  \\
&t=0                               & 
&(f_0, f_2)_0                      & 
&\stackrel{s_0}{\longrightarrow}   & 
&(f_0, f_2)_0 \\
&                                  & 
&(\alpha_j)_{0\leq j \leq 5}       & 
&\stackrel{s_0}{\longrightarrow}   & 
&
(
-\alpha_0, 
\alpha_0, 
\alpha_2, 
\alpha_3, 
\alpha_4, 
\alpha_5+\alpha_0
),
\end{align*}
and
\begin{align*}
&t=\infty                         & 
&(f_0, f_1, f_2, f_3)_{\infty}    & 
&\stackrel{s_1}{\longrightarrow}  & 
&(f_0, f_1, f_2, f_3)_{\infty}  \\
&t=0                              & 
&(f_0, f_2)_0                     & 
&\stackrel{s_1}{\longrightarrow}  & 
&\mathrm{holomorphic} \\
&                                 & 
&
(
-\alpha_0, 
\alpha_0, 
\alpha_2, 
\alpha_3, 
\alpha_4, 
\alpha_5+\alpha_0
) & 
&\stackrel{s_1}{\longrightarrow}  & 
&
(
0, 
-\alpha_0, 
\alpha_2+\alpha_0, 
\alpha_3, 
\alpha_4, 
\alpha_5+\alpha_0
).
\end{align*}
We set $\tilde{\alpha}_j=s_1  s_0(\alpha_j)\,(j=0,1,2,3,4,5).$ 
$(\tilde{\alpha}_j)_{0\leq j \leq 5}$ also satisfy (\ref{eqn:breg}). 
Therefore, we get
\begin{equation}
\label{eqn:bcon2}
-2\alpha_0-\alpha_3-\alpha_5, \,
-\alpha_0-\alpha_2+\alpha_4, \,
2\alpha_4, \,
2\alpha_5+2\alpha_0 \in  \mathbb{Z}.
\end{equation}
Equations (\ref{eqn:bcon1}) and (\ref{eqn:bcon2}) imply that 
$
2\alpha_j \in \mathbb{Z} \,\,(j=0,1,2,3,4,5).
$
\par
If $\alpha_1=0$ and $\alpha_2 \neq 0,$
we can prove that 
$(\alpha_j)_{0\leq j \leq 5}$ 
is in $\frac12 \mathbb{Z}^6$ 
and $(f_i)_{0\leq i \leq 5}$ can be transformed 
into a holomorphic solution at $t=0$ 
in the same way.
\par
We suppose that $\alpha_0=\alpha_1=\alpha_2=0$. 
Equation (\ref{eqn:bcon1}) implies that 
$
\alpha_4, \, 2\alpha_3, \, 2\alpha_5 \in \mathbb{Z}, 
$ 
which is case (11). 
From Proposition \ref{prop:a5zero}, 
it follows that 
$
\mathrm{Res}_{t=0} f_0
=-\alpha_3-\alpha_5.
$
Since $f_0$ has a pole at $t=0$ with the first order, 
it follows that 
$\alpha_3 \neq 0$ or $\alpha_5 \neq 0$. 
When $\alpha_3\neq 0$ and $\alpha_5 \neq 0,$ 
$s_5  s_3$ transforms 
the rational solution 
into a holomorphic solution at $t=0.$ 
If $\alpha_3 \neq 0$ and $\alpha_5=0,$ 
the solution 
can be transformed 
into a holomorphic solution at $t=0$ 
by 
$s_5  s_4  s_3$. 
If $\alpha_5 \neq 0$ and $\alpha_3=0,$ 
$(f_i)_{0\leq i \leq 5}$ 
can be transformed 
into a holomorphic solution at $t=0$ 
by 
$s_3  s_4  s_5$.
\end{proof}

\section{Necessary Conditions for Type C}
In this section, we prove 
Proposition \ref{prop:necC}, 
which shows the necessary conditions for  
$A_5^{(1)}(\alpha_i)_{0\leq i \leq 5}$ to have 
a rational solution of Type C, 
and transform the solution into a holomorphic solution at $t=0$.

\begin{proposition}
\label{prop:necC}
\it{
Suppose that $A_5^{(1)}(\alpha_j)_{0\leq j \leq 5}$ 
has a rational solution of Type C. 
The solution can then be transformed into a holomorphic solution at $t=0.$ 
\newline
(1) \quad
if all of $(f_i)_{0\leq i \leq 5}$ 
are holomorphic at $t=0$, 
then, 
\begin{equation*}
\alpha_2-\alpha_4
\equiv
\frac{n}{3}, \,
\alpha_3-\alpha_5
\equiv 
\frac{m+n}{3}, \,
\alpha_0-\alpha_4 
\equiv
\frac{2m}{3}, \,
\alpha_1-\alpha_5
\equiv
\frac{n}{3} \,
\mathrm{mod} \,
\mathbb{Z} \, (m,n=0, \pm1);
\end{equation*}
\newline
(2) \quad
if 
$f_k, f_{k+2}$ both have a pole at $t=0$ 
for some $k=0, 1, 2, 3, 4, 5,$ 
then, 
\begin{alignat*}{13}
\alpha_{k+1} &+\alpha_{k+2} & &-\alpha_{k+4} &  &\equiv  \frac{n}{3},   & \quad
&\alpha_{k+3} & &-\alpha_{k+5}    &  &\equiv  \frac{m+n}{3}   & \,\,
&\mathrm{mod} \, \mathbb{Z},     \\
\alpha_{k}   &+\alpha_{k+1} & &-\alpha_{k+4} &  &\equiv  \frac{2m}{3},   &  \quad
-&\alpha_{k+1} & &-\alpha_{k+5}   &  &\equiv  \frac{n}{3}      & 
&\mathrm{mod} \, \mathbb{Z},  
\,\,
&\alpha_{k+1} \neq 0, 
\end{alignat*}
or 
for some $j=0,1,2,3,4,5,$ 
\begin{equation*}
(
\alpha_{j}, \alpha_{j+1},\alpha_{j+2}, \alpha_{j+3}, 
\alpha_{j+4}, \alpha_{j+5}
)
\equiv
\frac{p}{3}
(1,0,1,0,1,0)
+
\frac{q}{3}
(1,0,-1,-1,0,1)\, \mathrm{mod} \, \mathbb{Z} \, 
(p,q=0, \pm1);
\end{equation*}
(3) \quad
if 
$f_k, f_{k+2}, f_{k+3}, f_{k+5}$ all have a pole at $t=0$ 
for some $k=0, 1, 2, 3, 4, 5,$ then, 
\begin{alignat*}{15}
\alpha_{k+1}   &+\alpha_{k+2} & &+\alpha_{k+4} &           &\equiv \frac{n}{3},   & \quad
                &\alpha_{k+3} & &-\alpha_{k+5} &           &\equiv \frac{m+n}{3}  & \,\,
&\mathrm{mod} \, \mathbb{Z},  \\
\alpha_k       &+\alpha_{k+1} & &+\alpha_{k+4} &           &\equiv \frac{2m}{3},  & \quad
-\alpha_{k+1}  -&\alpha_{k+4} & &-\alpha_{k+5} &           &\equiv \frac{n}{3}    & 
&\mathrm{mod} \, \mathbb{Z},\,
\alpha_{k+1}, \alpha_{k+4} \neq 0,
\end{alignat*}
or 
for some $j=0,1,2,3,4,5,$ 
\begin{equation*}
(
\alpha_{j}, \alpha_{j+1},\alpha_{j+2}, 
\alpha_{j+3}, \alpha_{j+4}, \alpha_{j+5}
)
\equiv
\frac{p}{3}
(0, 1, 1, 1,0, 0)+
\frac{q}{3}
(1,1, 0, 0, 0, 1) \,
\mathrm{mod} \, \mathbb{Z}. 
\end{equation*}
}
\end{proposition}

\begin{proof}
We prove cases (1) and (2). Case (3) can be proved in the same way.
\newline
Case (1) \quad
From Proposition \ref{prop:a5inf} and Corollary \ref{coro:a5res}, 
it follows that 
\begin{align}
2\alpha_1+\alpha_2-\alpha_4-2\alpha_5 &= m_0  \label{eqn:m_0}  \\
2\alpha_2+\alpha_3-\alpha_5-2\alpha_0 &= m_1 \label{eqn:m_1}   \\
2\alpha_3+\alpha_4-\alpha_0-2\alpha_1 &= m_2 \label{eqn:m_2}   \\
2\alpha_4+\alpha_5-\alpha_1-2\alpha_2 &= m_3 \label{eqn:m_3}   \\
2\alpha_5+\alpha_0-\alpha_2-2\alpha_3 &= m_4 \label{eqn:m_4}   \\
2\alpha_0+\alpha_1-\alpha_3-2\alpha_4 &= m_5, \label{eqn:m_5} 
\end{align}
where $m_0,m_1,m_2,m_3,m_4,m_5$ are all integers. 
Equations (\ref{eqn:m_0}) and (\ref{eqn:m_3}) imply that 
\begin{equation*}
3\alpha_1-3\alpha_5 \in \mathbb{Z}, 
\quad
3\alpha_2-3\alpha_4 \in  \mathbb{Z}.
\end{equation*}
Equations (\ref{eqn:m_0}), (\ref{eqn:m_1}) and (\ref{eqn:m_2}) imply that 
$$
3\alpha_3-3\alpha_5 \in \mathbb{Z}.
$$
Equations (\ref{eqn:m_1}), (\ref{eqn:m_2}) and (\ref{eqn:m_3}) imply that 
$$
3\alpha_0-3\alpha_4 \in \mathbb{Z}.
$$
We then define 
\begin{equation*}
\alpha_2-\alpha_4 = \frac{n_0}{3}, \, 
\alpha_3-\alpha_5 = \frac{n_1}{3}, \,
\alpha_0-\alpha_4 = \frac{n_2}{3}, \,
\alpha_1-\alpha_5 = \frac{n_3}{3}, 
\end{equation*}
where $n_0, n_1, n_2, n_3$ are all integers. 
By substituting the above equations into 
(\ref{eqn:m_0}), (\ref{eqn:m_1}), (\ref{eqn:m_2}), (\ref{eqn:m_3}), (\ref{eqn:m_4}) and (\ref{eqn:m_5}), 
we get 
\begin{align*}
n_0+2n_3      &\equiv 0 \quad \mathrm{mod} \, 3      \\
2n_0+n_1-2n_2 &\equiv 0 \quad \mathrm{mod} \, 3      \\
2n_1-n_2-2n_3 &\equiv 0 \quad \mathrm{mod} \, 3      \\
n_0+2n_1-n_2  &\equiv 0 \quad \mathrm{mod} \, 3.
\end{align*} 
By solving this system of equations in $\mathbb{Z}/3\mathbb{Z}$, 
we obtain 
\begin{equation*}
(n_0, n_1, n_2, n_3)
\equiv
n(1, 1, 0, 1)+
m(0, 1, -1, 0) \,\,\mathrm{mod} \, 3,
\end{equation*}
where $m,n =0, \pm1.$ We then prove the case (1).
\newline
\newline
Case (2) \quad 
By $\pi$, we assume that $f_0, f_2$ both have a pole at $t=0$. 
From Propositions \ref{prop:a5inf}, \ref{prop:a5zero} and Corollary \ref{coro:a5res}, 
it follows that 
\begin{equation}
\label{eqn:ccon1}
\alpha_1+2\alpha_3-\alpha_5, \,
-3\alpha_3+3\alpha_5, \,
-\alpha_2+\alpha_3+\alpha_4, 
\alpha_0-\alpha_4-\alpha_5 \in \mathbb{Z}.
\end{equation}
\par
We suppose that $\alpha_1 \neq 0$. 
We then get 
\begin{align*}
&t=\infty                         & 
&(f_i)_{0\leq i \leq 5}           & 
&\stackrel{s_1}{\longrightarrow}  & 
&(f_i)_{0\leq i \leq 5}  \\
&t=0                              & 
&(f_0, f_2)_0                     & 
&\stackrel{s_1}{\longrightarrow}  & 
&\mathrm{holomorphic}  \\
&                                 & 
&(\alpha_i)_{0\leq i \leq 5}      & 
&\stackrel{s_1}{\longrightarrow}  & 
&(
\alpha_0+\alpha_1, 
-\alpha_1, 
\alpha_2+\alpha_1, 
\alpha_3, 
\alpha_4, 
\alpha_5
).
\end{align*}
Since all of $s_1(f_i)_{0\leq i \leq  5}$ are holomorphic at $t=0,$ 
the parameters 
$s_1(\alpha_i)_{0\leq i \leq 5}$ 
satisfy the condition of case (1). 
Therefore, we obtain 
\begin{equation*}
\alpha_1-\alpha_2-\alpha_4 \equiv \frac{n}{3}, \,
\alpha_3-\alpha_5 \equiv \frac{m+n}{3},  \,
\alpha_0+\alpha_1-\alpha_4 \equiv \frac{-m}{3}, \,  
-\alpha_1-\alpha_5 \equiv \frac{n}{3} \,\,
\mathrm{mod} \, \mathbb{Z}, \,
\end{equation*}
\par
We suppose that $\alpha_1=0$. 
Equation (\ref{eqn:ccon1}) implies that 
\begin{equation}
\label{eqn:ccon2}
-\alpha_3+2\alpha_5, \,
-3\alpha_3+3\alpha_5, \,
-\alpha_2+\alpha_3+\alpha_4, \,
\alpha_0-\alpha_4-\alpha_5 \in \mathbb{Z}. 
\end{equation}
If $\alpha_1=0$ and $\alpha_0 \neq 1,$ 
we get 
\begin{align*}
&t=\infty                         & 
&(f_i)_{0 \leq i \leq 5}          & 
&\stackrel{s_0}{\longrightarrow}  & 
&(f_i)_{0 \leq i \leq 5}  \\
&t=0                              & 
&(f_0, f_2)_0                     & 
&\stackrel{s_0}{\longrightarrow}  & 
&(f_0, f_2)_0  \\
&                                 & 
&(\alpha_i)_{0\leq i \leq 5}      & 
&\stackrel{s_0}{\longrightarrow}  & 
&(
-\alpha_0, 
\alpha_0, 
\alpha_2, 
\alpha_3, 
\alpha_4, 
\alpha_5+\alpha_0
).
\end{align*}
We set $\hat{\alpha}_j:=s_0(\alpha_j) _{0\leq j \leq 5}$. 
Equation (\ref{eqn:ccon1}) implies that 
\begin{equation}
\label{eqn:ccon3}
\hat{\alpha}_1+2\hat{\alpha}_3-\hat{\alpha}_5, \, 
-3\hat{\alpha}_3+3\hat{\alpha}_5, \,
-\hat{\alpha}_2+\hat{\alpha}_3+\hat{\alpha}_4, \,
\hat{\alpha}_0-\hat{\alpha}_4-\hat{\alpha}_5 \in \mathbb{Z}.
\end{equation}
Therefore, equations (\ref{eqn:ccon2}) and (\ref{eqn:ccon3}) imply that 
\begin{equation}
\label{eqn:ccon4}
2\alpha_3-\alpha_5, \,
-3\alpha_3+3\alpha_5, \,
3\alpha_0, \,
-\alpha_2+\alpha_3+\alpha_4, \,
\alpha_0-\alpha_4-\alpha_5 \in \mathbb{Z}.
\end{equation}
By $s_1$, we get 
\begin{align*}
&t=\infty                          & 
&(f_i)_{0\leq i\leq 5}             & 
&\stackrel{s_1}{\longrightarrow}   & 
&(f_i)_{0\leq i \leq 5}   \\
&t=0                               & 
&(f_0, f_2)_0                      & 
&\stackrel{s_1}{\longrightarrow}   & 
&\mathrm{holomorphic}  \\
&                                  & 
&(
-\alpha_0, 
\alpha_0,
\alpha_2, 
\alpha_3, 
\alpha_4, 
\alpha_5+\alpha_0
) & 
&\stackrel{s_1}{\longrightarrow}   & 
&(
0, 
-\alpha_0, 
\alpha_2+\alpha_0, 
\alpha_3, 
\alpha_4, 
\alpha_5+\alpha_0
).
\end{align*}
We set $\tilde{\alpha}_j:=s_1  s_0 (\alpha_j)_{0\leq j \leq 5}$. 
Since 
all of 
$s_1  s_0 (f_i)_{0\leq i \leq 5}$ have a pole at $t=\infty$ 
and 
are holomorphic at $t=0,$ 
it follows that 
\begin{equation}
\label{eqn:ccon6}
3\tilde{\alpha}_2-3\tilde{\alpha}_4, 
3\tilde{\alpha}_3-3\tilde{\alpha}_5, 
3\tilde{\alpha}_0-3\tilde{\alpha}_4, 
3\tilde{\alpha}_1-3\tilde{\alpha}_5 \in \mathbb{Z}.
\end{equation} 
Equations (8), (\ref{eqn:ccon4}) and (\ref{eqn:ccon6}) 
imply that 
\begin{align*}
&3\alpha_0, \, \alpha_1(=0), \, 3\alpha_2, \, 3\alpha_3, \, 3\alpha_4, \, 3\alpha_5 \in \mathbb{Z},  \\
&
\alpha_0-\alpha_2+\alpha_5, \, 
-\alpha_2+\alpha_3+\alpha_4, \,
\alpha_0-\alpha_4-\alpha_5, \,
2\alpha_3+2\alpha_5 \in \mathbb{Z}.
\end{align*}
We set $\alpha_k=\frac{n_k}{3}, \, n_k \in \mathbb{Z}, \,(k=0, 2, 3, 4, 5).$ 
Therefore, 
we have 
\begin{align*}
n_0-n_2+n_5       &\equiv 0 \quad \mathrm{mod} \, 3 \\
-n_2+n_3+n_4      &\equiv 0 \quad \mathrm{mod} \, 3 \\
n_0-n_4-n_5       &\equiv 0 \quad \mathrm{mod} \, 3 \\
n_3+n_5           &\equiv 0 \quad \mathrm{mod} \, 3.
\end{align*}
By solving this system of equations in $\mathbb{Z}/3\mathbb{Z}$, 
we get 
\begin{equation}
\label{eqn:c0}
(\alpha_0, \alpha_2, \alpha_3, \alpha_4, \alpha_5)
\equiv
\frac{p}{3}
(1, 1, 0, 1, 0)
+
\frac{q}{3}
(1, -1, -1, 0, 1) \,
\mathrm{mod} \, \mathbb{Z} \,(p,q=0,\pm1).
\end{equation}
\par
If $\alpha_1=0$ and $\alpha_2\neq 0,$ 
we can prove that the parameters 
satisfy the condition (\ref{eqn:c0}) in the same way.
\par
We suppose that $\alpha_0=\alpha_1=\alpha_2=0$. 
Equation (\ref{eqn:ccon1}) implies that 
$
\alpha_3, \, \alpha_4, \, \alpha_5 \in \mathbb{Z}.
$ 
The parameters $(0, 0, 0, \alpha_3, \alpha_4, \alpha_5)$ 
then satisfy the condition (\ref{eqn:c0}). 
From Proposition \ref{prop:a5zero}, 
it follows that 
$
\mathrm{Res}_{t=0} f_0 = -\alpha_3-\alpha_5.
$ 
Since $f_0$ has a pole at $t=0$ with the first order, 
it follows that $\alpha_3 \neq 0$ or $\alpha_5 \neq 0$. 
If $\alpha_3 \neq 0 $ and $\alpha_5 \neq 0$, 
the solution 
can be transformed into a holomorphic solution at $t=0$ 
by $s_5  s_3$. 
If $\alpha_3 \neq 0$ and $\alpha_5 =0,$ 
the solution 
can be transformed into a holomorphic solution at $t=0$ by 
$s_5  s_4  s_3$. 
If $\alpha_3=0$ and $ \alpha_5 \neq 0,$ 
the solution 
can be transformed into a holomorphic solution at $t=0$ by 
$s_3  s_4  s_5$.	
\end{proof}

\section{The Standard Forms of The Parameters for Rational Solutions}
In Sections 5, 6 and 7, 
we have obtained the necessary conditions for $A_5^{(1)}(\alpha_j)_{0\leq j \leq 5}$ to have rational solutions of Type A, Type B and Type C, 
and 
expressed them by the parameters. 
In this section, 
using B\"acklund transformations,  
we transform the parameters into the standard forms. 
\par
This section consists of four subsections. 
In Subsection 8.1, 
following Noumi and Yamada \cite{NoumiYamada-A}, 
we introduce the shift operators of the parameters, $\alpha_j \,\,(0\leq j \leq 5).$ 
In Subsections 8.2, 8.3 and 8.4, 
we treat the necessary conditions for Type A, Type B and Type C and 
transform the parameters 
into the standard forms.

\subsection{Shift operators}
In this subsection, 
following Noumi and Yamada \cite{NoumiYamada-A}, 
we introduce the shift operators of the parameters, $\alpha_j \,\,(0\leq j \leq 5).$ 
Noumi and Yamada \cite{NoumiYamada-A} 
defined the shift operators of the parameters in the following way:
\begin{proposition}
\it{
For any $i=0,1,2,3,4,5,$ 
$T_i$ 
denote 
the shift operators  
which are defined by 
\begin{align*}
&T_1 = \pi s_5 s_4 s_3 s_2 s_1,      & 
&T_2 = s_1 \pi s_5 s_4 s_3 s_2,      & 
&T_3 = s_2 s_1 \pi s_5 s_4 s_3,       \\
&T_4 = s_3 s_2 s_1 \pi s_5 s_4,      & 
&T_5 = s_4 s_3 s_2 s_1 \pi s_5,      & 
&T_6 = s_5 s_4 s_3 s_2 s_1 \pi. 
\end{align*}
Then, 
$$
T_i (\alpha_{i-1}) 
= 
\alpha_{i-1}+1_, \, 
T_i(\alpha_i) 
=\alpha_i-1, \, 
T_i(\alpha_j)=\alpha_j \, (j \neq i - 1, i).
$$
}
\end{proposition}

\subsection{The standard forms of the parameters for rational solutions of Type A}
In this subsection, 
let us first suppose that 
for $A_5^{(1)}(\alpha_j)_{0\leq j \leq 5},$ 
there exists a rational solution of Type A. 
From Proposition \ref{prop:necA}, it follows that 
the solution and parameters 
can be transformed 
so that 
all of $(f_j)_{0\leq j \leq 5}$ are holomorphic at $t=0$ 
and 
the parameters 
satisfy one of the following conditions:
\newline
(1) \, for some $i=0, 1, 2, 3, 4, 5,$ 
$
\alpha_{i+2}, \, \alpha_{i+3}, \, 
\alpha_{i+4}, \, \alpha_{i+5} 
\in \mathbb{Z}; 
$
\newline
(2) \, for some $i=0, 1, 2, 3, 4, 5,$ 
$
\alpha_{i+1}, \, \alpha_{i+2}, \,
\alpha_{i+4}, \, \alpha_{i+5} \,
\in \mathbb{Z};
$
\newline
(3) \, for some $i=0, 1, 2, 3, 4, 5,$ 
$
\alpha_{i+3}, \, \alpha_{i+5}, \, 
\alpha_i+\alpha_{i+4}, \, 
\alpha_i-\alpha_{i+2} 
\in \mathbb{Z}.
$
\par
In the following proposition, 
by some B\"acklund transformations, 
we transform the three kinds of parameters into 
the two standard forms, 
$
(\alpha_0, 1-\alpha_0, 0, 0, 0, 0)
$ 
and 
$
(\alpha_0, 0, 0, 1-\alpha_0, 0, 0).
$

\begin{proposition}
\label{prop:standardA}
\it{
Suppose that for $A_5^{(1)}(\alpha_j)_{0\leq  j \leq 5},$ 
there exists a rational solution of Type A.  
By some B\"acklund transformations, 
the solution 
can then be transformed into one 
of the following parameters:
\begin{equation*}
\mathrm{(i)} \,\,
(\alpha_0, 1-\alpha_0, 0, 0, 0, 0); 
\quad
\mathrm{(ii)} \,\,
(\alpha_0, 0, 0, 1-\alpha_0, 0, 0).
\end{equation*}
}
\end{proposition}

\begin{proof}
Case (1) \quad 
By $\pi$, we assume that 
$
\alpha_{2}, \alpha_{3}, 
\alpha_{4}, \alpha_{5} 
\in \mathbb{Z}.
$ 
First, by $T_6$, we have $\alpha_5=0$. 
Second, by $T_4$, we get $\alpha_4=0$. 
Third, by $T_3$, we obtain $\alpha_3=0$. 
Lastly, by $T_2$, we have $\alpha_2=0$. 
Since $\sum_{k=0}^5 \alpha_k=1$, 
it follows that 
$
(\alpha_i)_{0\leq i \leq 5}
\longrightarrow
(\alpha_0, 1-\alpha_0, 0, 0, 0, 0).
$
\newline
Case (2)
\quad 
By $\pi$, we assume that 
$
\alpha_{1}, \alpha_{2}, 
\alpha_{4}, \alpha_{5} 
\in \mathbb{Z}.
$ 
First, by $T_6$, we have $\alpha_5=0$. 
Second, by $T_4$, we get $\alpha_4=0$. 
Third, by $T_3$, we obtain $\alpha_2=0$. 
Lastly, by $T_1$, we have $\alpha_1=0$. 
Since $\sum_{k=0}^5 \alpha_k=1$, 
it follows that 
$
(\alpha_i)_{0\leq i \leq 5}
\longrightarrow
(\alpha_0, 0, 0, 1-\alpha_0, 0, 0).
$
\newline
Case (3)
\quad
By $\pi$, 
we suppose that 
$
\alpha_0+\alpha_4, 
\alpha_0-\alpha_2, 
\alpha_3, \alpha_5 \in \mathbb{Z}.
$ 
This condition is equivalent to the following condition: 
$
\alpha_0+\alpha_4, 
\alpha_2+\alpha_4, 
\alpha_3, \alpha_5 \in \mathbb{Z}.
$ 
First, 
by $T_4, T_5,$ we have $\alpha_3=0, \alpha_5=0$, respectively. 
Second, 
by $T_1, T_2,$ 
we get $\alpha_0+\alpha_4=0, \alpha_2+\alpha_4=0,$ respectively. 
Since $\sum_{k=0}^5\alpha_k=1$, 
it follows that
$
(\alpha_i)_{0\leq i \leq 5}
\longrightarrow
(\alpha_0, 1-\alpha_0, \alpha_0, 0, -\alpha_0, 0).
$ 
If $\alpha_0 \neq 0,$ 
by 
$\pi^2  s_4  s_5  s_3  s_4$, 
we get 
$
(\alpha_0, 1-\alpha_0, \alpha_0, 0, -\alpha_0, 0)
\rightarrow
(\alpha_0, 0, 0, 1-\alpha_0, 0, 0).
$ 
If $\alpha_0=0$, 
we have 
$
(\alpha_0, 1-\alpha_0, \alpha_0, 0, -\alpha_0, 0)
=
(0, 1, 0, 0, 0, 0 ).
$
\end{proof}

\subsection{The standard forms of the parameters for Type B}
In this subsection, 
we suppose that for $A_5^{(1)}(\alpha_j)_{0\leq j \leq 5},$  
there exists a rational solution of Type B.  
From Proposition \ref{prop:necB}, 
it follows that by some B\"acklund transformations, 
the solutions and parameters 
can be transformed so that 
all of $(f_j)_{0\leq j \leq 5}$ are holomorphic at $t=0$ 
and 
the parameters 
satisfy the following condition: 
for some $i=0,1,2,3,4,5,$  
$$
-\alpha_{i}+\alpha_{i+2}-\alpha_{i+4}, \,
-\alpha_{i+1}+\alpha_{i+3}+\alpha_{i+5}, \,
2\alpha_{i+4}, \, -2\alpha_{i+5} \in \mathbb{Z}.
$$
\par
By $\pi$, we assume that 
$
2\alpha_{4}, 2\alpha_{5}, 
-\alpha_0+\alpha_{2}-\alpha_{4}, 
-\alpha_{1}+\alpha_{3}+\alpha_{5} \in \mathbb{Z}.
$ 
First, 
by $T_4$ and $T_6,$ we have 
$
\alpha_4=0, \, 1/2 \,\,
\mathrm{and} \,\,
\alpha_5=0, \, -1/2,
$ 
respectively.
Second, 
by $T_1  T_2$ and $T_2  T_3,$ 
we get 
$
\beta^0_{-1}=-\alpha_0+\alpha_2-\alpha_4 = 0, \, 1
$ 
and 
$
\gamma^0_{-1}=-\alpha_1+\alpha_3+\alpha_5 = 0, \, 1,
$
respectively. We then obtain 
\begin{equation*} 
\epsilon^0_{-1}=2\alpha_4=0, \, 1,  
\quad
\phi^0_{-1}=-2\alpha_5=0, \, 1,  
\quad
\beta^0_{-1}=0,1, 
\quad 
\gamma^0_{-1}=0, \, 1.
\end{equation*}
Therefore, 
we have only to consider the $2^4=16$ cases.
\begin{proposition}
\it{
\label{prop:standardB}
Suppose that 
for $A_5^{(1)}(\alpha_j)_{0\leq j \leq 5},$ 
there exists a rational solution of Type B. 
By some B\"acklund transformations, 
the parameters 
can be 
transformed into the following three types:
\begin{equation*}
\mathrm{(1)}
\,\,
(\alpha_0,-\alpha_0+1/2,\alpha_0,-\alpha_0+1/2,0,0), \,\,
\mathrm{(2)}
\,\,
(1/2, 0, 1/2, \alpha_0, 0, -\alpha_0), \,\,
\mathrm{(3)}
\,\,
(\alpha_0, 0, -\alpha_0+1, 0, 0, 0).
\end{equation*}
\par
The parameters in the sixteen cases 
can be transformed into the parameters of (1) 
if and only if  
$\alpha_i \,\,(i=0,1,2,3,4,5)$ satisfy 
the following condition: 
$
\beta^0_{-1}+\gamma^0_{-1} \equiv 0 \,\, \mathrm{mod} \,2.
$
\par
The parameters in the sixteen cases 
can be transformed into the parameters of (2) 
if and only if 
$\alpha_i \,\,(i=0,1,2,3,4,5)$ satisfy 
one of 
the following conditions:
\begin{equation*}
(\beta^0_{-1},\gamma^0_{-1},\epsilon^0_{-1},\phi^0_{-1})
=
(0,1,1,0),
(0,1,0,1),
(0,1,1,1),
(1,0,1,0),
(1,0,0,1),
(1,0,1,1).
\end{equation*}
\par
The parameters in the sixteen cases 
can be transformed into the parameters of (3) 
if and only if 
$\alpha_i \,\,(i=0,1,2,3,4,5)$ satisfy one of 
the following conditions: 
$
(\beta^0_{-1},\gamma^0_{-1},\epsilon^0_{-1},\phi^0_{-1})
=
(0,1,0,0),
(1,0,0,0).
$
}
\end{proposition}

\begin{proof}
We prove that the proposition is true in the following three cases:
\newline
(i) \quad
$
\beta^0_{-1}=\gamma^0_{-1}= 
\epsilon^0_{-1}=\phi^0_{-1}=1;
$
\newline
(ii) \quad
$
\beta^0_{-1}=0, \gamma^0_{-1}=1, 
\epsilon^0_{-1}=0, \phi^0_{-1}=1;
$ 
\newline
(iii) \quad
$
\beta^0_{-1}=0, \gamma^0_{-1}=1, 
\epsilon^0_{-1}=\phi^0_{-1}=0.
$ 
\newline
The other cases can be proved in the same way.
\newline
Case (i) 
\quad
Since 
$
\beta^0_{-1}=\gamma^0_{-1}= 
\epsilon^0_{-1}=\phi^0_{-1}=1,
$ 
it follows that 
$\alpha_2=\alpha_0+3/2, \alpha_3=\alpha_1+3/2, 
\alpha_4=1/2, \alpha_{5}=-1/2.
$
Since $\sum_{j=0}^5 \alpha_j =1$, it follows that $\alpha_1=-\alpha_0-1.$ 
By $s_2  \pi^{-1}  T_4^{-1}  s_1  s_4  T_2^{-1}$, 
we get 
\begin{equation*}
(\alpha_0, -\alpha_0-1, \alpha_0+3/2, -\alpha_0+1/2, 1/2, -1/2)
\longrightarrow
(\alpha_0, -\alpha_0+1/2, \alpha_0, -\alpha_0+1/2, 0, 0).
\end{equation*}
\newline
Case (ii) 
\quad
Since 
$
\beta^0_{-1}=0, \gamma^0_{-1}=1, 
\epsilon^0_{-1}=0, \phi^0_{-1}=1,
$ 
it follows that 
$
\alpha_2=\alpha_0, \alpha_3=\alpha_1+3/2, \alpha_4=0, \alpha_5=-1/2.
$ 
Since $\sum_{j=0}^5 \alpha_j =1$, we have $\alpha_1=-\alpha_0.$ 
By $\pi^3  s_2  T_5^{-1}  T_4^{-1},$ 
we get 
\begin{equation*}
(\alpha_0, -\alpha_0, \alpha_0, -\alpha_0+3/2, 0, -1/2)
\longrightarrow
(1/2, 0, 1/2,\alpha_0, 0, -\alpha_0).
\end{equation*}
\newline
Case (iii) 
\quad
Since 
$
\beta^0_{-1}=0, \gamma^0_{-1}=1, 
\epsilon^0_{-1}=\phi^0_{-1}=0,
$ 
it follows that 
$
\alpha_2=\alpha_0, \alpha_3=\alpha_1+1,\alpha_4=\alpha_5=0.
$ 
Since $\sum_{j=0}^5 \alpha_j =1$, we have $\alpha_1=-\alpha_0.$ 
By $\pi^{-1}  s_1$, we get 
\begin{equation*}
(\alpha_0, -\alpha_0, \alpha_0, -\alpha_0+1, 0, 0)
\longrightarrow
(\alpha_0, 0, -\alpha_0+1, 0, 0, 0).
\end{equation*}
\end{proof}

\subsection{The standard forms of the parameters for Type C}
In this subsection, 
we suppose that for $A_5^{(1)}(\alpha_j)_{0\leq j \leq 5},$  
there exists a rational solution of Type C. 
From Proposition \ref{prop:necC}, 
it follows that by some B\"acklund transformations, 
the solutions and parameters  
can be transformed 
so that 
all of $(f_j)_{0\leq j \leq 5}$ are holomorphic at $t=0$ 
and 
$(\alpha_j)_{0\leq j \leq 5}$ 
satisfy the following condition:
\begin{equation*}
x=\alpha_2-\alpha_4
\equiv
\frac{n}{3}, \,
y=\alpha_3-\alpha_5
\equiv 
\frac{m+n}{3}, \,
z=\alpha_0-\alpha_4 
\equiv
\frac{2m}{3}, \,
w=\alpha_1-\alpha_5
\equiv
\frac{n}{3} \,
\mathrm{mod} \,
\mathbb{Z} \, (m,n=0, \pm1).
\end{equation*}
\par
We consider the constant term $h_{\infty,0}$ of the Laurent series of $H$ at $t=\infty$ 
and get the following lemma:
\begin{lemma}
\it{
Suppose that for $A_5^{(1)}(\alpha_j)_{0\leq j \leq 5},$ 
there exists a rational solution of Type C, 
all of which are holomorphic at $t=0$. 
Then 
$m=n.$
}
\end{lemma}

\begin{proof}
We set 
$
x=(n+3k_0)/3, \,\,
y=(m+n+3k_1)/3, \,\,
z=(2m+3k_2)/3, \,\,
\omega=(n+3k_3)/3,
$
where 
$k_j \in \mathbb{Z} \,(j=0,1,2,3).$ 
From Proposition \ref{prop:hinf}, 
it follows that 
$
h_{\infty,0}
=
1/27\cdot
(6n^2+6m^2-3mn+9l),\,\,\,l \in \mathbb{Z}.
$
The proof of 
Proposition \ref{prop:hplus} shows that 
$
-2\sum_{k=1}^s \epsilon_k
=
-h_{\infty,0}+h_{0,0}
=
-1/27\cdot(6n^2+6m^2-3mn+9l),
$
where 
$\epsilon_k=1/6, 1/12, 5/12 \,\, 
(1 \leq k \leq s).$ 
We then obtain 
$
2n^2+2m^2-mn \equiv 0 \,\,
\mathrm{mod}\, 3.
$
Therefore, we have $m \equiv n \,\mathrm{mod} \, 3$.
\end{proof} 
For a holomorphic solution at $t=0$ of Type C, 
we set 
$$
\chi:=
x+y+z+\omega
=\alpha_0+\alpha_1+\alpha_2+\alpha_3
-2\alpha_4-2\alpha_5.
$$
By $T_4$, we have 
$\chi=0, \pm 1$.  
By $T_1,T_2,T_3,$ we get 
$
z=2n/3, \,
\omega=n/3, \, 
x=n/3, \,\,(n=0,\pm1),
$ 
respectively. 
Since $T_1,T_2,T_3$ all preserve the value of $\chi,$ 
we can determine the value of $y.$ 
Thus, 
we have only to consider the following 
$3\times3=9$ cases:
$$
\chi=0,\pm1, \,
x=n/3, \,
z=2n/3, \,
\omega=n/3, \, 
\,\,(n=0,\pm1).
$$

\begin{proposition}
\label{prop:standardC}
\it{
Suppose that 
for $A_5^{(1)}(\alpha_j)_{0\leq  j \leq 5},$ 
there exists a rational solution of Type C. 
By some B\"acklund transformations, 
the parameters can then be transformed into one of the following 
parameters:
\newline
(1) \,
$
(
\alpha_4,
-\alpha_4+1/3,
\alpha_4,
-\alpha_4+1/3,
\alpha_4,
-\alpha_4+1/3
),
$
\quad
(2) \,
$
(-\alpha_4+1/3,1/3,1/3,\alpha_4,0,0),
$
\newline
(3) \,
$
(\alpha_4,0,0,1-\alpha_4,0,0),
$
\hspace{42mm}
(4) \,
$(\alpha_4,1/3,1/3,-\alpha_4+1/3,0,0).$
\par
The parameters of the nine cases can be transformed 
into the parameters of (1) if and only if 
$$
(\chi,n)
=
(0,0),(0,-1),(-1,0),(-1,1).
$$
\par
The parameters of the nine cases can be transformed 
into the parameters of (2) if and only if 
$$
(\chi,n)
=
(0,1),(-1,-1). 
$$
\par
The parameters of the nine cases can be transformed 
into the parameters of (3) if and only if 
$$
(\chi,n)
=
(1,0).
$$
\par
The parameters of the nine cases can be transformed 
into the parameters of (4) if and only if 
$$
(\chi,n)
=
(1,1),(1,-1).
$$
}
\end{proposition}

\begin{proof}
We prove the proposition is true in the following four cases:
\begin{equation*}
\mathrm{(i)} \,\,(\chi,n)=(0,-1),  
\mathrm{(ii)} \,\,(\chi,n)=(0,1), 
\mathrm{(iii)} \,\,(\chi,n)=(1,0), 
\mathrm{(iv)} \,\,(\chi,n)=(1,1).
\end{equation*}
The other cases can be proved in the same way.
\newline
Case (i) \quad
Since $(\chi,n)=(0,-1)$, we have
$
\alpha_0=\alpha_4+1/3, 
\alpha_1=\alpha_5-1/3,
\alpha_2=\alpha_4-1/3, 
\alpha_3=\alpha_5+1/3.
$ 
Since $\sum_{k=0}^5 \alpha_k=1$, we get
$\alpha_4+\alpha_5=1/3.$ 
Therefore, we obtain
\begin{equation*}
(
\alpha_0,
\alpha_1,
\alpha_2,
\alpha_3,
\alpha_4,
\alpha_5
)
=
(
\alpha_4+1/3, 
-\alpha_4,
\alpha_4-1/3,
-\alpha_4+2/3,
\alpha_4,
-\alpha_4+1/3
).
\end{equation*}
By $s_1  s_2  s_1$, 
we have
\begin{equation*}
(
\alpha_4+1/3, 
-\alpha_4,
\alpha_4-1/3,
-\alpha_4+2/3,
\alpha_4,
-\alpha_4+1/3
)
\longrightarrow
(
\alpha_4,
-\alpha_4+1/3,
\alpha_4,
-\alpha_4+1/3,
\alpha_4,
-\alpha_4+1/3
).
\end{equation*}
Case (ii) \,
Since $(\chi,n)=(0,1)$, we get 
$
\alpha_0=\alpha_4-1/3,
\alpha_1=\alpha_5+1/3,
\alpha_2=\alpha_4+1/3,
\alpha_3=\alpha_5-1/3.
$
Since $\sum_{k=0}^5 \alpha_k=1$, we obtain 
$\alpha_4+\alpha_5=1/3.$ We then have
\begin{equation*}
(
\alpha_0,
\alpha_1,
\alpha_2,
\alpha_3,
\alpha_4,
\alpha_5
)
=
(
\alpha_4-1/3,
-\alpha_4+2/3,
\alpha_4+1/3,
-\alpha_4, 
\alpha_4,
-\alpha_4+1/3
).
\end{equation*}
By $s_0  s_3$, we get
$
(
\alpha_4-1/3,
-\alpha_4+2/3,
\alpha_4+1/3,
-\alpha_4, 
\alpha_4,
-\alpha_4+1/3
)
\longrightarrow
(-\alpha_4+1/3,1/3,1/3,\alpha_4,0,0).
$
\newline
Case (iii) \,
Since $(\chi,n)=(1,0)$, we obtain
$
\alpha_0=\alpha_4,
\alpha_1=\alpha_5,
\alpha_2=\alpha_4,
\alpha_3=\alpha_5+1.
$ 
Since $\sum_{k=0}^5 \alpha_k=1$, we have 
$\alpha_4+\alpha_5=0.$ Therefore, we get
$
(
\alpha_0,
\alpha_1,
\alpha_2,
\alpha_3,
\alpha_4,
\alpha_5
)
=
(
\alpha_4,
-\alpha_4,
\alpha_4,
-\alpha_4+1,
\alpha_4,
-\alpha_4
).
$ 
By $\pi^{-1}  T_4  s_4  s_1$, 
we obtain 
$
(
\alpha_4,
-\alpha_4,
\alpha_4,
-\alpha_4+1,
\alpha_4,
-\alpha_4
)
\longrightarrow
(\alpha_4,0,0,1-\alpha_4,0,0).
$
\newline
Case (iv) \,
Since $(\chi,n)=(1,1)$, we have 
$
\alpha_0=\alpha_4-1/3,
\alpha_1=\alpha_5+1/3,
\alpha_2=\alpha_4+1/3,
\alpha_3=\alpha_5+2/3.
$
Since $\sum_{k=0}^5 \alpha_k=1$, 
we get 
$\alpha_4+\alpha_5=0.$ Therefore, 
we obtain 
\begin{equation*}
(
\alpha_0,
\alpha_1,
\alpha_2,
\alpha_3,
\alpha_4,
\alpha_5
)
=
(
\alpha_4-1/3,
-\alpha_4+1/3,
\alpha_4+1/3,
-\alpha_4+2/3,
\alpha_4,
-\alpha_4
).
\end{equation*}
By $\pi^{-1}  s_1  s_0  s_4  s_5  s_0,$ 
we have 
\begin{equation*}
(
\alpha_4-1/3,
-\alpha_4+1/3,
\alpha_4+1/3,
-\alpha_4+2/3,
\alpha_4,
-\alpha_4
)
\longrightarrow
(
\alpha_4,
1/3,
1/3,
-\alpha_4+1/3,
0,0).
\end{equation*}
\end{proof}

\section{The Standard Forms of The Parameters and Rational Solutions of Type A}
In this section, 
we determine the rational solutions of Type A of 
$A_5^{(1)}(\alpha_0, 1-\alpha_0, 0, 0, 0, 0)$ 
and 
$A_5^{(1)}(\alpha_0, 0, 0, 1-\alpha_0, 0, 0),$ 
whose parameters are the standard forms of Type A.  
\par
This section consists of three subsections. 
In Subsection 9.1, 
we prove the lemmas in order to study 
the rational solutions of Type A of 
$A_5^{(1)}(\alpha_0, 1-\alpha_0, 0, 0, 0, 0)$ 
and 
$A_5^{(1)}(\alpha_0, 0, 0, 1-\alpha_0, 0, 0).$ 
\par
In Subsections 9.2 and 9.3, 
we determine the rational solutions of Type A 
of 
$A_5^{(1)}(\alpha_0, 1-\alpha_0, 0, 0, 0, 0)$ 
and 
$A_5^{(1)}(\alpha_0, 0, 0, 1-\alpha_0, 0, 0),$ 
respectively.

\subsection{Lemmas about rational solutions of Type A}
In this subsection, 
we prove lemmas in order to study 
the rational solutions of Type A of 
$A_5^{(1)}(\alpha_0, 1-\alpha_0, 0, 0, 0, 0)$ 
and 
$A_5^{(1)}(\alpha_0, 0, 0, 1-\alpha_0, 0, 0).$ 
\begin{lemma}
\label{lem:zeropole1}
\it{
Suppose that for some $i=0, 1, 2, 3, 4, 5,$ 
$
(
\alpha_i, 
\alpha_{i+1}, 
\alpha_{i+2}, 
\alpha_{i+3}, 
\alpha_{i+4}, 
\alpha_{i+5}
)
=
(1,0,0,0,0,0)
$
and 
for $
A_5^{(1)}(
\alpha_j
)_{0\leq j \leq 5},
$ 
there exists a rational solution. 
Then, 
\begin{equation*}
(f_{i+1}, f_{i+3}), 
(f_{i+3}, f_{i+5}), 
(f_{i+5}, f_{i+1}), 
(f_{i+2}, f_{i+4}, f_{i+5}, f_{i+1}), 
\end{equation*}
can have a pole at $t=0$.
}
\end{lemma}

\begin{proof}
We calculate the residues of $f_j \,\,(j=0,1,2,3,4,5)$ at $t=0$ 
in Proposition \ref{prop:a5zero} 
and 
obtain the lemma.
\end{proof}

In the following four lemmas, 
we determine the rational solutions of Type A 
of 
$A_5^{(1)}(0,1,0,0,0,0),$ 
$A_5^{(1)}(1/2,1/2,0,0,0,0),$ 
$A_5^{(1)}(1/2,0,0,1/2,0,0)$ 
and 
$A_5^{(1)}(1/3,0,0,2/3,0,0).$  

\begin{lemma}
\label{lem:a5Apoint1}
\it{
Suppose that 
for $A_5^{(1)}(0,1,0,0,0,0),$ 
there exists a rational solution $(f_i)_{0\leq i \leq 5}$ 
of Type A.
Then, 
$
(f_0, f_1, f_2, f_3, f_4, f_5)
=
(t,t,0,0,0,0), \,\,
(0,t,t,0,0,0), \,\,
(0,t,0,0,t,0),  \,\,
(t,t,t,0,-t,0).
$
}
\end{lemma}

\begin{proof}
If $(f_0,f_1),(f_1,f_2),(f_1,f_4),(f_0,f_1,f_2,f_4)$ have a pole at $t=\infty$, 
it follows 
from Proposition \ref{prop:a5inf} that 
$$
(f_0, f_1, f_2, f_3, f_4, f_5)
=
(t,t,0,0,0,0), \,\,
(0,t,t,0,0,0), \,\,
(0,t,0,0,t,0),  \,\,
(t,t,t,0,-t,0),
$$ 
respectively.
\par
We assume that $f_2,f_3$ both have a pole at $t=\infty$ 
and show a contradiction. 
The other cases can be proved in the same way. 
\par
If $f_2,f_3$ both have a pole at $t=\infty$, 
it follows 
from Proposition \ref{prop:uniqueness} that 
$f_4=f_5=f_0 \equiv 0.$ 
If some of $(f_i)_{0\leq i \leq 5}$ 
have a pole at $t=0,$ 
it follows from Proposition \ref{prop:a5zero} that 
$(f_i,f_{i+2})$ or $(f_i,f_{i+2},f_{i+3},f_{i+5})$ 
can have a pole at $t=0$ for some $i=0,1,2,3,4,5.$ 
Since $f_4=f_5=f_0\equiv 0,$ 
only $(f_1,f_3)$ can have a pole at $t=0.$ 
If $(f_1,f_3)$ have a pole at $t=0,$ we get 
$-h_{\infty,0}+h_{0,0}=1/6,$ 
which contradicts Proposition \ref{prop:hplus}.
\end{proof}

\begin{lemma}
\label{lem:a5Apoint2}
\it{
For 
$A_5^{(1)}(1/2,1/2,0,0,0,0),$ 
there exists a rational solution of Type A. 
Then, 
$$
(f_0,f_1,f_2,f_3,f_4,f_5) 
=
(t,t,0,0,0,0),
$$ 
and it is unique.
}
\end{lemma}

\begin{proof}
If $f_0,f_1$ both have a pole at $t=\infty$, 
it follows 
from 
Proposition \ref{prop:a5inf} that 
$
(f_0,f_1,f_2,f_3,f_4,f_5)
=
(t,t,0,0,0,0).
$
\par
We assume that $(f_1,f_2) \, \mathrm{or} \,(f_2,f_3)$ have a pole at $t=\infty$ and 
show a contradiction. The other cases can be proved in the same way.
\par
When $(f_1,f_2)$ have a pole at $t=\infty$, 
it follows 
from 
Propositions \ref{prop:hinf} and \ref{prop:h0} that 
$-h_{\infty,0}+h_{0,0} >0,$ 
which contradicts Proposition \ref{prop:hplus}. 
\par
We suppose that $(f_2,f_3)$ have a pole at $t=\infty.$ 
From Propositions \ref{prop:a5inf} and \ref{prop:uniqueness}, 
it follows that 
\begin{equation*}
-\mathrm{Res}_{t=\infty}f_2=-1/2, 
-\mathrm{Res}_{t=\infty}f_3=1/2, 
f_4=f_5\equiv 0, 
-\mathrm{Res}_{t=\infty}f_0=1/2, 
-\mathrm{Res}_{t=\infty}f_1=1/2. 
\end{equation*}
If some of $(f_i)_{0\leq i \leq 5}$ 
have a pole at $t=0,$ 
it follows 
from Proposition \ref{prop:a5zero} that 
$(f_i,f_{i+2})$ or $(f_i,f_{i+2},f_{i+3},f_{i+5})$ 
can have a pole at $t=0$ for some $i=0,1,2,3,4,5.$ 
Since $f_4=f_5 \equiv 0,$ 
$(f_0,f_2)$ or $(f_1,f_3)$ 
can have a pole at $t=0.$ 
When $(f_0,f_2)$ have a pole at $t=0,$ 
it follows 
from Propositions \ref{prop:hinf} and \ref{prop:h0} that 
$-h_{\infty,0}+h_{0,0}=1/6,$ 
which contradicts Proposition \ref{prop:hplus}. 
When $(f_1,f_3)$ have a pole at $t=0,$ 
we get $-h_{\infty,0}+h_{0,0}=1/6,$ 
which contradicts Proposition \ref{prop:hplus}.
\end{proof}

The following two lemmas can be proved in the same way.

\begin{lemma}
\label{lem:a5Apoint3}
\it{
For 
$A_5^{(1)}(1/2,0,0,1/2,0,0),$ 
there exists a rational solution of Type A. 
Then, 
\begin{equation*}
(f_0,f_1,f_2,f_3,f_4,f_5)
=
(t,0,0,t,0,0),
\end{equation*}
and it is unique.
}
\end{lemma}

\begin{lemma}
\label{lem:a5Apoint4}
\it{
For 
$A_5^{(1)}(1/3,0,0,2/3,0,0),$ 
there exists a rational solution of Type A.
Then, 
\begin{equation*}
(f_0,f_1,f_2,f_3,f_4,f_5)
=
(t,0,0,t,0,0),
\end{equation*}
and it is unique.
}
\end{lemma}

\subsection{Rational solutions of Type A of $A_5^{(1)}(\alpha_0,1-\alpha_0,0,0,0,0)$}
In this subsection, 
we decide the rational solutions of Type A of $A_5^{(1)}(\alpha_0,1-\alpha_0,0,0,0,0).$

\begin{proposition}
\it{
\label{prop:a5A1}
Suppose that 
for 
$A_5^{(1)}(\alpha_0,1-\alpha_0,0,0,0,0),$ 
there exist a rational solution of Type A. 
Then, 
$$
(f_0,f_1,f_2,f_3,f_4,f_5)=(t,t,0,0,0,0), 
$$ 
or by some B\"acklund transformations, 
the parameters and solution can be transformed 
so that either of the following occurs:
\newline
(1) \quad $(\alpha_0,\alpha_1,\alpha_2,\alpha_3, \alpha_4,\alpha_5)=(0,1,0,0,0,0)$ 
and 
$$
(f_0, f_1, f_2, f_3, f_4, f_5)
=
(t,t,0,0,0,0), \,\,
(0,t,t,0,0,0), \,\, 
(0,t,0,0,t,0),  \,\,
(t,t,t,0,-t,0),
$$
(2)\quad $(\alpha_0,\alpha_1,\alpha_2,\alpha_3, \alpha_4,\alpha_5)=(1/2,1/2,0,0,0,0)$ 
and 
$
(f_0,f_1,f_2,f_3,f_4,f_5) 
=
(t,t,0,0,0,0).
$
}
\end{proposition}

\begin{proof}
We treat the case where $(f_0,f_1), (f_1,f_2),(f_2,f_3,f_4,f_0)$ have a pole at $t=\infty$. 
The other cases can be proved in the same way.
\par
If $(f_0,f_1)$ have a pole at $t=\infty$, 
it follows 
from Proposition \ref{prop:a5inf} that 
$
(f_0,f_1,f_2.f_3,f_4,f_5)
=
(t,t,0,0,0,0).
$
\par
If $(f_1,f_2)$ have a pole at $t=\infty$, 
it follows 
from Proposition \ref{prop:a5inf} that 
$$
-\mathrm{Res}_{t=\infty} f_1=0, 
-\mathrm{Res}_{t=\infty} f_2=\alpha_0,
-\mathrm{Res}_{t=\infty} f_0=-\alpha_0, 
f_3=f_4=f_5\equiv 0.
$$ 
If some of $(f_i)_{0\leq i \leq 5}$ 
have a pole at $t=0,$ 
it follows 
from Proposition \ref{prop:a5zero} that 
$(f_i,f_{i+2})$ or $(f_i,f_{i+2},f_{i+3},f_{i+5})$ 
can have a pole at $t=0$ for some $i=0,1,2,3,4,5.$ 
Since $f_3=f_4=f_5\equiv 0,$ 
it follows 
from Proposition \ref{prop:a5zero} that 
only $(f_0,f_2)$ can have a pole at $t=0.$ 
\par
When $(f_0,f_2)$ have a pole at $t=0$, 
$
-h_{\infty,0}+h_{0,0}
=
-\left(-1/3\cdot\alpha_0 \right)+
\left(-1/3\cdot\alpha_0+1/3 \right)
=1/3,
$ 
which contradicts Proposition \ref{prop:hplus}. 
\par
When $(f_0,f_2)$ do not have a pole at $t=0$, 
all of $(f_i)_{0\leq i \leq 5}$ 
are holomorphic at $t=0.$ 
It then follows 
from Corollary \ref{coro:a5res} that 
$\alpha_0 \in \mathbb{Z}.$ 
By $T_1$, we have 
$
(\alpha_0,1-\alpha_0,0,0,0,0)
\longrightarrow
(0,1,0,0,0,0).
$
\par
If $(f_2,f_3,f_4,f_0)$ have a pole at $t=\infty$, 
it follows 
from Propositions \ref{prop:a5inf} and \ref{prop:uniqueness} 
that 
\begin{align*}
&-\mathrm{Res}_{t=\infty} f_2=-\alpha_0,   & 
&-\mathrm{Res}_{t=\infty} f_3=1-\alpha_0,  & 
&-\mathrm{Res}_{t=\infty} f_4=-\alpha_0+2,  \\
&\hspace{18mm}    f_5 \equiv 0,    & 
&-\mathrm{Res}_{t=\infty} f_0=2\alpha_0-2, & 
&-\mathrm{Res}_{t=\infty} f_1=\alpha_0-1,
\end{align*}
which implies that  
$2\alpha_0 \in  \mathbb{Z}$ 
from Proposition \ref{prop:a5zero} and Corollary \ref{coro:a5res}. 
By $T_1$, we obtain 
\begin{equation*}
(\alpha_0,1-\alpha_0,0,0,0,0)
\longrightarrow
(0,1,0,0,0,0) \,\,
\mathrm{or} \,\,
(1/2,1/2,0,0,0,0).
\end{equation*}
\end{proof}

\subsection{Rational solutions of $A_5^{(1)}(\alpha_0,0,0,1-\alpha_0,0,0)$}
In this subsection, 
we determine the rational solutions of Type A of $A_5^{(1)}(\alpha_0,0,0,1-\alpha_0,0,0).$

\begin{proposition}
\label{prop:a5A2}
\it{
Suppose that 
for 
$A_5^{(1)}(\alpha_0,0,0,1-\alpha_0,0,0),$ 
there exist a rational solution of Type A. 
Then, 
$$
(f_0,f_1,f_2,f_3,f_4,f_5)
=
(t,0,0,t,0,0)
$$ 
or by some B\"acklund transformations, 
the parameters and solution 
can be transformed so that 
one of the following occurs:
\newline
(1)\quad $(\alpha_0,\alpha_1,\alpha_2,\alpha_3,\alpha_4,\alpha_5)= (0,1,0,0,0,0),$
and 
$$
(f_0, f_1, f_2, f_3, f_4, f_5)
=
(t,t,0,0,0,0), \,\,
(0,t,t,0,0,0), \,\, 
(0,t,0,0,t,0),  \,\,
(t,t,t,0,-t,0),
$$
(2)\quad 
$(\alpha_0,\alpha_1,\alpha_2,\alpha_3,\alpha_4,\alpha_5)=(1/2,0,0,1/2,0,0)$ 
and 
$(f_0,f_1,f_2,f_3,f_4,f_5 )=(t,0,0,t,0,0),$  
\newline
(3)\quad 
$(\alpha_0,\alpha_1,\alpha_2,\alpha_3,\alpha_4,\alpha_5)=(1/3,0,0,2/3,0,0)$ 
and 
$(f_0,f_1,f_2,f_3,f_4,f_5 )=(t,0,0,t,0,0).$ 
}
\end{proposition}

\begin{proof}
We treat the case where $(f_0,f_3), (f_0,f_1,f_2,f_4)$ have a pole at $t=\infty$. 
The other cases can be proved in the same way.
\par
If $(f_0,f_3)$ have a pole at $t=\infty$, from Proposition \ref{prop:a5inf}, we get 
$$
(f_0,f_1,f_2,f_3,f_4,f_5)
=
(t,0,0,t,0,0).
$$
\par
If $(f_0,f_1,f_2,f_4)$ have a pole at $t=\infty$, 
it follows 
from Proposition \ref{prop:a5inf} that 
\begin{align*}
&-\mathrm{Res}_{t=\infty}f_0=2\alpha_0-2, & 
&-\mathrm{Res}_{t=\infty}f_1=\alpha_0-1,  & 
&-\mathrm{Res}_{t=\infty}f_2=\alpha_0, \\
&-\mathrm{Res}_{t=\infty}f_3=1-\alpha_0,  & 
&-\mathrm{Res}_{t=\infty}f_4=-3\alpha_0+2, & 
&-\mathrm{Res}_{t=\infty}f_5=0,
\end{align*}
which implies that 
$\alpha_0 \in \frac12 \mathbb{Z}$ or $\alpha_0 \in \frac13 \mathbb{Z}$ 
from Proposition \ref{prop:a5zero} and Corollary \ref{coro:a5res}. 
By $T_1  T_2  T_3$ or $\pi,$ we obtain 
\begin{equation*}
(\alpha_0,0,0,1-\alpha_0,0,0)
\longrightarrow
(0,1,0,0,0,0), \,\,
(1/2,0,0,1/2,0,0), \,
\mathrm{or} \,
(1/3,0,0,2/3,0,0).
\end{equation*}
Therefore, 
the proposition follows from Lemmas \ref{lem:a5Apoint1}, \ref{lem:a5Apoint2}, 
\ref{lem:a5Apoint3} and \ref{lem:a5Apoint4}. 
\end{proof}

\section{The Standard Forms of The Parameters and Rational Solutions of Type B}
In this section, 
we determine the rational solutions of Type B of 
$A_5^{(1)}(\alpha_j)_{0\leq j \leq 5}$ 
if the parameters are the standard forms, 
that is, 
if one of the following occurs:
\newline
(1)\quad $(\alpha_0,\alpha_1,\alpha_2,\alpha_3,\alpha_4,\alpha_5)=(\alpha_0,-\alpha_0+1/2,\alpha_0,-\alpha_0+1/2,0,0)$, 
\newline
(2)\quad $(\alpha_0,\alpha_1,\alpha_2,\alpha_3,\alpha_4,\alpha_5)=(1/2, 0, 1/2, \alpha_0, 0, -\alpha_0),$ 
\newline
(3)\quad $(\alpha_0,\alpha_1,\alpha_2,\alpha_3,\alpha_4,\alpha_5)=(\alpha_0, 0, -\alpha_0+1, 0, 0, 0).$ 
\par
This section consists of four subsections. 
In Subsection 10.1, 
we prove the lemmas in order to 
study the rational solutions of Type B if the parameters are the standard forms. 
In Subsections 10.2, 10.3 and 10.4, 
we determine the rational solutions of Type B of 
$A_5^{(1)}(\alpha_0,-\alpha_0+1/2,\alpha_0,-\alpha_0+1/2,0,0),$ 
$A_5^{(1)}(1/2, 0, 1/2, \alpha_0, 0, -\alpha_0)$ 
and 
$A_5^{(1)}(\alpha_0, 0, -\alpha_0+1, 0, 0, 0),$ 
respectively.

\subsection{Lemmas about rational solutions of Type B}
In this subsection, 
we prove the lemmas in order to 
determine the rational solutions of Type B of 
$A_5^{(1)}(\alpha_0,-\alpha_0+1/2,\alpha_0,-\alpha_0+1/2,0,0)$, 
$A_5^{(1)}(1/2, 0, 1/2, \alpha_0, 0, -\alpha_0)$ 
and 
$A_5^{(1)}(\alpha_0, 0, -\alpha_0+1, 0, 0, 0).$ 

\begin{lemma}
\label{lem:zeropole2}
\it{
For $A_5^{(1)}(1/2, 0,1/2, 0, 0, 0),$ 
there exists a rational solution. 
Then, 
all of $(f_i)_{0\leq i\leq 5}$ are holomorphic at $t=0$, 
or 
$(f_3, f_5)$ have a pole at $t=0$.
}
\end{lemma}

\begin{proof}
By using Proposition \ref{prop:a5zero}, 
we calculate the residues of $f_i\,\,(i=0,1,2,3,4,5)$ at $t=0$ 
when 
$$
(\alpha_0,\alpha_1,\alpha_2,\alpha_3,\alpha_4,\alpha_5)
=
(1/2, 0,1/2, 0, 0, 0).
$$
We then get the lemma.
\end{proof}

\begin{lemma}
\label{lem:bpoint1}
\it{
Suppose that for 
$A_5^{(1)}(1/2, 0,1/2, 0, 0, 0),$ there exists rational solutions 
of Type B. 
Then 
\begin{equation*}
(f_0,f_1,f_2,f_3,f_4,f_5) 
=
(t/2,t/2,t/2,t/2,0,0),  \,\,
(t/2,t/2,t/2,0,0,t/2).
\end{equation*}
}
\end{lemma}

\begin{proof}
Suppose that 
$f_i, f_{i+1}, f_{i+2}, f_{i+3}$ all have a pole at $t=\infty$ 
for some $i=0, 1, 2, 3, 4, 5.$ 
If $(f_0,f_1,f_2,f_3)$ or $(f_5,f_0,f_1,f_2)$ 
have a pole at $t=\infty,$ 
it follows 
from Proposition \ref{prop:a5inf} and Proposition \ref{prop:uniqueness} that 
\begin{equation*}
(f_0,f_1,f_2,f_3,f_4,f_5)
=
(t/2,t/2,t/2,t/2,0,0),  \,\,
(t/2,t/2,t/2,0,0,t/2),
\end{equation*}
respectively.
\par
We assume that $(f_2, f_3, f_4, f_5)$ have a pole at $t=\infty$ 
and show a contradiction. 
The other cases can be proved in the same way. 
\par
From Propositions \ref{prop:a5inf} and \ref{prop:uniqueness}, 
it follows that 
\begin{equation*}
-\mathrm{Res}_{t=\infty} f_2=-1, 
-\mathrm{Res}_{t=\infty} f_3=-1, 
-\mathrm{Res}_{t=\infty} f_4=0, 
-\mathrm{Res}_{t=\infty} f_5=1, 
-\mathrm{Res}_{t=\infty} f_0=1, 
f_1 \equiv 0.
\end{equation*}
It follows 
from Proposition \ref{prop:hinf} that 
$h_{\infty,0}=0.$ 
Moreover, Lemma \ref{lem:zeropole2} shows that 
all of $(f_i)_{0\leq i \leq 5}$ 
are holomorphic at $t=0$ 
or 
$(f_3,f_5)$ have a pole at $t=0.$
\par
If all of $(f_i)_{0\leq i \leq 5}$ are holomorphic 
at $t=0$, it follows from Proposition \ref{prop:h0} that 
$-h_{\infty,0}+h_{0,0}=0.$ 
Thus 
Proposition \ref{prop:hplus} shows that 
all of $(f_i)_{0\leq i \leq 5}$ 
are holomorphic in $\mathbb{C}^{*},$ 
which contradicts the residue theorem. 
\par
If $(f_3, f_5)$ have a pole at $t=0$, 
it follows 
from Proposition \ref{prop:h0} that 
$
-h_{\infty,0}+h_{0,0}=1/6, 
$
which contradicts Proposition \ref{prop:hplus}.
\end{proof}

In order to determine the rational solutions of Type B 
of 
$A_5^{(1)}(\alpha_0, 0, -\alpha_0+1, 0, 0, 0),$  
we have the following two lemmas:

\begin{lemma}
\label{lem:bpoint2}
\it{
For 
$A_5^{(1)}(1, 0, 0, 0, 0, 0),$ there exists no rational solution 
of Type B.
}
\end{lemma}

\begin{proof}

We treat the case where $(f_0, f_1, f_2, f_3)$ have a pole at $t=\infty$. 
The other cases can be proved in the same way.
\par
From 
Propositions \ref{prop:a5inf} and \ref{prop:uniqueness}, 
it follows that 
\begin{equation}
\label{eqn:bpoint}
-\mathrm{Res}_{t=\infty} f_0 =0, 
-\mathrm{Res}_{t=\infty} f_1 =-1, 
-\mathrm{Res}_{t=\infty} f_2 =0, 
-\mathrm{Res}_{t=\infty} f_3 =1, 
f_4 \equiv f_5 \equiv 0.
\end{equation}
Lemma \ref{lem:zeropole1} shows that 
$(f_1,f_3), (f_3,f_5),(f_5,f_1),(f_2,f_4,f_5,f_1)$ 
can have a pole at $t=0.$ 
Since $f_4 \equiv f_5 \equiv 0,$ 
all of 
$(f_i)_{0\leq i \leq 5}$ are holomorphic at $t=0$ 
or 
$(f_1,f_3)$ have a pole at $t=0$.
\par
Suppose that all of 
$(f_j)_{0\leq j \leq 5}$ are holomorphic at $t=0$. 
It then follows from Propositions \ref{prop:hinf} and \ref{prop:h0} that 
$-h_{\infty,0}+h_{0,0}=-1/6.$ 
Suppose that 
$\pm c_1, \pm c_2,\ldots, \pm c_n \in \mathbb{C}^{*}$ are poles of $f_i$ for some $i=0,1,2,3,4,5$ 
because $f_j \,\,(j=0,1,2,3,4,5)$ are odd functions from Corollary \ref{prop:a5odd}. 
Let $\epsilon_k c_k \,\,(k=1,2,\ldots,n)$ be the residue of $H$ at $t=c_k$. 
It then follows 
from the proof of Proposition \ref{prop:hplus} that 
$
-h_{\infty,0}+h_{0,0}=-2 \sum_{k=1}^n \epsilon_k.
$
Proposition \ref{prop:hc} shows that $\epsilon_k \,\,(k=1,2,\ldots, n)$ are 
$1/6$ or $1/12$ or $5/12.$ 
Since $-h_{\infty,0}+h_{0,0}=-1/6,$ 
$H$ has poles at $t= \pm c$ for some $c \in \mathbb{C}^{*}$ 
and 
the residues of $H$ at $t=\pm c$ are $\pm c/12.$ 
It then follows from Proposition \ref{prop:hc} that 
$(f_1, f_3)(II)$ occurs 
for $t=\pm c \in \mathbb{C}^{*}$, 
because $f_4=f_5\equiv 0.$ 
Therefore, we have 
\begin{equation*}
f_0=\frac{t}{2} ,  \,
f_1=\frac{t}{2} -\frac12 \frac{1}{t-c} -\frac12 \frac{1}{t+c},  \,
f_2=\frac{t}{2},  \,
f_3=\frac{t}{2} +\frac12 \frac{1}{t-c} +\frac12 \frac{1}{t+c}, \,
f_4 \equiv f_5 \equiv 0. 
\end{equation*}
On the other hand, 
from Proposition \ref{prop:a5c}, 
it follows 
that 
the constant term of the Taylor series of $f_2$ at $t= \pm c$ 
is zero, 
which is a contradiction because $f_2(\pm c)=\pm c/2.$
\par
When $(f_1, f_3)$ have a pole at $t=0$, 
it follows 
from Lemma \ref{prop:h0} that 
$
-h_{\infty, 0}+h_{0,0}
=
-1/6+1/6=0.
$ 
From Proposition \ref{prop:hc}, 
all of $(f_i)_{0\leq i \leq 5}$ are then holomorphic in $\mathbb{C}^{*}$. 
Thus, 
it follows 
from equation (\ref{eqn:bpoint}) and the residue theorem 
that 
$
f_0=t/2,\,  
f_1=t/2 -1/t, \,
f_2=t/2, \,
f_3=t/2 +1/t, \,
f_4 \equiv f_5 \equiv 0.
$
By substituting this solution into $A_5^{(1)}(1, 0, 0, 0, 0, 0)$, 
we can prove the contradiction.
\end{proof}

The following lemma can be proved in the same way.

\begin{lemma}
\label{lem:bpoint3}
\it{
For 
$
A_5^{(1)}
(1/2,1/2,0,0,0,0) \, and
\,
A_5^{(1)} 
(1/2,0,0,1/2,0,0),
$ 
there exists no
rational solution of Type B.
}
\end{lemma}

\subsection{Rational solutions of Type B of $A_5^{(1)}(\alpha_0, -\alpha_0+1/2, \alpha_0, -\alpha_0+1/2, 0, 0)$}

In this subsection, we determine 
the rational solutions of Type B of $A_5^{(1)}(\alpha_0, -\alpha_0+1/2, \alpha_0, -\alpha_0+1/2, 0, 0).$

\begin{proposition}
\it{
\label{prop:bpara1}
Suppose that 
for 
$A_5^{(1)}(\alpha_0, -\alpha_0+1/2, \alpha_0, -\alpha_0+1/2, 0, 0),$ 
there exists a rational solution of Type B. 
Then, 
$$
(f_0,f_1, f_2, f_3,f_4,f_5)=(t/2,t/2,t/2,t/2,0,0),
$$ 
or 
by some B\"acklund transformations, 
the parameters and solution can be transformed so that 
one of the following occurs:
\newline
(1)\quad 
$(\alpha_0,\alpha_1,\alpha_2,\alpha_3,\alpha_4,\alpha_5)=(1/2,0,1/2,0,0,0)$ 
and 
$
(f_0,f_1,f_2,f_3,f_4,f_5)=
(
t/2,
t/2,
t/2,
t/2,
0,0
)
$,
\newline
(2)\quad 
$(\alpha_0,\alpha_1,\alpha_2,\alpha_3,\alpha_4,\alpha_5)=(0,1/2,0,1/2,0,0)$ 
and 
$
(f_0,f_1,f_2,f_3,f_4,f_5)=
(
t/2,
t/2,
t/2,
t/2,
0,0
).
$
}
\end{proposition}

\begin{proof}
We treat the case where $(f_0,f_1,f_2,f_3)$ or $(f_1,f_2,f_3,f_4)$ have a pole at $t=\infty$.
The other cases can be proved in the same way.
\par
If $(f_0,f_1,f_2,f_3)$ have a pole at $t=\infty$, 
it follows 
from Propositions \ref{prop:a5inf} and \ref{prop:uniqueness} that 
$$
(f_0,f_1,f_2,f_3,f_4,f_5)
=
(
t/2,
t/2,
t/2,
t/2,
0,0
).
$$
\par
If $(f_1,f_2,f_3,f_4)$ have a pole at $t=\infty$, 
it follows 
from Propositions \ref{prop:a5inf} and \ref{prop:uniqueness} that 
\begin{equation*}
-\mathrm{Res}_{t=\infty} f_1=-\mathrm{Res}_{t=\infty} f_2=-\mathrm{Res}_{t=\infty} f_3=0, 
-\mathrm{Res}_{t=\infty} f_4=2\alpha_0, f_5 \equiv 0, 
-\mathrm{Res}_{t=\infty} f_0=-2\alpha_0.
\end{equation*}
Moreover, when $(f_4,f_0)$ have a pole at $t=0$, 
we have 
$
-h_{\infty,0}+h_{0,0}=
-
\left(
-1/3 \cdot \alpha_0
\right)
+
\left(
-1/3 \cdot \alpha_0+1/6
\right)
=1/6,
$
which contradicts Proposition \ref{prop:hplus}. 
When $(f_4,f_0)$ do not have a pole at $t=0$, 
it follows 
from Corollary \ref{coro:a5res} that 
$2\alpha_0 \in  \mathbb{Z}$. 
\par
When $\alpha_0 \in \mathbb{Z},$ 
by $\pi^{-1}  T_1  T_3$, 
we obtain 
$
(\alpha_0, -\alpha_0+1/2, \alpha_0, -\alpha_0+1/2, 0, 0)
\longrightarrow
(1/2,0,1/2,0,0,0). 
$
Therefore, 
the proposition follows 
from Lemma \ref{lem:bpoint1}. 
\par
When $\alpha_0-1/2 \in  \mathbb{Z}$, 
by $T_1  T_3$, we have 
$
(\alpha_0, -\alpha_0+1/2, \alpha_0, -\alpha_0+1/2, 0, 0)
\longrightarrow
(1/2,0,1/2,0,0,0). 
$
Therefore, 
the proposition follows 
from Lemma \ref{lem:bpoint1}. 
\end{proof}

\subsection{Rational solutions of Type B of $A_5^{(1)}(1/2, 0, 1/2,\alpha_0, 0, -\alpha_0)$}
In this subsection, 
we determine the rational solutions of Type B 
of 
$A_5^{(1)}(1/2, 0, 1/2,\alpha_0, 0, -\alpha_0)$.

\begin{proposition}
\it{
\label{prop:bpara2}
Suppose that 
for $A_5^{(1)}(1/2, 0, 1/2,\alpha_0, 0, -\alpha_0),$ 
there exists a rational solution of Type B. 
$2\alpha_0$ is then an integer. 
Furthermore, 
by some B\"acklund transformations, 
the parameters and solution can be transformed so that 
either of the following occurs:
\newline
(1)\quad 
$(\alpha_0,\alpha_1,\alpha_2,\alpha_3,\alpha_4,\alpha_5)=(1/2,0,1/2,0,0,0)$ 
and 
$
(f_0,f_1,f_2,f_3,f_4,f_5)=
(
t/2,
t/2,
t/2,
t/2,
0,0
)
$,
\newline
(2)\quad 
$(\alpha_0,\alpha_1,\alpha_2,\alpha_3,\alpha_4,\alpha_5)=(0,1/2,0,1/2,0,0)$ 
and 
$
(f_0,f_1,f_2,f_3,f_4,f_5)=
(
t/2,
t/2,
t/2,
t/2,
0,0
).
$
}
\end{proposition}

\begin{proof} 
We treat the case where $(f_1, f_2, f_3, f_4)$ have a pole at $t=\infty$. 
The other cases can be proved in the same way. 
\par
From Proposition \ref{prop:a5inf}, 
it follows that 
\begin{align*}
&-\mathrm{Res}_{t=\infty} f_1 =  2\alpha_0,           & 
&-\mathrm{Res}_{t=\infty} f_2 =  2\alpha_0,           & 
&-\mathrm{Res}_{t=\infty} f_3 =  0,                      \\
&-\mathrm{Res}_{t=\infty} f_4 =-2\alpha_0+1,          & 
&-\mathrm{Res}_{t=\infty} f_5=-2\alpha_0,             & 
&-\mathrm{Res}_{t=\infty}f_0 =-1.
\end{align*}
\par
When $(f_2, f_4, f_5, f_1)$ have a pole at $t=0$, 
it follows 
from Propositions \ref{prop:hinf} and \ref{prop:h0} that 
\begin{equation*}
-h_{\infty,0}+h_{0,0}=-\left(\alpha_0^2-1/2 \cdot \alpha_0-1/6 \right)+
\left( \alpha_0^2-1/2 \cdot \alpha_0+1/6 \right)
=1/3,
\end{equation*}
which contradicts Proposition \ref{prop:hplus}.
\par
Then we can suppose that $(f_2, f_4, f_5, f_1)$ do not have a pole at $t=0,$ 
which implies that 
from Proposition \ref{prop:a5zero} and Corollary \ref{coro:a5res}, 
$2\alpha_0 \in  \mathbb{Z}.$ 
If $\alpha_0 \in \mathbb{Z}$, 
by $T_4  T_5$, we obtain 
$
(1/2, 0, 1/2, \alpha_0, 0, -\alpha_0) 
\rightarrow
(1/2, 0, 1/2, 0, 0, 0).
$
Therefore, 
the proposition follows 
from Lemma \ref{lem:bpoint1}. 
\par
If $\alpha_0 -1/2 \in \mathbb{Z}$, 
by $T_4  T_5$, we have 
$
(1/2, 0, 1/2, \alpha_0, 0, -\alpha_0) 
\rightarrow
(1/2, 0, 1/2, 1/2, 0, -1/2).
$
By $\pi^{-2}  s_4  s_5$, 
we get 
$
(1/2, 0, 1/2, 1/2, 0, -1/2)
\rightarrow
(1/2, 0, 1/2, 0, 0, 0).
$
Therefore, 
the proposition follows 
from Lemma \ref{lem:bpoint1}. 
\end{proof}

\subsection{Rational solutions of Type B of $A_5^{(1)}(\alpha_0,0,-\alpha_0+1, 0, 0, 0)$}

In this subsection, 
we determine 
the rational solutions of Type B of $A_5^{(1)}(\alpha_0,0,-\alpha_0+1, 0, 0, 0).$

\begin{proposition}
\label{prop:bpara3}
\it{
Suppose that 
$A_5^{(1)}(\alpha_0,0,-\alpha_0+1, 0, 0, 0)$ 
has a rational solution of Type B. 
$2\alpha_0$ is then an integer. Furthermore, 
by some B\"acklund transformations, 
the parameters and solution can be transformed so that 
either of the following occurs:
\newline
(1)\quad 
$(\alpha_0,\alpha_1,\alpha_2,\alpha_3,\alpha_4,\alpha_5)=(1/2,0,1/2,0,0,0)$ 
and 
$
(f_0,f_1,f_2,f_3,f_4,f_5)=
(
t/2,
t/2,
t/2,
t/2,
0,0
)
$,
\newline
(2)\quad 
$(\alpha_0,\alpha_1,\alpha_2,\alpha_3,\alpha_4,\alpha_5)=(0,1/2,0,1/2,0,0)$ 
and 
$
(f_0,f_1,f_2,f_3,f_4,f_5)=
(
t/2,
t/2,
t/2,
t/2,
0,0
).
$
}
\end{proposition}

\begin{proof}
We treat the case where $(f_0, f_1, f_2, f_3)$ have a pole at $t=\infty$. 
The other cases can be proved in the same way. 
\par
From Propositions \ref{prop:a5inf} and \ref{prop:uniqueness}, 
it follows that 
\begin{equation*}
-\mathrm{Res}_{t=\infty} f_0 =0, 
-\mathrm{Res}_{t=\infty} f_1 =-2\alpha_0+1, 
-\mathrm{Res}_{t=\infty} f_2 =0, 
-\mathrm{Res}_{t=\infty} f_3 =2\alpha_0-1, 
f_4 \equiv f_5 \equiv 0.
\end{equation*}
\par
When $(f_1, f_3)$ have a pole at $t=0$, 
it follows 
from Propositions \ref{prop:a5zero}, 
\ref{prop:hinf} and 
\ref{prop:h0} that 
\begin{equation*}
-h_{\infty,0}+h_{0,0}
=
-\left(\alpha_0^2-7/6\cdot\alpha_0+1/3 \right)
+\left(\alpha_0^2-7/6\cdot\alpha_0+1/3 \right)
=0,
\end{equation*}
which implies that 
all of 
$(f_i)_{0\leq i \leq 5}$ are holomorphic in $\mathbb{C}^{*}$ 
from Proposition \ref{prop:hplus}.
Therefore, 
from Proposition \ref{prop:a5zero}, 
it follows 
that 
\begin{equation*}
f_0=t/2, \,
f_1=t/2 +t^{-1}(-2\alpha_0+1), \,
f_2=t/2, \,
f_3=t/2  +t^{-1}(2\alpha_0-1), \, 
f_4\equiv f_5 \equiv 0.
\end{equation*} 
By substituting this solution into $A_5^{(1)}(\alpha_0,0,-\alpha_0+1,0,0,0)$, 
we get $\alpha_0=1/2$, which is contradiction. 
\par
When $(f_1, f_3)$ do not have a pole at $t=0$, 
it follows 
from Corollary \ref{coro:a5res} that 
$2\alpha_0 \in  \mathbb{Z}.$ 
If $\alpha_0 \in \mathbb{Z}$, by $\pi^{-2}  T_1  T_2$, 
we have 
$
(\alpha_0,0,-\alpha_0+1,0,0,0)
\rightarrow
(1,0,0,0,0,0).
$ 
Therefore, 
for $A_5^{(1)}(1,0,0,0,0,0),$ there exists a rational solution Type B, 
which contradicts Lemma \ref{lem:bpoint2}.
\par
If $\alpha_0 -1/2\in  + \mathbb{Z}$, by $T_1  T_2$, 
we get 
$
(\alpha_0,0,-\alpha_0+1,0,0,0)
\rightarrow
(1/2,0,1/2,0,0,0).
$ 
Therefore,  
the proposition follows from Lemma \ref{lem:bpoint1}. 
\end{proof}

\section{The Standard Forms of The Parameters and Rational Solutions of Type C}
In this section, 
we classify the rational solutions of Type C of  
$A_5^{(1)}(\alpha_j)_{0\leq j \leq 5}$ if the parameters are the standard forms, 
that is, 
if one of the following occurs:
\newline
(1) \quad $
(\alpha_0,\alpha_1,\alpha_2,\alpha_3,\alpha_4,\alpha_5)=
(
\alpha_4, 
-\alpha_4+1/3, 
\alpha_4, 
-\alpha_4+1/3, 
\alpha_4, 
-\alpha_4+1/3
),
$ 
\newline
(2) \quad  $(\alpha_0,\alpha_1,\alpha_2,\alpha_3,\alpha_4,\alpha_5)=(-\alpha_4+1/3,1/3,1/3,\alpha_4,0,0)$, 
\newline
(3) \quad  $(\alpha_0,\alpha_1,\alpha_2,\alpha_3,\alpha_4,\alpha_5)=(\alpha_4,0,0,1-\alpha_4,0,0)$, 
\newline
(4) \quad  
$(\alpha_0,\alpha_1,\alpha_2,\alpha_3,\alpha_4,\alpha_5)=(\alpha_4,1/3,1/3,-\alpha_4+1/3,0,0).$ 
\par 
This section consists of five subsections. 
In Subsection 11.1, 
we prove the lemmas for the classifications. 
In Subsections 11.2, 11.3, 11.4 and 11.5, 
we determine 
the rational solutions of Type C of 
$
A_5^{(1)}
(
\alpha_4, 
-\alpha_4+1/3, 
\alpha_4, 
-\alpha_4+1/3, 
\alpha_4, 
-\alpha_4+1/3
),
$ 
$A_5^{(1)}(-\alpha_4+1/3,1/3,1/3,\alpha_4,0,0)$, 
$A_5^{(1)}(\alpha_4,0,0,1-\alpha_4,0,0)$, 
and 
$A_5^{(1)}(\alpha_4,1/3,1/3,-\alpha_4+1/3,0,0)$ 
respectively.

\subsection{Lemmas about rational solutions of Type C}
In this subsection, 
we prove the lemmas in order to  
classify the rational solutions of Type C for the standard forms.

\begin{lemma}
\it{
For 
$A_5^{(1)}(1,0,0,0,0,0),$ there exists no rational solution 
of Type C.
}
\end{lemma}

\begin{proof}
Suppose that for  
$A_5^{(1)}(1,0,0,0,0,0),$ 
there exists a rational solution of Type C. 
From Lemma \ref{lem:zeropole1}, 
it follows that 
all of $(f_i)_{0\leq i \leq 5}$ are holomorphic at $t=0$ 
or $(f_1,f_3), \, (f_3,f_5), \, (f_5,f_1), \, (f_2,f_4,f_5,f_1)$ 
have a pole at $t=0$. 
We prove that 
the proposition is true 
if all of $(f_i)_{0\leq i \leq 5}$ are holomorphic at $t=0.$ 
The other cases can be proved in the same way.
\par
We assume that all of $(f_i)_{0\leq i \leq 5}$ are holomorphic at $t=0.$ 
From Proposition \ref{prop:a5inf}, 
it follows that 
\begin{equation}
\label{eqn:cpointres1}
-\mathrm{Res}_{t=\infty} f_0=0, 
-\mathrm{Res}_{t=\infty} f_1=-2,
-\mathrm{Res}_{t=\infty} f_2=-1, 
\end{equation}
and
\begin{equation}
\label{eqn:cpointres2}
-\mathrm{Res}_{t=\infty} f_3=0, 
-\mathrm{Res}_{t=\infty} f_4=1, 
-\mathrm{Res}_{t=\infty} f_5=2.
\end{equation}
From Propositions \ref{prop:hinf} and \ref{prop:h0}, 
it follows that 
$
-h_{\infty,0}+h_{0,0}=-2/3.
$ 
We show that this contradicts 
Proposition \ref{prop:a5c}. 
\par
Let $\pm c_1, \pm c_2, \ldots, \pm c_n \in \mathbb{C}^{*}$ be poles 
of $(f_j)_{0\leq j \leq 5}.$ 
From the proof of Proposition \ref{prop:hplus}, 
it follows that 
$$
-h_{\infty,0}+h_{0,0}=
-2 
\sum_{k=1}^n \epsilon_k, 
$$ 
where $\pm \epsilon_k c_k \,\,(1\leq k \leq n)$ 
are the residues of $H$ at $t=\pm c_k,$ respectively 
and 
$\epsilon_k =1/6, 1/12, 5/12.$ 
We then consider the following two cases:
\newline
(1)\quad $n=2, \,\, \epsilon_1=1/6,$ 
\newline
(2)\quad $n=4, \,\, 
\epsilon_1=\epsilon_2=\epsilon_3=\epsilon_4=1/12,$
\newline
(3)\quad $n=3, \,\,
\epsilon_1=1/6, \,\,\epsilon_2=\epsilon_3=1/12.$ 
\par
If case (1) occurs, it follows 
from Proposition \ref{prop:hc} that 
$(f_i,f_{i+2})(I)$ or 
$(f_i,f_{i+2},f_{i+3},f_{i+5})(I)$ 
can occur for some $i=0,1,2,3,4,5$ and $\pm c \in \mathbb{C}^{*}.$ 
In case of $(f_i,f_{i+2})(I)$ and $(f_i,f_{i+2},f_{i+3},f_{i+5})(I)$, 
the residues $(f_j)_{0\leq j \leq 5}$ at $t=\pm c_1, \pm c_2$ are $1/2$ or $-1/2,$ 
which contradicts equations (\ref{eqn:cpointres1}), (\ref{eqn:cpointres2}) and the residue theorem. 
\par
If case (2) occurs, it follows from Proposition \ref{prop:hc} that 
$(f_{i_1}, f_{i_{1}+2})(II)$, 
$(f_{i_2}, f_{i_{2}+2})(II)$, 
$(f_{i_3}, f_{i_{3}+2})(II)$, 
$(f_{i_4}, f_{i_{4}+2})(II)$ 
can occur for some $i_1,i_2,i_3,i_4=0,1,2,3,4,5$ 
and 
$\pm c_1, \pm c_2, \pm c_3, \pm c_4 \in \mathbb{C}^{*}.$ 
In this case, 
the residues $(f_j)_{0\leq j \leq 5}$ at $t=\pm c_1, \pm c_2, \pm c_3, \pm c_4 \in \mathbb{C}^{*}$ 
are $1/2$ or $-1/2,$ 
which contradicts equations (\ref{eqn:cpointres1}), (\ref{eqn:cpointres2}) and the residue theorem. 
\par
If case (3) occurs, it follows from Proposition \ref{prop:hc} that 
either of the following occurs:
\newline
(i)\quad $(f_{i_1}, f_{i_{1}+2})(I), (f_{i_2}, f_{i_{2}+2})(II), (f_{i_3}, f_{i_{3}+2})(II)$ for some $i_1,i_2,i_3=0,1,2,3,4,5,$ 
\newline
(ii)\quad $(f_{i_1},f_{i_1+2},f_{i_1+3},f_{i_1+5})(I), (f_{i_2}, f_{i_{2}+2})(II), (f_{i_3}, f_{i_{3}+2})(II)$ for some $i_1,i_2,i_3=0,1,2,3,4,5.$ 
\newline
In both cases (i) and (ii), 
the residues of $(f_j)_{0\leq j \leq 5}$ at $t=\pm c_1, \pm c_2, \pm c_3 \in \mathbb{C}^{*}$ 
are $1/2$ or $-1/2,$ 
which contradicts equations (\ref{eqn:cpointres1}), (\ref{eqn:cpointres2}) and the residue theorem. 

\end{proof}

\begin{lemma}
\label{lem:cpoint1}
\it{
For 
$A_5^{(1)}(1/3,1/3,1/3,0,0,0),$ 
there exists no 
rational solution of Type C.
}
\end{lemma}

\begin{proof}
Suppose that for $A_5^{(1)}(1/3,1/3,1/3,0,0,0)$ 
there exists a rational solution of Type C. 
From Propositions \ref{prop:hinf} and \ref{prop:h0}, 
it follows that 
$
-h_{\infty,0}+h_{0,0}
=
-4/27, 
-1/27,
-5/54, 
1/54, 
1/18,
-1/18,
$ 
which contradicts 
Proposition \ref{prop:hplus}. 
\end{proof}

\begin{lemma}
\it{
For 
$A_5^{(1)}(1/3,0,0,2/3,0,0),$ 
there exists no 
rational solution of Type C.
}
\end{lemma}

\begin{proof}
It can be proved in the same way as Lemma \ref{lem:cpoint1}.
\end{proof}

\subsection{Rational solutions of Type C of $
A_5^{(1)}
(
\alpha_4, 
-\alpha_4+1/3, 
\alpha_4, 
-\alpha_4+1/3, 
\alpha_4, 
-\alpha_4+1/3
)
$} 
In this subsection, we determine the 
rational solutions of Type C of 
$
A_5^{(1)}
(
\alpha_4, 
-\alpha_4+1/3, 
\alpha_4, 
-\alpha_4+1/3, 
\alpha_4, 
-\alpha_4+1/3
).
$

\begin{proposition}
\label{prop:cpara1} 
\it{
Suppose that 
for 
$
A_5^{(1)}
(
\alpha_4, 
-\alpha_4+1/3, 
\alpha_4, 
-\alpha_4+1/3, 
\alpha_4, 
-\alpha_4+1/3
),
$ 
there exists a rational solution 
of Type C. Then,   
\begin{equation*}
(f_0,f_1,f_2,f_3,f_4,f_5)
=
(t/3,t/3,t/3,t/3,t/3,t/3),
\end{equation*}
and it is unique.
}
\end{proposition}

\begin{proof}
The proposition follows from 
Proposition \ref{prop:a5inf}.

\end{proof}

\subsection{Rational solutions of Type C of 
$
A_5^{(1)}
(-\alpha_4+1/3,1/3,1/3,\alpha_4,0,0)
$
}

In this subsection, 
we determine the rational solutions of Type C of  
$
A_5^{(1)}
(-\alpha_4+1/3,1/3,1/3,\alpha_4,0,0).
$

\begin{proposition}
\label{prop:cpara2} 
\it{
For $
A_5^{(1)}
(-\alpha_4+1/3,1/3,1/3,\alpha_4,0,0)
$ 
there exists no rational solution of Type C.
}
\end{proposition}

\begin{proof}
Suppose that 
for $
A_5^{(1)}
(-\alpha_4+1/3,1/3,1/3,\alpha_4,0,0),
$ 
there exists a rational solution of Type C. 
From Proposition \ref{prop:a5inf}, 
it follows that 
\begin{align*}
&-\mathrm{Res}_{t=\infty} f_0=1,         & 
&-\mathrm{Res}_{t=\infty} f_1=3\alpha_4, & 
&-\mathrm{Res}_{t=\infty} f_2=3\alpha_4-1, \\
&-\mathrm{Res}_{t=\infty} f_3=-1,        & 
&-\mathrm{Res}_{t=\infty} f_4=-3\alpha_4,& 
&-\mathrm{Res}_{t=\infty} f_5=-3\alpha_4+1.
\end{align*}
\par
When $(f_2,f_4,f_5,f_1)$ have a pole at $t=0$, 
it follows 
from Propositions \ref{prop:hinf} and \ref{prop:h0} that 
\begin{equation*}
-h_{\infty,0}+h_{0,0}
=
-\left(2\alpha_0^2-2/3\cdot\alpha_4+2/9 \right)
+\left(2\alpha_0^2-2/3 \cdot\alpha_4+2/9 \right)
=0,
\end{equation*}
which implies that 
all of 
$(f_i)_{0\leq i \leq 5}$ are holomorphic in $\mathbb{C}^{*}$ 
from Proposition \ref{prop:hplus}.
On the other hand, from Proposition \ref{prop:a5zero}, 
it follows that 
$
-\mathrm{Res}_{t=\infty} f_4
-\mathrm{Res}_{t=0} f_4
=
-1, 
-\mathrm{Res}_{t=\infty} f_5
-\mathrm{Res}_{t=0} f_5
=
1,
$
which contradicts the residue theorem.
\par
When $(f_2,f_4,f_5,f_1)$ do not have a pole at $t=0$, 
we have $3\alpha_4 \in \mathbb{Z}$. 
If $\alpha_4 \in \mathbb{Z},$ 
by $T_1  T_2  T_3,$ 
we get
\begin{equation*}
(-\alpha_4+1/3,1/3,1/3,\alpha_4,0,0)
\longrightarrow
(1/3,1/3,1/3,0,0,0).
\end{equation*} 
If $\alpha_4 -1/3 \in \mathbb{Z},$ 
by $T_1  T_2  T_3,$ and $\pi$, 
we obtain 
$
(-\alpha_4+1/3,1/3,1/3,\alpha_4,0,0)
\longrightarrow
(1/3,1/3,1/3,0,0,0).
$
If $\alpha_4 +1/3 \in \mathbb{Z},$
by $T_3$ and $s_1  s_2$, 
we have 
$
(-\alpha_4+1/3,1/3,1/3,\alpha_4,0,0)
\longrightarrow
(1/3,1/3,1/3,0,0,0).
$
Therefore, 
the proposition follows 
from Lemma \ref{lem:cpoint1}.
\end{proof}

\subsection{Rational solutions of Type C of 
$A_5^{(1)}(\alpha_4,0,0,1-\alpha_4,0,0)$
}
In this subsection, 
we determine the rational solutions of Type C of 
$A_5^{(1)}(\alpha_4,0,0,1-\alpha_4,0,0).$

\begin{proposition}
\label{prop:cpara3} 
\it{
For 
$A_5^{(1)}(\alpha_4,0,0,1-\alpha_4,0,0),$ 
there exists no rational solution of Type C.
}
\end{proposition}

\begin{proof}
It 
can be proved in the same way as 
Proposition \ref{prop:cpara2}.
\end{proof}

\subsection{Rational solutions of Type C of 
$A_5^{(1)}(\alpha_4,1/3,1/3,-\alpha_4+1/3,0,0)$ 
}
In this subsection, 
we determine the rational solutions of Type C of 
$A_5^{(1)}(\alpha_4,1/3,1/3,-\alpha_4+1/3,0,0).$

\begin{proposition}
\label{prop:cpara4} 
\it{
For 
$A_5^{(1)}(\alpha_4,1/3,1/3,-\alpha_4+1/3,0,0),$ 
there exists no rational solution of Type C.
}
\end{proposition}

\begin{proof}
It 
can be proved in the same way as 
Proposition \ref{prop:cpara2}.
\end{proof}

\section{Main Theorems for Type A, Type B and Type C}
In this section, 
we obtain the main theorems for the rational solutions of Types A, B and C of 
$A_5^{(1)}(\alpha_j)_{0\leq j \leq 5}. $

\subsection{Complete classification of rational solutions of Type A}
In this subsection, we 
classify 
the rational solutions of Type A 
of $A_5^{(1)}(\alpha_j)_{0\leq j \leq 5}.$

\begin{theorem}
\label{thm:a5mainA}
\it{
For a rational solution of Type A of $A_5^{(1)}(\alpha_j)_{0\leq j \leq 5},$ 
by some B\"acklund transformations, 
the parameters and solution 
can be transformed so that one of the following occurs:
\newline
(a-1)\quad $(\alpha_0,\alpha_1,\alpha_2,\alpha_3,\alpha_4,\alpha_5)=(\alpha_0,1-\alpha_0,0,0,0,0),$ 
and 
$
(f_0,f_1,f_2,f_3,f_4,f_5)=(t,t,0,0,0,0),
$
\newline
(a-2)\quad $(\alpha_0,\alpha_1,\alpha_2,\alpha_3,\alpha_4,\alpha_5)=(\alpha_0,0,0,1-\alpha_0,0,0),$ 
and 
$
(f_0,f_1,f_2,f_3,f_4,f_5)=(t,0,0,t,0,0),
$
\newline
(a-3)\quad $(\alpha_0,\alpha_1,\alpha_2,\alpha_3,\alpha_4,\alpha_5)=(0,1,0,0,0,0),$ 
and 
\begin{align*}
(f_0, f_1, f_2, f_3, f_4, f_5)
&=
(t,t,0,0,0,0), \,\,
(0,t,t,0,0,0), \,\,
(0,t,0,0,t,0),  \,\,
(t,t,t,0,-t,0).
\end{align*}
The orbit of the parameters in cases (a-1), (a-2) and (a-3) 
by the B\"acklund transformation group $\tilde{W}(A_5^{(1)})$ 
consists of the parameters which satisfy 
one of the following five conditions:
\newline
(1) \quad for some $i=0,1,2,3,4,5,$
$
\alpha_{i+2}, \alpha_{i+3}, 
\alpha_{i+4}, \alpha_{i+5} 
\in \mathbb{Z};
$
\newline
(2) \quad for some $i=0,1,2,3,4,5,$
$
\alpha_{i+1}, \alpha_{i+2}, 
\alpha_{i+4}, \alpha_{i+5} 
\in \mathbb{Z};
$
\newline
(3) \quad for some $i=0,1,2,3,4,5,$
$
\alpha_{i+3}, \alpha_{i+5}, 
\alpha_i+\alpha_{i+4}, 
\alpha_i-\alpha_{i+2} 
\in \mathbb{Z};
$
\newline
(4) \quad for some $i=0,1,2,3,4,5,$
$
\alpha_{i+3}+\alpha_{i+4}, 
\alpha_{i+4}+\alpha_{i+5}, 
\alpha_i+\alpha_{i+1}, 
\alpha_i-\alpha_{i+4} 
\in \mathbb{Z};
$
\newline
(5) \quad for some $i=0,1,2,3,4,5,$
$
\alpha_{i}+\alpha_{i+1}, 
\alpha_{i}+\alpha_{i+5}, 
\alpha_{i+2}+\alpha_{i+3}, 
\alpha_{i+3}+\alpha_{i+4}, 
\alpha_i+\alpha_{i+3} \in \mathbb{Z}. 
$

}
\end{theorem}

\begin{proof}
Proposition \ref{prop:necA}  
shows that 
one of 
cases 
(1), (2), $\ldots,$ (5) occurs
if for $A_5^{(1)}(\alpha_j)_{0\leq j \leq 5},$ 
there exists a rational solution of Type A. 
Proposition \ref{prop:standardA} 
proves that 
if for $A_5^{(1)}(\alpha_i)_{0\leq i \leq 5},$ 
there exists a rational solution of Type A, 
the parameters can be transformed into the two standard forms, 
$
(\alpha_0, 1-\alpha_0, 0, 0, 0,0) 
$ 
and 
$
(\alpha_0, 0, 0, 1-\alpha_0, 0,0).
$ 
Propositions
\ref{prop:a5A1} and \ref{prop:a5A2} 
prove that 
$A_5^{(1)}(\alpha_0, 1-\alpha_0, 0, 0, 0,0) $ 
and 
$A_5^{(1)}(\alpha_0, 0, 0, 1-\alpha_0, 0,0)$ 
have rational solutions of (a-1) and (a-2), 
and 
show that case (a-3) happens when $\alpha_0\in\mathbb{Z}$. 

\end{proof}

{\bf Remark}
\newline
The rational solutions 
$(t,t,0,0,0,0), \,(t,0,0,t,0,0)$ 
correspond to the rational solutions of the fifth Painlev\'e equation.

\subsection{Complete classification of rational solutions of Type B}
In this subsection, we 
classify 
the rational solutions of Type B 
of $A_5^{(1)}(\alpha_j)_{0\leq j \leq 5}.$ 
For this purpose, 
we have the following lemma:

\begin{lemma}
\label{lem:bpointgeneral}
\it{
Suppose that $2 \alpha_j \in \mathbb{Z} \,\,(0\leq j \leq 5)$. 
By some B\"acklund transformations, 
the parameters 
can then be transformed into 
$(1/2,0,1/2,0,0,0),$ 
$(1/2,1/2,0,0,0,0),$ 
$(1/2,0,0,1/2,0,0).$
Especially, 
the parameters 
are transformed into $(1/2,0,1/2,0,0,0)$ 
if and only if 
for some $i=0,1,2,3,4,5,$ 
\begin{align*}
(\alpha_i,\alpha_{i+1},\alpha_{i+2},\alpha_{i+3},\alpha_{i+4},\alpha_{i+5})
&\equiv
(1/2,1/2,1/2,1/2,0,0), \,
(1/2,1/2,0,1/2,1/2,0), \,  \\
&\quad\,\,(1/2,0,1/2,0,0,0) \, \mathrm{mod} \, \mathbb{Z}. 
\end{align*}
}
\end{lemma}

\begin{proof}
Since $\sum_{k=0}^5 \alpha_k =1,$ 
it follows that 
\begin{align*}
(\alpha_i,\alpha_{i+1},\alpha_{i+2},\alpha_{i+3},\alpha_{i+4},\alpha_{i+5})  
\equiv                                                                        
&(1/2,1/2,1/2,1/2,1/2,1/2),                            
(1/2,1/2,1/2,1/2,0,0),   \\                                      
&(1/2,1/2,0,1/2,1/2,0), \,                                                                      
(1/2,1/2,1/2,0,1/2,0),    \\                                    
&(1/2,1/2,0,0,0,0),                                                  
(1/2,0,1/2,0,0,0), \\
&(1/2,0,0,1/2,0,0)                                                      
\mathrm{mod} \, \,\mathbb{Z},                                                           
\end{align*}
for some $i=0,1,2,3,4,5.$ 
We only prove that 
for some $i=0,1,2,3,4,5,$ 
\begin{equation*}
(\alpha_i,\alpha_{i+1},\alpha_{i+2},\alpha_{i+3},\alpha_{i+4},\alpha_{i+5})
\equiv
(1/2,1/2,1/2,1/2,1/2,1/2)
\longrightarrow 
(1/2,0,0,1/2,0,0).
\end{equation*}
The other cases can be proved in the same way.
\par 
Since $\sum_{k=0}^5 \alpha_k =1,$ by $\pi$ and the shift operators, we get 
$$
(\alpha_0,\alpha_{1},\alpha_{2},\alpha_{3},\alpha_{4},\alpha_{5})
=
(1/2,1/2,1/2,-1/2,1/2,-1/2).
$$
By $\pi^{-1}  s_4  s_5  s_3,$ 
we get 
$
(1/2,1/2,1/2,-1/2,1/2,-1/2)
\rightarrow
(1/2,0,0,1/2,0,0).
$
\end{proof}

We then have the following theorem:

\begin{theorem}
\label{thm:a5mainB}
\it{
For a rational solution of Type B of $A_5^{(1)}(\alpha_j)_{0\leq j \leq 5},$ 
by some B\"acklund transformations, 
the parameters and solution 
can be transformed so that 
\begin{align*}
(\alpha_0,\alpha_1,\alpha_2,\alpha_3,\alpha_4,\alpha_5) &=(\alpha_0,-\alpha_0+1/2,\alpha_0,-\alpha_0+1/2,0,0) \,\,{\it and}  \\
(f_0,f_1,f_2,f_3,f_4f_5) &=
(t/2,t/2,t/2,t/2,0,0).
\end{align*}
The orbit of $(\alpha_0,-\alpha_0+1/2,\alpha_0,-\alpha_0+1/2,0,0)$ 
by the B\"acklund transformation group $\tilde{W}(A_5^{(1)})$ 
consists of the parameters 
which satisfy one of the following conditions:
\newline
(1) \quad for some $i=0,1,2,3,4,5,$
\begin{align*}
&-\alpha_i+\alpha_{i+2}-\alpha_{i+4}, 
-\alpha_{i+1}+\alpha_{i+3}+\alpha_{i+5}, 
2\alpha_{i+4}, -2\alpha_{i+5} \in \mathbb{Z},  \\
&(-\alpha_i+\alpha_{i+2}-\alpha_{i+4})
+
(-\alpha_{i+1}+\alpha_{i+3}+\alpha_{i+5}) 
\in 2\mathbb{Z};
\end{align*}  
(2) \quad for some $i=0,1,2,3,4,5,$
\begin{align*}
& \alpha_{i+1} \neq 0, 
-\alpha_i+\alpha_{i+2}-\alpha_{i+4},
\alpha_{i+1}+\alpha_{i+3}+\alpha_{i+5}, 
2\alpha_{i+4}, -2\alpha_{i+5}, 
\in \mathbb{Z},  \\
& (-\alpha_i+\alpha_{i+2}-\alpha_{i+4})
+
(\alpha_{i+1}+\alpha_{i+3}+\alpha_{i+5})
\in 2\mathbb{Z};
\end{align*}
(3) \quad for some $i=0,1,2,3,4,5,$
\begin{align*}
&
\alpha_{i+2} \neq 0, 
-\alpha_{i}-\alpha_{i+2}-\alpha_{i+4},
-\alpha_{i+1}+\alpha_{i+3}+\alpha_{i+5}, 
2\alpha_{i+4}, -2\alpha_{i+5}, 
\in \mathbb{Z},  \\
& 
(-\alpha_{i}-\alpha_{i+2}-\alpha_{i+4})
+
(-\alpha_{i+1}+\alpha_{i+3}+\alpha_{i+5})
\in 2\mathbb{Z}; 
\end{align*}
(4) \quad for some $i=0,1,2,3,4,5,$
\begin{align*}
& 
\alpha_{i+3} \neq 0, 
-\alpha_{i}+\alpha_{i+2}-\alpha_{i+4}, 
-\alpha_{i+1}-\alpha_{i+3}+\alpha_{i+5},
2\alpha_{i+3}+2\alpha_{i+4}, 
-2\alpha_{i+5}, 
\in \mathbb{Z},  \\
& 
(-\alpha_{i}+\alpha_{i+2}-\alpha_{i+4})
+
(-\alpha_{i+1}-\alpha_{i+3}+\alpha_{i+5})
\in 2\mathbb{Z}; 
\end{align*}
(5) \quad for some $i=0,1,2,3,4,5,$
\begin{align*}
&
\alpha_{i+4} \neq 0, 
\alpha_{i+1}-\alpha_{i+3}-\alpha_{i+5},
-\alpha_i+\alpha_{i+2}-\alpha_{i+4},
2\alpha_{i+3}+2\alpha_{i+4},
-2\alpha_{i+4} 
\in \mathbb{Z},   \\
& 
(
\alpha_{i+1}-\alpha_{i+3}-2\alpha_{i+4}-\alpha_{i+5}
)
+
(
-\alpha_i+\alpha_{i+2}-\alpha_{i+4}
)
\in 2\mathbb{Z};
\end{align*}
(6) \quad for some $i=0,1,2,3,4,5,$
\begin{align*}  
&
\alpha_{i+5} \neq 0, 
-\alpha_{i+1}+\alpha_{i+3}+\alpha_{i+5},
\alpha_i-\alpha_{i+2}+\alpha_{i+4}+2\alpha_{i+5}, 
-2\alpha_{i+5}, 
-2\alpha_{i}-2\alpha_{i+5}  
\in \mathbb{Z},   \\
& 
(
-\alpha_{i+1}+\alpha_{i+3}+\alpha_{i+5}
)
+
(
\alpha_{i}-\alpha_{i+2}+\alpha_{i+4}+2\alpha_{i+5}
)
\in 2\mathbb{Z};  
\end{align*}
(7) \quad for some $i=0,1,2,3,4,5,$
\begin{align*}
& 
\alpha_i \neq 0, 
\alpha_{i}+\alpha_{i+2}-\alpha_{i+4}, 
-\alpha_{i+1}+\alpha_{i+3}+\alpha_{i+5}, 
2\alpha_{i+4},
-2\alpha_{i}-2\alpha_{i+5} \in \mathbb{Z},   \\
& 
(
\alpha_i+\alpha_{i+2}-\alpha_{i+4}
)
+
(
-\alpha_{i+1}+\alpha_{i+3}+\alpha_{i+5}
)
\in 2\mathbb{Z};
\end{align*}
(8) \quad 
for some $i=0,1,2,3,4,5,$ 
\begin{align*}
&
\alpha_{i+1}, \alpha_{i+4} \neq 0, 
-\alpha_{i+1}-\alpha_{i+3}-2\alpha_{i+4}-\alpha_{i+5},
-\alpha_i+\alpha_{i+2}-\alpha_{i+4}, 
2\alpha_{i+3}+2\alpha_{i+4},
2\alpha_{i+4} 
\in \mathbb{Z},  \\
& 
(
-\alpha_{i+1}-\alpha_{i+3}-2\alpha_{i+4}-\alpha_{i+5}
)
+
(
-\alpha_i+\alpha_{i+2}-\alpha_{i+4}
)
\in 2\mathbb{Z}; 
\end{align*}
(9) \quad 
for some $i=0,1,2,3,4,5,$
\begin{align*}
&
\alpha_{i+2},\alpha_{i+5} \neq 0, 
-\alpha_i-\alpha_{i+2}-\alpha_{i+4}-2\alpha_{i+5},
-\alpha_{i+1}+\alpha_{i+3}-\alpha_{i+5},
-2\alpha_{i+5}, 
-2\alpha_i-2\alpha_{i+5}, 
\in \mathbb{Z},    \\
& 
(
-\alpha_i-\alpha_{i+2}-\alpha_{i+4}-2\alpha_{i+5}
)
+
(
-\alpha_{i+1}+\alpha_{i+3}-\alpha_{i+5}
)
\in 2 \mathbb{Z}; 
\end{align*}
(10) 
\quad 
for some $i=0,1,2,3,4,5,$ 
\begin{align*}
&
\alpha_{i+3},\alpha_{i} \neq 0,
\alpha_i+\alpha_{i+2}-\alpha_{i+4},
-\alpha_{i+1}-\alpha_{i+3}+\alpha_{i+5}, 
2\alpha_{i+3}+2\alpha_{i+4}, 
-2\alpha_{i}-2\alpha_{i+5} 
\in \mathbb{Z},   \\
& 
(
\alpha_i+\alpha_{i+2}-\alpha_{i+4}
)
+
(
-\alpha_{i+1}-\alpha_{i+3}+\alpha_{i+5}
)
\in 2\mathbb{Z};  
\end{align*}
(11) 
\quad 
for some $i=0,1,2,3,4,5,$
\begin{align*}
(
\alpha_{i},
\alpha_{i+1},
\alpha_{i+2},
\alpha_{i+3},
\alpha_{i+4},
\alpha_{i+5}
)
\equiv&
(1/2,1/2,1/2,1/2,0,0),  \,
(1/2,1/2,0,1/2,1/2,0),  \\
&(1/2,0,1/2,0,0,0) \,\mathrm{mod} \mathbb{Z},
\end{align*}
where (1), (2),..., (11) in this theorem correspond to 
(1), (2),..., (11) in Proposition \ref{prop:necB}, respectively.
\newline

}
\end{theorem}

{\bf Remark}
\newline
The rational solution 
$(t/2,t/2,t/2,t/2,0,0) $ 
corresponds to the rational solution of the fifth Painlev\'e equation.

\begin{proof} 
Suppose that for $A_5^{(1)}(\alpha_j)_{0\leq j \leq 5},$ there exists a rational solution of Type B. 
By Proposition \ref{prop:necB}, we then obtain eleven conditions.  
Furthermore, it follows from Proposition \ref{prop:standardB} that 
the parameters 
can be transformed into 
$(\alpha_0,-\alpha_0+1/2,\alpha_0,-\alpha_0+1/2,0,0),$ 
$(1/2, 0,1/2, \alpha_0, 0, -\alpha_0)$ 
or 
$(\alpha_0,0,-\alpha_0+1,0,0,0).$ 
Especially, 
the parameters satisfy one of the conditions in this theorem 
if and only if 
the parameters can be transformed into 
$(\alpha_0,-\alpha_0+1/2,\alpha_0,-\alpha_0+1/2,0,0).$
\par
Proposition \ref{prop:bpara1} 
shows that 
for $A_5^{(1)}(\alpha_0,-\alpha_0+1/2,\alpha_0,-\alpha_0+1/2,0,0),$ 
the parameters and solutions can be transformed so that 
$(\alpha_0,\alpha_1,\alpha_2,\alpha_3,\alpha_4,\alpha_5)= 
(\alpha_0,-\alpha_0+1/2,\alpha_0,-\alpha_0+1/2,0,0)$ 
and 
$(f_0,f_1,f_2,f_3,f_4,f_5)=(t/2,t/2,t/2,t/2,0,0)$. 
\par
If the parameters 
are transformed into 
$(1/2, 0, 1/2, \alpha_0, 0, -\alpha_0),$ 
it follows from Proposition \ref{prop:bpara2} 
that $2\alpha_j \in \mathbb{Z} \,\,(0\leq j \leq 5)$ and 
$(\alpha_j)_{0\leq j \leq 5}$ 
are transformed into $(1/2,0,1/2,0,0,0).$ 
\par
If the parameters 
are transformed into 
$(\alpha_0,0,-\alpha_0+1,0,0,0),$ 
it follows from Proposition \ref{prop:bpara3} 
that $2\alpha_j \in \mathbb{Z} \,\,(0\leq j \leq 5)$ and 
$(\alpha_{j})_{0\leq j \leq 5}$ 
are transformed into $(1/2,0,1/2,0,0,0).$

\end{proof}

\subsection{Complete classification of rational solutions of Type C}
In this subsection, we 
classify 
the rational solutions of Type C 
of $A_5^{(1)}(\alpha_j)_{0\leq j \leq 5}.$ 
For this purpose, 
we have

\begin{lemma}
\label{lem:cpointgeneral}
\it{
Suppose that 
for some $i=0,1,2,3,4,5,$ 
\begin{align*}
(
\alpha_i,
\alpha_{i+1},
\alpha_{i+2},
\alpha_{i+3},
\alpha_{i+4},
\alpha_{i+5}
)
&\equiv
\frac{p}{3}
(1,0,1,0,1,0)
+
\frac{q}{3}
(1,0,-1,-1,0,1), \,
{\it or} \\ 
&\equiv 
\frac{r}{3}
(0,1,1,1,0,0)
+
\frac{s}{3}
(1,1,0,0,0,1) 
\, \mathrm{mod} \, \mathbb{Z}, \,\,(p,q,r,s=0,\pm1).
\end{align*}
By some B\"acklund transformations, 
the parameters 
$(\alpha_j)_{0\leq j \leq 5}$ 
can then be transformed into 
\begin{equation*}
(\alpha_0,\alpha_1,\alpha_2,\alpha_3,\alpha_4,\alpha_5)
=
(1,0,0,0,0,0), \,
(1/3,1/3,1/3,0,0,0), \,
(1/3,0,1/3,0,1/3,0).
\end{equation*}
The parameters $(\alpha_i)_{0\leq i \leq 5}$ can be transformed into 
$(1/3,0,1/3,0,1/3,0)$ 
if and only if 
\begin{align*}
(
\alpha_i,
\alpha_{i+1},
\alpha_{i+2},
\alpha_{i+3},
\alpha_{i+4},
\alpha_{i+5}
)
\equiv 
&
\frac{\pm 1}{3}
(1,-1,1,1,0,1), \,\,
\frac{\pm 1}{3}
(1,0,-1,-1,0,1), \\
&
\frac{\pm 1}{3}
(1,0,1,0,1,0) \,\,
\mathrm{mod} \, \mathbb{Z},
\end{align*}
for some $i=0,1,2,3,4,5.$ 
}
\end{lemma}

\begin{proof}
We deal with the following two cases: 
$
\mathrm{(1)}
\quad 
(p,q)=(1,1),
\quad 
\mathrm{(2)}
\quad
(r,s)=(1,1). 
$
The other cases can be proved in the same way. 
\newline
(1) 
\quad
By $\pi,$ 
we assume that 
\begin{align*}
(
\alpha_0,
\alpha_{1},
\alpha_{2},
\alpha_{3},
\alpha_{4},
\alpha_{5}
)
&\equiv
\frac{1}{3}
(1,0,1,0,1,0)
+
\frac{1}{3}
(1,0,-1,-1,0,1) &\mathrm{mod} \, \mathbb{Z} \\
&\equiv 
\left(\frac23,0,0,-\frac13,\frac13,\frac13 \right) 
&\mathrm{mod} \, \mathbb{Z}.
\end{align*}
By some shift operators $T_j \,\,(0\leq j \leq 5),$ 
we have 
$(\alpha_i)_{0\leq i \leq 5}
\longrightarrow
(2/3,0,0,-1/3,1/3,1/3).
$
By 
$\pi  s_1  s_2  s_3,$ 
we get
$
(2/3,0,0,-1/3,1/3,1/3)
\longrightarrow
(1/3,1/3,1/3,0,0,0).
$
\newline
(2) 
\quad 
By $\pi,$ 
we assume that 
\begin{align*}
(
\alpha_0,
\alpha_{1},
\alpha_{2},
\alpha_{3},
\alpha_{4},
\alpha_{5}
)
&\equiv
\frac{1}{3}
(0,1,1,1,0,0)
+
\frac{1}{3}
(1,1,0,0,0,1)
\,\, \mathrm{mod} \, \mathbb{Z}  \\
&\equiv 
\left(\frac13,\frac23,\frac13,\frac13,0,\frac13 \right).
\end{align*}
By some shift operators $T_j \,\,(0\leq j \leq 5),$ 
we have 
$(\alpha_i)_{0\leq i \leq 5}
\longrightarrow
(2/3,0,0,-1/3,1/3,1/3).
$
By $\pi  s_1,$ 
we have 
$
(2/3,0,0,-1/3,1/3,1/3)
\longrightarrow
(1/3,0,1/3,0,1/3,0).
$
\end{proof}

\par
In order to state Theorem \ref{thm:a5mainC}, 
we define 
\begin{align*}
& 
\hat{x}_k:=\hat{\alpha}_{k+2}-\hat{\alpha}_{k+4}, \,
\hat{y}_k:=\hat{\alpha}_{k+3}-\hat{\alpha}_{k+5}, \,
\hat{z}_k:=\hat{\alpha}_k-\hat{\alpha}_{k+4}, \,
\hat{\omega}_k:=\hat{\alpha}_{k+1}-\hat{\alpha}_{k+5}, \, \\
& 
\hat{\chi}_k:=\hat{x}_k +\hat{y}_k +\hat{z}_k +\hat{\omega}_k \,\,
(0\leq k \leq 5),
\end{align*}
where $\hat{\alpha}_k \,\,(0\leq k \leq 5)$ 
are defined in Theorem \ref{thm:a5mainC}.

\begin{theorem}
\label{thm:a5mainC}
\it{
For a rational solution of Type C of $A_5^{(1)}(\alpha_j)_{0\leq j \leq 5},$ 
by some B\"acklund transformations, 
the parameters and solution 
can be transformed so that 
\begin{align*}
(\alpha_0, \alpha_1,\alpha_2,\alpha_3,\alpha_4,\alpha_5) &=(\alpha_4, -\alpha_4+1/3,\alpha_4,-\alpha_4+1/3,\alpha_4, -\alpha_4+1/3)\,\, {\it and} \\
(f_0, f_1,f_2,f_3,f_4,f_5) &=(
t/3,t/3,t/3,t/3,t/3,t/3
).
\end{align*}
The orbit of $(\alpha_4, -\alpha_4+1/3,\alpha_4,-\alpha_4+1/3,\alpha_4, -\alpha_4+1/3)$ 
by the B\"acklund transformation group $\tilde{W}(A_5^{(1)})$ 
consists of the parameters which satisfy 
one of the following conditions:
for some $k=0,1,2,3,4,5,$ 
\begin{align*}
&\mathrm{(1)} \quad 
\hat{x}_{k},\hat{y}_{k},\hat{z}_{k},\hat{\omega}_{k} \in \mathbb{Z}, \hat{\chi}_k \in 3 \mathbb{Z},     \\
&\mathrm{(2)} \quad 
(\hat{x}_{k},\hat{y}_{k},\hat{z}_{k},\hat{\omega}_{k}) 
\equiv \frac13 (-1,1,1,-1) \, \mathrm{mod} \, \mathbb{Z},  \, \hat{\chi}_k \in 3\mathbb{Z}, \\
&\mathrm{(3)} \quad 
\hat{x}_{k},\hat{y}_{k},\hat{z}_{k},\hat{\omega}_{k} \in \mathbb{Z}, \, \hat{\chi}_k +1\in 3\mathbb{Z},  \\ 
&\mathrm{(4)} \quad 
(\hat{x}_{k},\hat{y}_{k},\hat{z}_{k},\hat{\omega}_{k}) 
\equiv -\frac13 (-1,1,1,-1) \, \mathrm{mod} \, \mathbb{Z},  \, \hat{\chi}_k +1\in 3\mathbb{Z},  \\
&\mathrm{(5)} \quad 
(
\alpha_{k},\alpha_{k+1},\alpha_{k+2},\alpha_{k+3},\alpha_{k+4},\alpha_{k+5}
)
\equiv
\pm \frac13
(1,-1,1,1,0,1), \,\,
\pm \frac13
(1,0,-1,-1,0,1),  \\  
&\hspace{68mm}  
\equiv                                                          
\pm \frac13  (1,0,1,0,1,0) \,\,  \mathrm{mod} \, \mathbb{Z}.                                 
\end{align*}
where $\hat{\alpha}_k \,\,(k=0,1,2,3,4,5)$ are defined by one of the following equations:
\begin{align*}
&\mathrm{(i)}   \quad \hat{\alpha}_k=\alpha_k; \\
&\mathrm{(ii)} \quad 
\hat{\alpha}_k=\alpha_k+\alpha_{k+1}, 
\hat{\alpha}_{k+1}=-\alpha_{k+1}, 
\hat{\alpha}_{k+2}=\alpha_{k+2}+\alpha_{k+1},
\hat{\alpha}_{k+3}=\alpha_{k+3},  \\
&\hspace{7mm}
\hat{\alpha}_{k+4}=\alpha_{k+4},
\hat{\alpha}_{k+5}=\alpha_{k+5}, 
\mathrm{and} \, \alpha_{k+1} \neq 0;  \\
&\mathrm{(iii)} \quad 
\hat{\alpha}_k=\alpha_{k+1}+\alpha_{k},
\hat{\alpha}_{k+1}=-\alpha_{k+1},
\hat{\alpha}_{k+2}=\alpha_{k+2}+\alpha_{k+1},
\hat{\alpha}_{k+3}=\alpha_{k+3}+\alpha_{k+4},  \\
&\hspace{8mm}
\hat{\alpha}_{k+4}=-\alpha_{k+4},
\hat{\alpha}_{k+5}=\alpha_{k+5}+\alpha_{k+4}, 
\mathrm{and} \, \alpha_{k+1}, \alpha_{k+4} \neq 0.  
\end{align*}

}
\end{theorem}

\begin{proof}
Suppose that for $A_5^{(1)}(\alpha_j)_{0\leq j \leq 5},$ there exists a rational solution of Type C. 
It then follows from 
Proposition \ref{prop:standardC} 
that the parameters can be transformed into 
$(
\alpha_4,
-\alpha_4+1/3,
\alpha_4,
-\alpha_4+1/3,
\alpha_4,
-\alpha_4+1/3
)$ 
or 
$
(-\alpha_4+1/3,1/3,1/3,\alpha_4,0,0),
$
or 
$
(\alpha_4,0,0,1-\alpha_4,0,0),
$ 
or 
$(\alpha_4,1/3,1/3,-\alpha_4+1/3,0,0).$ 
Especially, 
the parameters can be transformed into 
$(
\alpha_4,
-\alpha_4+1/3,
\alpha_4,
-\alpha_4+1/3,
\alpha_4,
-\alpha_4+1/3
),$ 
if and only if they satisfy one of the conditions in this theorem. 
Proposition \ref{prop:cpara1} shows that 
for $A_5^{(1)}(
\alpha_4,
-\alpha_4+1/3,
\alpha_4,
-\alpha_4+1/3,
\alpha_4,
-\alpha_4+1/3
),
$ 
there exists 
a rational solution of Type C 
and 
$
(f_0,f_1,f_2,f_3,f_4,f_5)=
(
t/3,t/3,t/3,t/3,t/3,t/3
) 
$
and it is unique. 
Propositions \ref{prop:cpara2}, \ref{prop:cpara3} and \ref{prop:cpara4} 
shows that 
for 
$
A_5^{(1)}(-\alpha_4+1/3,1/3,1/3,\alpha_4,0,0),
$
or 
$
A_5^{(1)}(\alpha_4,0,0,1-\alpha_4,0,0),
$ 
or 
$A_5^{(1)}(\alpha_4,1/3,1/3,-\alpha_4+1/3,0,0)$ 
there exists no rational solution of Type C.

\end{proof}



\end{document}